\documentclass[a4paper,11pt]{article}
\author{Jelena~Jovanovi\'c\\ \emph{University of Belgrade, Faculty of Mathematics }}
\title{Omitting unary and affine types}
\usepackage{amsmath}
\usepackage{amssymb}
\usepackage{amsthm}
\usepackage{verbatim}
\date{}
\theoremstyle{plain}
\newtheorem{teorema}{Theorem}[section]
\newtheorem{definicija}[teorema]{Definition}
\newtheorem{tvrdjenje}[teorema]{Proposition}
\newtheorem{fakt}[teorema]{Fact}

\begin{document}
\maketitle
\begin{abstract}
In this paper we examine the possibility of describing omitting types 1 and 2 by two at most ternary terms and any number of linear identities. All possible cases of systems of linear identities on two at most ternary terms are being analyzed, and it is shown that only three of these systems might describe omitting types 1 and 2. However, we do not resolve whether either of them actually describes omitting mentioned types, but only prove that each of them implies this property, so this question is left for further examination.
\end{abstract}
\section{Introduction}
The theory of finite algebras that we base this paper on is called \emph{tame congruence theory}, and is developed by David Hobby and Ralph McKenzie during 1980s. The main idea of the theory  is that we can obtain much information about an algebra and the variety it generates from so--called \emph{local behavior} of the algebra. We shall present only a brief overview on this (taken almost entirely from \cite{nm}), and for more information the reader is referred to \cite{hm}.
\begin{definicija}
Let $\mathbf{A}$ be a finite algebra and $\alpha$ a minimal congruence of $\mathbf{A}$ (i.e. $0_{\mathbf{A}} <  \alpha$ and if $\beta$ is a congruence of $\mathbf{A}$ with $0_{\mathbf{A}} <  \beta \le \alpha $ then $\beta = \alpha$.)
\begin{itemize}
\item An $\alpha$--minimal set of $\mathbf{A}$ is a subset $U$ of $\mathbf{A}$ that satisfies following two conditions:
\begin{itemize}
\item[-]$U = p(\mathbf{A})$ for some unary polynomial $p(x)$ of $\mathbf{A}$ that is not constant on at least one $\alpha$--class 
\item[-] with respect to containment, $U$ is minimal having this property.
\end{itemize}
\item An $\alpha$--neighbourhood (or $\alpha$--trace) of $\mathbf{A}$ is a subset $N$ of $\mathbf{A}$ such that:
\begin{itemize}
\item[-] $N = U \cap (a /_{\alpha})$ for some $\alpha$--minimal set $U$ and $\alpha$--class $a /_{\alpha}$
\item[-] $\vert N \vert > 1$. 
\end{itemize}
\end{itemize} 
\end{definicija}
We can easily see that a given $\alpha$--minimal set $U$ must contain at least one, and possibly more, $\alpha$--neighbourhoods. The union of all $\alpha$--neighbourhoods in $U$ is called the body of $U$, and the remaining elements of $U$ form the tail of $U$. What is important here is that algebra $\mathbf{A}$ induces uniform structures on all its $\alpha$--neighbourhoods, meaning they (the structures induced) all belong to the same of five possible types. Let us now define an induced structure.
\begin{definicija}
Let $\mathbf{A}$ be an algebra and $U \subseteq \mathbf{A}$. The algebra induced by $\mathbf{A}$ on $U$ is the algebra with universe $U$ whose basic operations consist of the restriction to $U$ of all polynomials of $\mathbf{A}$  under which $U$ is closed. We denote this induced algebra by $\mathbf{A} \vert_{U}$. 
\end{definicija}
\begin{teorema}
Let $\mathbf{A}$ be a finite algebra and $\alpha$ a minimal congruence of $\mathbf{A}$.
\begin{itemize}
\item If $U$ and $V$ are $\alpha$--minimal sets then $\mathbf{A} \vert_{U}$ and $\mathbf{A} \vert_{V}$ are isomorphic and in fact there is a polynomial $p(x)$ that maps $U$ bijectively onto $V$.
\item If $N$ and $M$ are $\alpha$--neighbourhoods then $\mathbf{A} \vert_{N}$ and $\mathbf{A} \vert_{M}$ are isomorphic via the restriction of some polynomial of $\mathbf{A}$.
\item If $N$ is $\alpha$--neighbourhood then $\mathbf{A} \vert_{N}$ is polynomially equivalent to one of:
\begin{enumerate}
\item A unary algebra whose basic operations are all permutations (unary type);
\item A one--dimensional vector space over some finite field (affine type);
\item A $2$--element boolean algebra (boolean type);
\item A $2$--element lattice (lattice type);
\item A $2$--element semilattice (semilattice type);
\end{enumerate}
\end{itemize} 
\end{teorema}
\begin{proof}
The theorem in this form is given in \cite{nm}, and the proof can be found in \cite{hm}.
\end{proof}
The previous theorem allows us to assign a type to each minimal congruence $\alpha$ of an algebra according to the behaviour of the $\alpha$--neighbourhoods (for example, a minimal congruence whose $\alpha$--neighbourhoods are polynomially equivalent to a vector space is said to have affine type or type $2$).

Taking this idea one step further, given a pair of congruences $(\alpha, \beta)$ of $\mathbf{A}$ with $\beta$ covering $\alpha$ (i.e. $\alpha < \beta$ and there are no congruences of $\mathbf{A}$ strictly between the two), one can form the quotient algebra $\mathbf{A}/_{\alpha}$, and then consider the congruence $\beta /_{\alpha} = \{(a/_{\alpha}, b/_{\alpha}):(a,b) \in \beta \}$. Since $\beta$ covers $\alpha$ in the congruence lattice of $\mathbf{A}$,  $\beta /_{\alpha}$ is a minimal congruence of $\mathbf{A}/_{\alpha}$, so it can be assigned one of the five types. In this way we can assign to each covering pair of congruences of $\mathbf{A}$ a type (unary, affine, boolean, lattice, semilattice, or 1, 2, 3, 4, 5 respectively). Therefore, going through all covering pairs of congruences of this algebra we obtain a set of types, so--called typeset of $\mathbf{A}$, denoted by $typ \{\mathbf{A}\}$. Also, for $\mathcal{K}$  a class of algebras, the typeset of $\mathcal{K}$ is defined to be the union of all the typesets of its finite members, denoted by $typ \{\mathcal{K} \}$.

A finite algebra or a class of algebras is said to omit a certain type if that type does not appear in its typeset. For locally finite varieties omitting certain types can be characterized by Maltsev conditions, i.e. by the existence of certain terms that satisfy certain linear identities, and there are quite a few results on this so far. We shall present two of them concerning omitting types 1 and 2.
\begin{definicija}
An n--ary term $t$, for $n > 1$, is a near--unanimity term for an algebra $\mathbf{A}$ if the identities $t(x,x,\dots ,x,y) \approx t(x,x,\dots ,y,x) \approx  \dots  \approx t(x,y,\dots ,x,x) \approx t(y,x,\dots ,x,x) \approx x $ hold in $\mathbf{A}$.
\end{definicija}
\begin{definicija}
An n--ary term $t$, for $n > 1$, is a weak near--unanimity term for an algebra $\mathbf{A}$ if it is idempotent and the identities $t(x,x,\dots ,x,y) \approx t(x,x,\dots ,y,x) \approx  \dots  \approx t(x,y,\dots ,x,x) \approx t(y,x,\dots ,x,x)$ hold in $\mathbf{A}$.
\end{definicija}
\begin{teorema}

A locally finite variety $\mathcal{V}$ omits the unary and affine types (i.e. types 1 and 2) if and only if there is some $N > 0$ such that for all $k > N$, $\mathcal{V}$ has a weak near--unanimity term of arity $k$.\label{th}

\end{teorema}
\begin{proof} The proof can be found in \cite{hhm}.
\end{proof}
\begin{teorema}
A locally finite variety $\mathcal{V}$ omits the unary and affine types if and only if it has 3--ary and 4--ary weak near--unanimity terms, $v$ and $w$ respectively, that satisfy the identity $v(y,x,x) \approx w(y,x,x,x)$.
\end{teorema}
\begin{proof}
The proof can be found in \cite{hhhm}.
\end{proof}

Therefore, omitting types 1 and 2 for a locally finite variety can be described by linear identities on 3--ary and a 4--ary term, both idempotent. In this paper we examine whether the same can be done by two at most 3--ary idempotent terms. It is sufficient to focus our attention to idempotent terms only, for the reasons that are explained in detail in \cite{hhm}. We shall use two particular algebras for our examination, both of them omitting types 1 and 2. 

\section{ Examples of algebras omitting unary and affine types } \label{section two}
\noindent $\mathbf{Example \  1}$\\

\noindent Let $\mathbf{B}=\langle \ \{\ \!0 \ ,\ 1 \} \ ,\ \land  \  \rangle$ \label{example 1} be the semilattice with two elements (i.e. $\land$ stands for a commutative, associative and idempotent binary operation). This algebra omits types 1 and 2, for it is easy to see that $typ \{\mathbf{B}\}=\{5\}$ ($1_{\mathbf{B}}$ is the minimal congruence, so $B$ is $1_{\mathbf{B}}$--neighbourhood and induced structure is, of course, polynomially equivalent to algebra $\mathbf{B}$).\\

\noindent $\mathbf{Example \  2}$\\

\noindent Let $\mathbf{A}$ be a finite algebra with at least two elements and a single idempotent basic operation $ f(x_1,x_2,x_3)$, which is a ternary near--unanimity term (i.e. a majority term): \begin{equation*} f(x,x,y) \approx f(x,y,x) \approx f(y,x,x) \approx x \end{equation*} In case no arguments are equal, we can define $f$ like this:
\begin{equation*} f(a,b,c) = a,\  \text{for all}\  a,b,c \in \mathbf{A}\   \text{and}\  a\ne b,\  b \ne c,\  c\ne a
\end{equation*}
We shall prove now that algebra $\mathbf{A}$ omits types 1 and 2:

Let $g(x,y,w,z)$ be a $4$--ary term of this algebra defined by: $g(x,y,w,z) \approx f(x,y,f(x,w,z))$, $f$ being the basic operation. It is easy to check that $g$ is a weak near--unanimity term, and the identity $g(y,x,x,x) \approx  f(y,x,x)$ holds, so algebra $\mathbf{A}$ omits unary and affine types according to the theorem 1.6 from above. 
\vspace{0.3 cm}

This algebra has some interesting properties:
\begin{enumerate}
\item every binary term--operation $t$ of $\mathbf{A}$ must satisfy one of these two identities:\label{binary term} \\$t(x,y)\approx x$ \\ $t(x,y)\approx y$;\\
\noindent In other words the only binary term--operations on $\mathbf{A}$  are projections $\pi_1 , \pi_2$;
\begin{proof}
We shall prove the statement by induction on the complexity of the term $t$: 
\begin{itemize}
\item[-] if $t(x,y)$ is a projection the statement holds
\item[-] if $t(x,y)\approx f(t_1(x,y), t_2(x,y), t_3(x,y))$, where $t_1,t_2,t_3$ are less complex binary terms, then these three are projections by the induction hypothesis, therefore at least two of them are equal, so $t$ must be a projection too.
\end{itemize}
\end{proof}
\item every ternary term--operation $p$  of $\mathbf{A}$ satisfies exactly one of the following:\label{ternary term}\\$p(x,y,z)\approx x$\\ $p(x,y,z)\approx y$\\ $p(x,y,z)\approx z$\\ $p(x,x,y)\approx p(x,y,x)\approx p(y,x,x)\approx x$;\\

\vspace{0.1 cm}
\noindent This means $p$ is either one of the projections $\pi_1,\pi_2,\pi_3$  or a majority term--operation, that is there are no other ternary term--operations except for these four kinds.
\begin{proof} 
We prove the second statement also by induction on the complexity of the term--function $p$:
\begin{itemize}
\item if $p(x,y,z)$ is a projection or the basic operation $f(x,y,z)$, the statement holds 
\item if $p(x,y,z)\approx f(p_1(x,y,z), p_2(x,y,z), p_3(x,y,z))$,where $p_1,p_2,p_3$ are less complex ternary terms, then each of these three is either a projection or a  majority term by the induction hypothesis, so we have the following cases:
\begin{itemize} 
\item[-] if at least two of  $p_1,p_2,p_3$ are majority terms, then $p$ is also a majority term;
\item[-] if exactly one of  $p_1,p_2,p_3$ is a majority term and remaining two are the same projection $\pi_j \ \text{for some}\  j \in \{1,2,3\}$ then $p$ is also a projection $\pi_j$;
\item[-] if exactly one of  $p_1,p_2,p_3$ is a majority term and remaining two are projections $\pi_i,\pi_j \ \text{for some}\ \\ i,j \in \{1,2,3,\}\ \  \text{and}\ \  i \ne j$\  then $p$ is a majority term;
\item[-] if the terms $p_1,p_2,p_3$ are projections $\pi_1,\pi_2,\pi_3$ (in whichever order) then $p(x,y,z)$ is $f(x,y,z)$ up to the permutation of variables, which is still a majority term;
\item[-] if the terms $p_1,p_2,p_3$ are projections $\pi_i,\pi_i,\pi_j$ (again in whichever order) $ \text{for some}\  i,j \in \{1,2,3\} \ \text{and}\  i \ne j$  then $p$ is a projection $\pi_i$;
\end{itemize}
\end{itemize}
\end{proof}
\end{enumerate}
  
\vspace{0.2 cm}
Now that we have listed the examples needed, let us notice that any system of (linear) identities possibly describing omitting types 1 and 2 (including any number of terms) must hold in algebras $\mathbf{B}$ and $\mathbf{A}$ from examples 1 and 2 respectively. We shall make use of this fact in the rest of the paper. 

In section 3 we discuss systems of linear identities on a single binary term, two binary terms, a single ternary term and a binary and a ternary term. We prove that none of these systems describes omitting types 1 and 2 ( in fact none of them even implies omitting these two types). 

In section 4 systems on two ternary terms are being discussed, and we prove that there are only three of them that could possibly describe omitting types 1 and 2 ( all three imply omitting these two types). As mentioned in the abstract, we do not resolve whether any of them actually describes this property.  

\section{Systems of linear identities on a single binary term, two binary terms, a single ternary term and a binary and a ternary term}\label{section2}

As we have already mentioned, in this section we discuss systems of linear identities on a single binary term, two binary terms, a single ternary term and a binary and a ternary term, each of these cases being analyzed in a separate subsection. We shall prove here that none of these systems describes omitting types 1 and 2 (in fact none of them even implies omitting these two types).
\subsection {A single binary idempotent term} \label{singlebinary}
\vspace{0.1cm}

Let $t(x,y)$ be an idempotent binary term. If a system of linear identities on $t(x,y)$ (a single identity or more) describes omitting types 1 and 2, it must hold in algebra $\mathbf{A}$ from example 2, which means $t(x,y)$ has to be a projection map in $\mathbf{A}$. So, the system considered must allow $t(x,y)$ to be a projection map and it must not yield a trivial variety (algebra), but such a system holds in every algebra, therefore does not describe omitting types 1 and 2. 

\vspace{0.1cm}

\noindent We can conclude that omitting types 1 and 2 cannot be described by a single binary idempotent term (using any number of identities).

\subsection{Two binary idempotent terms}\label{twobinary}

\vspace{0.1cm}

Let $t(x,y)$ and $s(x,y)$ be idempotent binary terms. If a system of identities on $t(x,y)$ and $s(x,y)$ describes omitting types 1 and 2, it must hold in algebra $\mathbf{A}$ from example 2, which means both $t$ and $s$ have to be projection maps. This means the identities of the system considered must allow both $t$ and $s$ to be projection maps, but these exist in every algebra, so the system cannot describe omitting types 1 and 2. 

\vspace{0.3 cm}

\noindent From the previous we conclude that omitting types 1 and 2 cannot be described by two binary idempotent terms (using any number of identities).

\subsection{A single ternary idempotent term}\label{singleternary}

In this subsection we prove that omitting types 1 and 2 cannot be described by any number of linear identities on a single ternary idempotent term.
\noindent Previously, let us consider a specific reduct of a module that we use in the proof-- a full idempotent reduct of a module over $\mathbb{Z}_5$ (this is an algebraic structure obtained from a module over $\mathbb{Z}_5$ by taking into consideration only the idempotent term--operations of the module and all such term--operations):

In a full idempotent reduct of a module over $\mathbb{Z}_5$ a ternary term $p(x,y,z)$ must satisfy one of the following identities:
\begin{equation}
\begin{array}{lll}
p(x,y,z)\approx x & p(x,y,z)\approx y & p(x,y,z)\approx z\\
p(x,y,z)\approx 4x + 2y & p(x,y,z)\approx 4x + 2z & p(x,y,z)\approx 4y + 2z\\
p(x,y,z)\approx 2x + 4y & p(x,y,z)\approx 2x + 4z & p(x,y,z)\approx 2x + 4z\\
p(x,y,z)\approx 3x + 3z & p(x,y,z)\approx 3x + 3y & p(x,y,z)\approx 3y + 3z \\
p(x,y,z)\approx x + 2y + 3z & p(x,y,z)\approx x + 3y + 2z & p(x,y,z)\approx 2x + y + 3z\\
p(x,y,z)\approx 2x + 3y + z & p(x,y,z)\approx 3x + 2y + z & p(x,y,z)\approx 3x + y + 2z\\
p(x,y,z)\approx 4x + y + z& p(x,y,z)\approx x + y + 4z & p(x,y,z)\approx x + 4y + z\\
p(x,y,z)\approx 2x + 2y + 2z\\
\end{array}
\end{equation}There are no other ternary terms in this reduct.
\vspace{0.1cm}

\noindent Now we can discuss systems of linear identities on a single ternary idempotent term.

\vspace{0.1cm}

\begin{fakt} Suppose a system of identities on $p(x,y,z)$ describes omitting types 1 and 2, and let us denote it by $\sigma$. Then the system $\sigma$ has to hold in algebras $\mathbf{B}$ and $\mathbf{A}$ from examples 1 and 2 respectively, and it must not hold in any full idempotent reduct of a module over a finite ring (theorem 8 in \cite{nm}, or more detailed in \cite{nnm}).
\end{fakt}

\noindent Based on this fact we can state the following:
\begin{itemize}
\item if identities of the system $\sigma$ allow $p$ to be defined as a projection map in algebra $\mathbf{A}$ (any projection map) then $p$ can be defined as a projection map in any algebra, so the system does not describe omitting types 1 and 2. Therefore the identities of the system must have forms that allow $p$ to be a majority term and only a majority term in $\mathbf{A}$.  
\item if there is an identity of the form $p(x,y,z)\approx p(u,v,w)$ in the system $\sigma$, then $\{x,y,z\} = \{u,v,w\}$, i.e. $(u,v,w)$ is a permutation of $(x,y,z)$, for otherwise $p$ could not be a majority term in $\mathbf{A}$.
\item considering identities with less than three variables on either side: if the left hand side of an identity is some of $p(x,x,y), p(x,y,x), p(y,x,x)$, then on the right there has to be either $x$ alone, or one of the terms $p(x,x,y), p(x,y,x), p(y,x,x),p(x,x,z),p(z,x,x),p(x,z,x)$(again for the same reason, $p$ being necessarily a majority term in $\mathbf{A}$). 
\item if the system $\sigma$ contains only identities having variables $x,y,z$ on both sides and/or identities with $x,x,y$ on both sides (that is $x$ occurs twice, $y$ once on both sides) then the system holds in a full idempotent reduct of a module over $\mathbb{Z}_5$ (and therefore does not describe omitting types 1 and 2), for we can define $p$ to be $2x + 2y + 2z$ in this reduct. This means $\sigma$ has to include an identity with $x,x,y$ on the left and $x,x,z$ on the right (up to a permutation of these variables, of course), or $x,x,y$ on the left and $x$ alone on the right. 
\item since the system $\sigma$ has to hold in algebra $\mathbf{B}$ (example 1), from the previous item we can conclude that $p$ has to be a binary term in $\mathbf{B}$ (it cannot be a projection map for the system does not allow that). Now, if the system allows $p(x,y,z)$ to be defined as a binary term in $\mathbf{B}$, i.e. one of the terms $x \land y$, $y \land z$, $x \land z$, then it also allows $p$ to be one of the terms $3x+3y$, $3y+3z$, $3x+3z$ in a full idempotent reduct of a module over $\mathbb{Z}_5$. This means $\sigma$ does not describe omitting types 1 and 2.
\end{itemize}

\noindent This proves that no system on $p(x, y, z)$ satisfies the necessary conditions for describing omitting types 1 and 2 given above ( i.e. to hold in algebras $\mathbf{B}$ and $\mathbf{A}$ and not to hold in any full idempotent reduct of a module over a finite ring).

\subsection{A binary idempotent term and a ternary idempotent term}\label{bin_ter}

\vspace{0.1cm}

\noindent In this subsection we shall discuss whether a system of any number of linear identities on $t(x,y)$ and $p(x,y,z)$ can describe omitting types 1 and 2 ($t$ and $p$ both being idempotent terms). 

\vspace{0.1 cm}
\noindent Suppose we have a system on $t$ and $p$ describing omitting types 1 and 2, and let us denote it by $\tau$. It has to hold in $\mathbf{A}$, so $t$ has to be a projection map, and $p$ either a projection map or a majority term in this algebra (this is proved in section \ref{section two}). If the system $\tau$ allows both $t$ and $p$ to be projections in $\mathbf{A}$, it holds in any algebra ($t$ and $p$ being the same projection maps as in $\mathbf{A}$), so it cannot describe omitting types 1 and 2. Therefore $\tau$ must hold in $\mathbf{A}$ only for $p$ being a majority term (and $t$ a projection map, of course). Further more, algebra $\mathbf{B}$ has to satisfy the system also, so let us analyze possible cases:
\begin{itemize}
\item both $t$ and $p$ can be defined as projection maps in $\mathbf{B}$ -- this is impossible because if this were the case then both terms could be defined as projection maps in $\mathbf{A}$, and we have already excluded that.
\item $t$ can be defined as a projection map and $p$ as a binary term (i.e. one of the $x \land y$, $y \land z$, $x \land z$) in $\mathbf{B}$ -- then we can define $t$ to be the same projection map, and $p$ to be one of the terms $3x+3y$, $3y+3z$, $3x+3z$ in a full idempotent reduct of a module over $\mathbb{Z}_5$, i.e. $t$ and $p$ exist in this reduct, therefore the system $\tau$ does not describe omitting types 1 and 2.
\item $t$ can be defined as a projection map and $p$ as a ternary term in $\mathbf{B}$ (meaning, of course, that we cannot define $p$ as either a projection map or a binary term in $\mathbf{B}$ )-- in this case the system $\tau$ cannot contain an identity on $t$ and $p$, i.e. all the identities are either only on $t$ or only on $p$. Since the identities on $t$ allow $t$ to be a projection map they can be ignored, so $\tau$ describes omitting types 1 and 2 if and only if remaining identities only on $p$ do the same. This is already proved to be impossible (subsection \ref{singleternary}).
\item $t$ can be defined only as a binary term and $p$ as a projection map in $\mathbf{B}$ -- once again $\tau$ cannot include an identity on $t$ and $p$, i.e. all the identities are either only on $t$ or only on $p$. Since $t$ has to be a projection map in $\mathbf{A}$, the identities on $t$ must allow that, which means $t$ can be defined as the same projection map in $\mathbf{B}$. Therefore this case is impossible. 
\item both $t$ and $p$ can be defined as  binary terms in $\mathbf{B}$ (and of course, none of them as a projection map) -- if this is the case it is easily seen that both can be defined as binary terms in a full idempotent reduct of a module over $\mathbb{Z}_5$ (some of the terms $3x+3y$, $3y+3z$, $3x+3z$), so the system $\tau$ does not describe omitting types 1 and 2.
\item $t$ can be defined as a binary term and $p$ as a ternary term in $\mathbf{B}$ (and no other possibilities, as before) -- then there is no identity only on $t$ in $\tau$, since it would have to be this one $t(x,y) \approx t(y,x)$, and it cannot hold in $\mathbf{A}$. Further more, if $t(x,y)$ is on the left, then on the right we have term $p$ with variables $x$ and $y$ only. This allows us to eliminate term $t$ from all the identities except for one, obtaining an equivalent system. Now we can ignore the identity with $t$ (the only one of the form $t(x,y) \approx p(u,v,w)$, where $\{u,v,w\}=\{x,y\}$) and state that the system $\tau$ describes omitting types 1 and 2 if and only if the remaining identities only on $p$ do the same, which is impossible (subsection \ref{singleternary}).   
\end{itemize}

\vspace{0.1cm}
\noindent By this we have proved that omitting types 1 and 2 cannot be described by a binary and a ternary term, both idempotent (using any number of linear identities).
\section{Two ternary idempotent terms}\label{stt}
\vspace{0.1 cm}
In the following section we shall discuss systems of linear identities on two ternary terms, both idempotent, and we shall prove that only three of these systems could possibly describe omitting types 1 and 2 (all three imply omitting these two types). However, we do not resolve whether any of them actually describes this property.

\vspace{0.1 cm}
  
\noindent Let $p(x,y,z)$ and $q(x,y,z)$ be ternary idempotent terms; as before we shall suppose there is a system of linear identities on $p$ and $q$ describing omitting types 1 and 2, and we shall denote it by $\phi$. 

\noindent Let us notice the important fact: if the system $\phi$ has no identity on $p$ and $q$, i.e. all its identities are either only on $p$ or only on $q$, then we can apply the conclusion that we came to in subsection \ref{singleternary}: if $p$ exists in algebras $\mathbf{B}$ and $\mathbf{A}$ from examples 1 and 2, then $p$ also exists in a full idempotent reduct of a module over $\mathbb{Z}_5$, and the same holds for $q$. Therefore the system $\phi$ needs to have at least one identity on $p$ and $q$.

\noindent Regarding the fact that $\phi$ has to hold in algebra $\mathbf{A}$ (i.e. terms $p$ and $q$ have to exist in this algebra) there are three possible cases:
\begin{enumerate}
\item $p$ and $q$ can both be projection maps in algebra $\mathbf{A}$ -- then $\phi$ holds in any algebra ($p$ and $q$ being the same projection maps as in $\mathbf{A}$), so there is no need to analyze this case any further. 
\item the system $\phi$ allows only one of $p$, $q$ to be a projection map in $\mathbf{A}$, and the other term has to be a majority term in this algebra.
\item the system $\phi$ does not allow either of $p$ and $q$ to be a projection map in $\mathbf{A}$, i.e. both are majority terms in this algebra. 
\end{enumerate}
We shall analyze cases 2 and 3 in the following two subsections -- in the subsection \ref{cases12} we deal with case 2, and in subsection \ref{case3} with case 3.
\subsection{ $\mathbf{p}$ and $\mathbf{q}$ are a projection and a majority term in $\mathbf{A}$ } \label{cases12}

\vspace{0.3 cm}

\noindent Suppose the system $\phi$ holds in $\mathbf{A}$ for a projection map and a majority term (case 2 from above). We can assume with no loss of generality that $p$ is $\pi_1$ and $q$ a majority term, therefore these terms satisfy the following in $\mathbf{A}$: 
\begin{equation}
\begin{array}{r}
x\approx p(x,x,y)\approx p(x,y,y)\approx p(x,y,x)\approx q(x,x,y)\approx q(x,y,x)\approx q(y,x,x)
\label{sys}
\end{array}
\end{equation}
\noindent It is easily seen that $p$ and $q$ satisfying the system do not exist in $\mathbf{B}$ (example 1), so \eqref{sys} cannot be the system describing omitting types 1 and 2. We shall try to obtain the system mentioned by eliminating some identities from \eqref{sys}, so let us eliminate the first one:
\begin{equation}
\begin{array}{r}
p(x,x,y)\approx p(x,y,y)\approx p(x,y,x) \approx q(x,x,y)\approx q(x,y,x)\approx q(y,x,x)
\end{array}\label{system1}
\end{equation}
\noindent We shall prove here that system \eqref{system1} implies omitting types 1 and 2, because terms $p$, $q$ do not exist in any full idempotent reduct of a module over a finite ring:
\begin{proof}
If term $p$ is projection map $\pi_1$ in a reduct, then $q$ would have to be at most a binary term, which is impossible. Term $p$ cannot be a binary term in any reduct either (we can see this directly from the identities), so it has to be a ternary term, i.e. $\alpha x + \beta y + \gamma z$ for $\alpha + \beta + \gamma =1$ and none of $\alpha, \beta, \gamma$ is zero. Now, from the first identity we obtain $\alpha x + \beta x + \gamma y = \alpha x + \beta y + \gamma y$, which yields $\alpha  + \beta = \alpha$, i.e. $\beta = 0$. Therefore $p$ does not exist in any reduct over a finite ring, and this completes the proof.
\end{proof}
\vspace{0.1cm}
\noindent So, we have obtained the system that implies omitting types 1 and 2, but it is not minimal -- namely, the system given below, obtained from the previous by eliminating identity $p(x,y,y)\approx p(x,y,x)$, also implies omitting types 1 and 2 (this is proven the same way like the above case). 
\begin{equation}
\left\{
\begin{array}{r}
p(x,x,y)\approx p(x,y,y)\\
p(x,y,x)\approx q(x,x,y) \approx q(x,y,x) \approx q(y,x,x)
\end{array}\right. \label{system2}
\end{equation}
\noindent Moreover, it can be proved that \eqref{system2} is a minimal system with this property:
\begin{proof}
It'll be sufficient to prove that any set of identities which is a proper subset of \eqref{system2} allows us to define $p$ and $q$ in some reduct of a module over a finite ring.

\vspace{0.4 cm}

\noindent If we eliminate the first identity, we obtain the system:
\begin{equation*}
\begin{array}{r}
p(x,y,x)\approx q(x,x,y) \approx q(x,y,x)\approx q(y,x,x)
\end{array} 
\end{equation*}
Now both terms can be defined as $2x + 2y + 2z$ in a reduct over $\mathbb{Z}_5$. So the first identity ($p(x,x,y)\approx p(x,y,y)$) must stay.

\vspace{0.2 cm}

\noindent If the term $p(x,y,x)$ is completely omitted, we have an identity only on $p$ (the first one), and the rest of them are only on $q$, but in this case $p$ can be defined as $\pi_1$ and $q$ as $2x + 2y + 2z$ in a reduct over $\mathbb{Z}_5$. This means we have to keep one of the following identities (i.e. at least one) $p(x,y,x) \approx q(x,x,y), p(x,y,x) \approx q(x,y,x), p(x,y,x) \approx q(y,x,x)$, and we shall consider it the second. Let us go through the cases now:
\begin{itemize}
\item if we have $p(x,y,x) \approx q(x,x,y)$ as the second identity (the first identity being $p(x,x,y)\approx p(x,y,y)$, as explained), obviously we need to add one more identity from the system \eqref{system2} to these two, for if we do not, $p$ and $q$ can both be projection maps (in fact a single identity is all we can add here -- adding any two identities from  \eqref{system2} gives us the whole system \eqref{system2}). We have three options for the third identity: 
\begin{equation*}
\left\{
\begin{array}{r}
p(x,x,y)\approx p(x,y,y)\\
p(x,y,x) \approx q(x,x,y)\\
q(x,y,x)\approx q(y,x,x)
\end{array}\right. 
\end{equation*}
this system allows $p$ to be $\pi_1$ and $q$ to be $3x + 3y$ in a reduct over $\mathbb{Z}_5$;
\begin{equation*}
\left\{
\begin{array}{r}
p(x,x,y)\approx p(x,y,y)\\
p(x,y,x) \approx q(x,x,y)\approx q(y,x,x)
\end{array}\right. 
\end{equation*}
this system allows $p$ to be $\pi_1$ and $q$ to be $\pi_2$ in a reduct over $\mathbb{Z}_5$;
\begin{equation*}
\left\{
\begin{array}{r}
p(x,x,y)\approx p(x,y,y)\\
p(x,y,x) \approx q(x,x,y)\approx q(x,y,x)
\end{array}\right. 
\end{equation*}
this system allows both $p$ and $q$ to be $\pi_1$ in a reduct over $\mathbb{Z}_5$;

So, nothing new can be obtained with the second identity being $p(x,y,x) \approx q(x,x,y)$.

\item if we have $p(x,y,x) \approx q(x,y,x)$ as the second identity, adding a third one from \eqref{system2} may give us the following (again, adding any two identities from  \eqref{system2} gives us the whole system \eqref{system2}):
\begin{equation*}
\left\{
\begin{array}{r}
p(x,x,y)\approx p(x,y,y)\\
p(x,y,x) \approx q(x,y,x)\\
q(x,x,y)\approx q(y,x,x)
\end{array}\right. 
\end{equation*}
this system allows $p$ to be $\pi_1$ and $q$ to be $3x + 3z$ in a reduct over $\mathbb{Z}_5$;
\begin{equation*}
\left\{
\begin{array}{r}
p(x,x,y)\approx p(x,y,y)\\
p(x,y,x) \approx q(x,y,x)\approx q(x,x,y)
\end{array}\right. 
\end{equation*}
this system allows both $p$ and $q$ to be $\pi_1$ in a reduct over $\mathbb{Z}_5$;
\begin{equation*}
\left\{
\begin{array}{r}
p(x,x,y)\approx p(x,y,y)\\
p(x,y,x) \approx q(x,y,x)\approx q(y,x,x)
\end{array}\right. 
\end{equation*}
this system allows $p$ to be $\pi_1$ and $q$ to be $\pi_3$ in a reduct over $\mathbb{Z}_5$;

There are no more cases with $p(x,y,x) \approx q(x,y,x)$ as the second identity, except for the whole system \eqref{system2}.

\item if we have $p(x,y,x) \approx q(y,x,x)$ as the second identity, adding a third one from \eqref{system2} may give us the following (adding two identities gives the whole system, as before):
\begin{equation*}
\left\{
\begin{array}{r}
p(x,x,y)\approx p(x,y,y)\\
p(x,y,x) \approx q(y,x,x)\approx q(x,y,x)
\end{array}\right. 
\end{equation*}
this system allows $p$ to be $\pi_1$ and $q$ to be $\pi_3$ in a reduct over $\mathbb{Z}_5$;
\begin{equation*}
\left\{
\begin{array}{r}
p(x,x,y)\approx p(x,y,y)\\
p(x,y,x) \approx q(y,x,x)\approx q(x,x,y)
\end{array}\right. 
\end{equation*}
this system allows $p$ to be $\pi_1$ and $q$ to be $\pi_2$ in a reduct over $\mathbb{Z}_5$;
\begin{equation*}
\left\{
\begin{array}{r}
p(x,x,y)\approx p(x,y,y)\\
p(x,y,x) \approx q(y,x,x)\\
q(x,x,y)\approx q(x,y,x)
\end{array}\right. 
\end{equation*}
this system allows $p$ to be $\pi_1$ and $q$ to be $3y + 3z$ in a reduct over $\mathbb{Z}_5$;

Once again we have obtained nothing new with $p(x,y,x) \approx q(y,x,x)$ as the second identity. 
\end{itemize}
By this we have proved minimality of the system \eqref{system2}, with respect to implying omitting types 1 and 2.
\end{proof}

\noindent Further more, the following can be proved (simply by analyzing all possible cases): elimination of identities from the system \eqref{system1} gives us either a system equivalent to \eqref{system2}, up to the permutation of variables, or a system that holds in a reduct of a module over some finite ring, which means \eqref{system2} is the only system implying omitting types 1 and 2 obtainable from \eqref{system1} (and a proper subset of \eqref{system1}). We shall not provide the whole proof here, for it is too long, but only analyze two proper subsets of the system \eqref{system1} (the whole proof, however, is provided in the section \ref{peta}):
\begin{itemize}
\item[subset 1] 
\begin{equation*}
\left\{
\begin{array}{r}
p(x,x,y)\approx p(x,y,y)\\
p(x,y,x)\approx q(y,x,x)\approx q(x,y,x) \approx q(x,x,y)
\end{array}\right. 
\end{equation*}
This system implies nonexistence of $p$ (and $q$) in any reduct of a module over a finite ring, so it implies omitting types 1 and 2, but it is equivalent to system \eqref{system2}, and is obtainable from it by a permutation of variables of the term $q$ (we substitute $q(z,x,y)$ for $q(x,y,z)$).
\item[subset 2]
\begin{equation*}
\left\{
\begin{array}{r}
p(x,x,y)\approx p(x,y,x)\\
p(x,y,y)\approx q(y,x,x) \approx q(x,y,x)\approx q(x,x,y)
\end{array}\right. 
\end{equation*}
This system allows $p$ and $q$ to be respectively $4x + y + z$, $2x + 2y + 2z$ in a reduct over $\mathbb{Z}_5$, so it does not describe omitting types 1 and 2.
\end{itemize}

 \vspace{0.3cm}
 
\noindent Up to this point we have started from \eqref{sys}, eliminated the first identity obtaining \eqref{system1}, and then eliminated yet another identity from \eqref{system1} obtaining \eqref{system2}, which is proven to imply omitting types 1 and 2, and to be the only (minimal) system with this property obtainable from \eqref{system1}, up to the permutation of variables. So, the next step would be analyzing what happens if we eliminate an identity other than the first one from \eqref{sys}.

\vspace{0.2 cm}

\noindent First we shall suppose the system obtained includes the identity $p(x,x,y)\approx x$ . It also has to include an identity on  both $p$ and $q$, as explained at the beginning of the current section. So, up to this point we have a subset of the system \eqref{sys} with the first identity being $p(x,x,y)\approx x$, and the second is an identity from \eqref{sys} on $p$ and $q$. There may be more identities from the system \eqref{sys} in this subset, which (the subset) we shall denote by  $\sigma$. We keep in mind $\sigma$ needs to hold in algebra $\mathbf{B}$ from example 1, so let us discuss on possible cases:
\begin{itemize}
\item in the identity on both $p$ and $q$ (the second identity in $\sigma$), left hand side must not be $p(x,x,y)$, for if this is the case, $q$ has to be at most a binary term in algebra $\mathbf{B}$, as well as $p$. This would allow us to define both terms in a reduct over $\mathbb{Z}_{5}$, as projections and/or binary terms (some of the terms $3x + 3y$, $3x + 3z$, $3y + 3z$). So, in the second identity of the system $\sigma$ left hand side is either $p(x,y,y)$ or $p(x,y,x)$.
\item if the terms $x$ and $p(x,x,y)$ do not occur at all in the rest of the system $\sigma$, then all the identities except for the first one include only some of the terms $p(x,y,x)$, $p(x,y,y)$, $q(x,x,y)$, $q(x,y,x)$, $q(y,x,x)$. Therefore we could define $p$ and $q$ to be $4x+2y$ and $2x + 2y + 2z$ respectively in a reduct over $\mathbb{Z}_5$. So, $p(x,x,y)$ or $x$ alone  must occur somewhere in the rest of the system $\sigma$. 
\item if the term $p(x,x,y)$ (or $x$ alone) occurs in an identity on $p$ in the rest of the system, we shall obtain either $p(x,x,y)\approx x \approx p(x,y,y)$, or $p(x,x,y)\approx x \approx p(x,y,x)$. In both cases  $p$ has to be a projection map in $\mathbf{B}$, and because of the identity on both $p$ and $q$ (i.e. the second identity of the system $\sigma$), $q$ is at most a binary term in this algebra, but this allows $p$ and $q$ to be defined in a reduct over $\mathbb{Z}_{5}$, as explained earlier.  
\item if the term $p(x,x,y)$ (or $x$ alone) occurs in an identity on $p$ and $q$ (i.e. on $x$ and $q$) we shall obtain one of the following identities: $p(x,x,y)\approx q(x,x,y) \approx x $, $p(x,x,y)\approx q(y,x,x) \approx x $, $p(x,x,y)\approx q(x,y,x) \approx x $. Again each of them   means $q$ is at most a binary term in $\mathbf{B}$, as well as $p$, so both are definable in a reduct over $\mathbb{Z}_{5}$.   
\end{itemize}
By this it is proved that the system $\sigma$ can not describe omitting types 1 and 2.

\vspace{0.1 cm}

\noindent By elimination of identities from the system \eqref{sys}, we can also obtain a system that includes the identity $p(x,y,x) \approx x$ or $p(x,y,y) \approx x$ (of course, in each case it has to include an identity on both $p$ and $q$). Let us provide a brief overview on these systems:
\begin{itemize}
\item[case 1] Let $\sigma_1$ be a subset of the system \eqref{sys} with the first identity being $p(x,y,x) \approx x$ (the second identity in $\sigma_1$ is an identity on $p$ and $q$ from \eqref{sys}).  We can obtain an equivalent system  by a permutation of variables of the term $p$ (we substitute $p(x,y,z)$ for $p(x,z,y)$ in $\sigma_1$ ) but this is the system of the form $\sigma$, with the first identity being $p(x,x,y)\approx x$,  which we have already proved not to describe omitting types 1 and 2. 
\item [case 2] Let $\sigma_2$ be a subset of the system \eqref{sys} with the first identity being $p(x,y,y) \approx x$ (the second identity in $\sigma_2$ is an identity on $p$ and $q$ from \eqref{sys}). Since the system $\sigma_2$ needs to hold in $\mathbf{B}$, $p$ is to be $\pi_1$ and $q$ at most a binary term in this algebra. This allows both of them to be defined in a reduct over $\mathbb{Z}_{5}$. Therefore $\sigma_2$ cannot describe omitting types 1 and 2.
\end{itemize}

\noindent We can conclude now that none of the systems obtained from \eqref{sys} including either of the identities $x \approx p(x,x,y)$, $x \approx p(x,y,x)$, $x \approx p(x,y,y)$ can describe omitting types 1 and 2.

\vspace{0.3 cm}

\noindent Let us now discuss the subsets of the system \eqref{sys} in which the first identity is one of the following three: $x \approx q(x,x,y)$, $x \approx q(x,y,x)$, $x \approx q(y,x,x)$. By examination of these systems (which is, for being too long, provided in a separate section, i.e.   \ref{sesta}, of this paper), we come to a single system, up to the permutation of variables, that implies (and possibly describes) omitting types 1 and 2:
\begin{equation}
\left\{
\begin{array}{r}
x\approx q(x,y,x)\\
p(x,y,y)\approx p(x,y,x)\\ 
p(x,x,y)\approx q(x,x,y) \approx q(y,x,x)\label{newone}
\end{array}\right.
\end{equation}

\vspace{0.3 cm}

\subsection{both majority terms} \label{case3}
If both $p$ and $q$ have to be majority terms in $\mathbf{A}$ they satisfy the following in this algebra:
\begin{equation}
\begin{array}{r}
x\approx p(x,x,y)\approx p(x,y,x)\approx p(y,x,x)\approx q(y,x,x)\approx q(x,y,x)\approx q(x,x,y)
\label{system4}
\end{array}
\end{equation}
\noindent It is easily seen that $p$ and $q$ satisfying the system cannot exist in $\mathbf{B}$ (example 1), so some identities have to be eliminated. Notice that we have to keep at least one  identity having $x$ on the left, for if we do not, the remaining system holds in a reduct of a module over $\mathbb{Z}_5$ (we can define both $p$, $q$ to be $2x + 2y + 2z$). 

Let us try eliminating the identity $p(x,x,y)\approx p(x,y,x)$. Now we have the following system that holds in $\mathbf{B}$:
\begin{equation}
\left\{
\begin{array}{r}
x\approx p(x,x,y)\\
p(x,y,x)\approx p(y,x,x)\approx q(y,x,x)\approx q(x,y,x)\approx q(x,x,y)
\label{system3}
\end{array}\right.
\end{equation}
\noindent We shall prove now that the system \eqref{system3} implies omitting types 1 and 2. 
\begin{proof}
To prove that the system implies omitting types 1 and 2 it is sufficient to show that $p$, $q$ cannot exist in any reduct of a module over a finite ring. From the identities of the system it is easily seen that $q$ can only be a ternary term in any reduct (if it exists at all) i.e. $\alpha x + \beta y + \gamma z$ for $\alpha + \beta + \gamma =1$ and none of $\alpha, \beta, \gamma$ is zero. From the forth identity we obtain the following: $\alpha y + (\beta + \gamma)x = (\alpha + \gamma )x + \beta y$, and this implies $\alpha = \beta$. Then from the last identity we have: $(\alpha + \gamma)x + \alpha y = 2\alpha x + \gamma y$ , which gives $\alpha = \gamma$ also. Therefore, if $q$ exists in any reduct of a module over a finite ring, it has the form $\alpha x + \alpha y + \alpha z$, for $3\alpha = 1$. Then, from the third identity we have $p(x,y,z)\approx \alpha x + 2\alpha y$ (since $p$ can only be a binary term in any reduct), and therefore the second identity gives $\alpha = 2\alpha$ (this should hold in a finite ring mentioned), which implies $\alpha = 0$, and this is a contradiction. By this we have proved that $q$ (and consequently $p$) cannot exist in any reduct of a module over any finite ring, which means the system  \eqref{system3} implies omitting types 1 and 2.
\end{proof}

\vspace{0.3cm}

\noindent The system \eqref{system3} is also a minimal system implying omitting types 1 and 2, for any proper subset of \eqref{system3} holds in some reduct of a module over a finite ring (for some terms $p$ and $q$). This is proven the same way as minimality of the system \eqref{system2}, i.e. by analyzing all possible cases. The proof is omitted in this subsection, but is provided in the section \ref{sedma}. 

\vspace{0.3 cm}

\noindent If we return to the system \eqref{system4} and analyze other ways to eliminate identities in order to obtain a system that describes omitting types 1 and 2, we come to the following conclusion: if a system obtained (by elimination of identities from \eqref{system4}) holds in algebra $\mathbf{B}$, it is either equivalent to \eqref{system3} up to the permutation of variables, or it holds in a reduct of a module over some finite ring (once again, the examination of all the subsets of the system \eqref{system4} is done in the section \ref{sedma} of this paper).

\vspace{0.1cm}

\noindent We can conclude the following: from the system \eqref{system4} we can obtain a single system, up to the permutation of variables, which possibly describes omitting types 1 and 2 and that is the system \eqref{system3}.

\vspace{0.3 cm}

\noindent By this we have examined all possible forms of a system on $p$ and $q$ regarding the existence of terms in $\mathbf{A}$, and obtained three systems (that is \eqref{system2}, \eqref{newone}, \eqref{system3}), on variables $x$ and $y$, that imply and possibly describe omitting types 1 and 2. Before the conclusion there is yet another question left to examine -- can we obtain anything new from systems of identities on $p$ and $q$ with more than two variables?
 
\vspace{0.3 cm}
\subsection{systems of identities on $\mathbf{p}$ and $\mathbf{q}$ including more than two variables}
\noindent In this subsection we shall prove that nothing new can be obtained from systems of identities on $p$ and $q$ including more than two variables.

\vspace{0.1 cm}

\noindent Suppose we have a system on two ternary idempotent terms $p$ and $q$ that describes omitting types 1 and 2, and suppose there is an identity (or more of them) including more than two variables in this system. We shall denote this system by $\tau$ and discuss possible cases according to the number of variables:
\begin{itemize} 
\item if there is an identity including six variables in $\tau$ , it can only be one of these two identities: $p(x,y,z) \approx p(u,v,w)$, $p(x,y,z) \approx q(u,v,w)$. In both cases we obtain the identity  $x \approx w$, which only holds in a trivial algebra (i.e. a trivial variety), so this case is impossible (assuming that $\tau$ describes omitting types 1 and 2). 
\item if there is an identity including five variables in the system $\tau$, there are two possibilities:
\begin{itemize}
\item the identity mentioned can yield a trivial variety -- this happens there are $x,y,z$ on the left and $u,v$ on the right,  e.g. $p(x,y,z) \approx p(u,u,v)$, $p(x,y,z) \approx q(v,u,v)$, but also in these two cases: $p(x,y,z) \approx p(u,x,v)$, $p(x,y,z) \approx p(u,v,x)$ (if we substitute $x$ for $u$ and $v$ in the latter two identities, we shall obtain $p(x,y,z) \approx x$, which means $p$ has to be the first projection map, but that would give us  $x \approx u$, which yields a trivial variety). 
\item the identity considered can imply that one of the terms $p$ and $q$, or both of them, must be a projection map (maps) -- this happens if the identity is one of the following: $p(x,y,z) \approx p(x,u,v)$, $p(x,y,z) \approx q(x,u,v)$, $p(x,y,z) \approx q(u,x,v)$, $p(x,y,z) \approx q(u,v,x)$. The identity $p(x,y,z) \approx p(x,u,v)$ implies that $p$ has to be a projection map, which exists in any algebra, so in this case we can substitute $x$ (or $y$ or $z$) for $p(x,y,z)$ in the system $\tau$, obtaining a system only on $q$, which cannot describe omitting types 1 and 2 (this has been proven in the subsection \ref{singleternary}). If both terms have to be projection maps then the system $\tau$ obviously can not describe omitting types 1 and 2.
 \end{itemize}
\item if there is an identity including four variables in the system $\tau$, there are three possibilities:
\begin{itemize}
\item the identity mentioned can yield a trivial variety, like in these cases: $p(x,y,z) \approx w$, $p(x,y,z) \approx p(w,w,x)$, $p(x,y,z) \approx p(w,x,w)$, $p(x,y,z) \approx p(w,x,x)$, $p(x,y,y) \approx p(w,w,z)$... Obviously none of these identities cannot occur in the system $\tau$.
\item the identity including four variables can imply that one of the terms has to be a projection map in any algebra, i.e.: $p(x,y,z) \approx p(x,w,w)$, $p(x,y,z) \approx q(w,w,x)$, $p(x,y,z) \approx p(x,w,x)$, $p(x,y,z) \approx q(w,x,x)$, $p(x,y,x) \approx p(z,w,x)$, $p(x,y,x) \approx q(z,w,x)$ ... If this were the case, we could substitute a single variable for one of the terms in $\tau$ (for example, if $p$ has to be the first projection map then we can substitute $x$ for $p(x,y,z)$ in the whole system), obtaining a system on a single ternary term which does not describe omitting types 1 and 2 (proved in the subsection \ref{singleternary}). Therefore none of the identities from above may occur in the system $\tau$.
\item the identity including four variables can imply that one of the terms has to be at most a binary term in any algebra, i.e.: $p(x,y,z) \approx p(w,x,y)$, $p(x,y,z) \approx p(x,w,y)$, $p(x,y,z) \approx q(y,x,w)$...(from the first identity we obtain $p(x,y,z) \approx p(x,x,y)\approx t(x,y)$, for some new binary term $t$). This means we can substitute a binary term $t(x,y)$ for $p(x,y,z)$ in the whole system $\tau$, obtaining a system on a binary and a ternary term ($t$ and $q$ respectively), and this is proven not to describe omitting types 1 and 2 (subsection \ref{bin_ter}). Therefore none of the identities from above may occur in the system $\tau$.
\end{itemize}
\item if there is an identity including three variables in the system $\tau$, these are the possible cases:
\begin{itemize}
\item there are three variables on one side of the identity , and only two of them on the other:$p(x,y,z) \approx p(x,x,y)$, $p(x,y,z) \approx q(x,y,y)$... In these cases we can substitute a new binary term $t(x,y)$ for the term $p$ in the whole system $\tau$, obtaining a system on a binary and a ternary term which cannot describe omitting types 1 and 2 (subsection \ref{bin_ter}). Therefore none of these identities may occur in the system $\tau$. 
\item there are three variables on both sides of the identity: $p(x,y,z) \approx p(x,z,y)$, $p(x,y,z) \approx p(y,x,z)$, $p(x,y,z) \approx p(y,z,x)$... Each of these identities may occur in $\tau$. On the other side, $\tau$ can not include any of the identities $p(x,y,z) \approx q(x,z,y)$, $p(x,y,z) \approx q(y,z,x)$, $p(x,y,z) \approx q(y,x,z)$, etc, for if this were the case we could simply substitute $p$ for $q$ in the whole system obtaining a system only on $p$ which cannot describe omitting types 1 and 2 (subsection \ref{singleternary}). So, $\tau$ can include only some of the identities on $p$ with three variables on both sides.
\item there are two variables on each side of the identity: $p(x,y,y) \approx p(x,z,x)$, $p(x,y,x) \approx p(z,z,x)$... Each of these identities may occur in the system $\tau$. As for the identities on both $p$ and $q$ (such as $p(x,y,y) \approx q(x,z,x)$, $p(x,x,y) \approx q(x,z,z)$...), we can notice the following: in algebra $\mathbf{B}$ from example 1 both terms have to be at most binary, but if the system $\tau$ allows that, both terms can be defined in a reduct over $\mathbb{Z}_5$ (as at most binary also). This is impossible assuming that $\tau$ describes omitting types 1 and 2, so $\tau$ includes no identities on $p$ and $q$ with two variables on each side. 
\end{itemize}

\vspace{0.1 cm}

\noindent So, if $\tau$ includes identities with more than two variables, they can only have three variables, and be of two kinds:
\begin{itemize}
\item identities on $p$ with $x$, $y$, $z$ on both sides
\item identities on $p$ with $x$, $y$ on one side and $x$, $z$ on another
\end{itemize}
\end{itemize}

\vspace{0.1 cm}

\noindent Now, the system $\tau$ has to hold in algebras $\mathbf{A}$ and $\mathbf{B}$, so if we substitute both $x$ and $y$ for $z$ in $\tau$ we shall obtain a system -- consequence (with more identities, but including only $x$ and $y$) that also holds in these two algebras. Let us denote this new system by $\tau_1$, and discuss what happens in a reduct of a module over a finite ring:
\begin{itemize}
\item identities on $p$ with $x$, $y$, $z$ on both sides
\begin{itemize}
\item if we substitute both $x$ and $y$ for $z$ in the identity $p(x,y,z) \approx p(x,z,y)$, we shall obtain the following two identities: $p(x,y,x) \approx p(x,x,y)$, $p(x,y,y) \approx p(x,y,y)$. The second identity is obviously a trivial one, and the first can hold in a reduct of a module over a finite ring if we define $p$ to be the first projection map, or a term $\alpha x+\beta y+ \beta z$, where $\alpha + 2\beta =1$. In both cases the identity $p(x,y,z) \approx p(x,z,y)$ holds in the same reduct.
\item if we substitute both $x$ and $y$ for $z$ in the identity $p(x,y,z) \approx p(z,x,y)$, we shall obtain the following two identities: $p(x,y,x) \approx p(x,x,y)$, $p(x,y,y) \approx p(y,x,y)$. These two hold in a reduct of a module over a finite ring if we define $p$ to be $\alpha x+\alpha y+ \alpha z$ (of course $3 \alpha = 1$), but this means that the identity $p(x,y,z) \approx p(z,x,y)$ also holds in this reduct.
\item the same holds for the identity $p(x,y,z) \approx p(z,y,x)$-- namely, by substituting $x$ and $y$ for $z$ we obtain two identities, $p(x,y,x) \approx p(x,y,x)$, $p(x,y,y) \approx p(y,y,x)$  . If we define $p$ (in any possible way) so that the two identities obtained hold in some reduct of a module over a finite ring, then the identity $p(x,y,z) \approx p(z,y,x)$ also holds in that reduct for the same term $p$. 
\item it is easy to check that the same holds for the identities $p(x,y,z) \approx p(y,z,x)$, $p(x,y,z) \approx p(y,x,z)$.
\end{itemize}
\item identities on $p$ with $x$, $y$ on one side and $x$, $z$ on another 
\begin{itemize}
\item it is easy to see that an identity like that does not hold in any reduct of a module over a finite ring if and only if the variable $z$ on the right side of the identity occurs in all the positions where $x$ is on the left (and perhaps in some more), e.g. $p(x,y,x) \approx p(z,x,z)$, $p(x,y,y) \approx p(z,z,x)$... For all of these identities (that do not hold in any reduct) holds the following: if we substitute both $x$ and $y$ for $z$, we shall obtain two identities such that both of them also cannot hold in any reduct (e.g. from the identity $p(x,y,x) \approx p(z,x,z)$ we would obtain the following two $p(x,y,x) \approx x$ , $p(x,y,x) \approx p(y,x,y)$. Now from the first one we have that $p$ would have to be one of the terms $\pi_1$, $\pi_3$, $\alpha x +\beta z$, but the second identity does not allow either of these).
\item according to the previous item, the following holds: if we have an identity on $p$ with $x$, $y$ on one side and $x$, $z$ on another, and by substituting both $x$ and $y$ for $z$ we obtain two identities that hold in some full idempotent reduct over a finite ring (for some term $p$), then the identity we started with (the one including $z$) also holds in that reduct for the same term $p$.  
\end{itemize}
\end{itemize}

\vspace{0.1 cm}

\noindent We can now state the following: if the system $\tau_1$ (the system -- consequence,  obtained by substituting both $x$ and $y$ for $z$ in the system $\tau$) holds in some reduct of a module over a finite ring then the system $\tau$ also holds in that reduct. Since we want $\tau$ to describe omitting types 1 and 2, by substituting both $x$ and $y$ for $z$ in $\tau$ we need to  obtain a system with only $x$ and $y$ (i.e. $\tau_1$) that holds in $\mathbf{A}$ and $\mathbf{B}$ and does not hold in any reduct of a module. There are only three systems with only $x$ and $y$ that satisfy this, and these are the systems  \eqref{system2}, \eqref{newone}, \eqref{system3} obtained in the subsections \ref{cases12} and \ref{case3}  . Therefore $\tau_1$ has to contain one of these systems or actually be one of them. In either case, the system $\tau$ with three variables is a stronger condition compared to the obtained system $\tau_1$ with $x$ and $y$ only, so there is no need to consider it. In other words, we have just proved that nothing new can be obtained from systems of identities on $p$ and $q$ including more than two variables.

\vspace{0.3 cm}

\noindent Now we can state the following result:
\begin{tvrdjenje} If it is possible to describe omitting types 1 and 2 by two ternary terms $p$ and $q$, it can only be done by one or more of the systems \eqref{system2}, \eqref{newone}, \eqref{system3}.
\end{tvrdjenje} 
\vspace{1cm}
 
\noindent $\mathbf{Problem:}$ As we have mentioned in the abstract, it is not resolved whether any of these systems actually describes omitting unary and affine types. Finding a counterexample for each of the systems (i.e. a finite algebra omitting types 1 and 2 and not having terms $p$ and $q$) would, of course, lead to conclusion that it is impossible to describe omitting unary and affine types by two ternary terms. On the other hand, perhaps it is possible to prove that any algebra having terms $p$ and $q$ (from any of the systems above) also has terms $v$ and $w$ from theorem 1.3, or that its congruence lattice is meet semi--distributive. In either case it would mean that the system considered describes omitting types 1 and 2.

\vspace{1 cm}

\noindent In the following sections, i.e. \ref{peta}, \ref{sesta} and \ref{sedma}, we shall provide complete proofs of some statements given above. The proofs are given at the end of the paper and in the separate sections for they are too long and based almost entirely on case--analysis.

\section {} \label{peta}

In this section we shall prove the following (this proof is excluded from the subsection \ref{cases12}): 

\vspace{0.5 cm}

\noindent from the system 
\begin{equation}
\begin{array}{r}
p(x,x,y)\approx p(x,y,y)\approx p(x,y,x) \approx q(x,x,y)\approx q(x,y,x)\approx q(y,x,x)
\end{array}\label{jedan}
\end{equation}
(also denoted by (3) in the subsection \ref{cases12} above), which implies omitting types 1 and 2 but is not minimal in that respect, we can obtain only one of the following:
\begin{itemize}
\item [-] the system 
\begin{equation}
\left\{
\begin{array}{r}
p(x,x,y)\approx p(x,y,y)\\
p(x,y,x)\approx q(x,x,y) \approx q(x,y,x) \approx q(y,x,x)
\end{array}\right. \label{dva}
\end{equation} 
(also denoted by (4) in the subsection \ref{cases12} above). This system implies omitting types 1 and 2 and is minimal having this property, which is already proved in \ref{cases12}.
\item [-] a system equivalent to \eqref{dva} (obtained by a permutation of variables)
\item [-] a system that holds in some full idempotent reduct of a module over some finite ring.
\end{itemize}
In other words, up to a permutation of variables, \eqref{dva} is the only system implying omitting types 1 and 2 obtainable from \eqref{jedan} (and a proper subset of \eqref{jedan}).

\vspace{0.2 cm}

\noindent To prove this we need to analyze all proper subsets of the system \eqref{jedan}, and to show that each of them is either equivalent to \eqref{dva} up to a permutation of variables, or can hold in a full idempotent reduct of a module over a finite ring. 
Let us notice the following:
\begin{enumerate}
\item if there is no identity on $p$ and $q$ in a system considered (i.e. a proper subset of \eqref{jedan}), then $p$ and $q$ exist in a full idempotent reduct of a module over $\mathbb{Z}_5$, for we can define $p$ to be $\pi_1$ (the first projection map) and $q$ to be $2x+2y+2z$.  This means there have to be an identity on both $p$ and $q$ in any subset of \eqref{jedan} possibly describing omitting types 1 and 2.
\item if, except for the mentioned identity on $p$ and $q$, term $p$ does not occur in any other identity (i.e. we have an identity on $p$ and $q$ and some identities only on $q$), again both $p$ and $q$ can be defined in a full idempotent reduct of a module over $\mathbb{Z}_5$, $q(x,y,z)$ being $2x+2y+2z$, and $p(x,y,z)$ being one of the terms $4x+2y$, $4x+2z$. So, term $p$ has to occur in at least one more identity -- an identity only on $p$ or on $p$ and $q$.
\end{enumerate}
We shall now go through the cases in the following manner: we shall take an identity on $p$ and $q$ from \eqref{jedan} to be the first one, add an identity only on $p$ from \eqref{jedan} as the second one and discuss what happens. Then we change the second identity for another one only on $p$ from \eqref{jedan} and so on. After we have examined all the cases including the second identity on $p$, we examine the cases with second identity on both $p$ and $q$ (the first identity stays the same all the time). Finally, we change the first identity for another identity on $p$ and $q$ (taken from \eqref{jedan}, of course), and repeat the whole procedure. 
\subsection{} \label{sub1}
In this subsection we shall analyze all proper subsets of the system \eqref{jedan} including   the identity $p(x,y,x) \approx q(x,x,y)$. We shall consider this identity the first in each of these subsets. As explained, now we vary the second identity:
\begin{itemize}
\item the second identity is $p(x,x,y) \approx p(x,y,y)$, i.e. so far we have this system:
\begin{equation}
\left\{
\begin{array}{c}
p(x,y,x) \approx q(x,x,y)\\
p(x,x,y) \approx p(x,y,y) \\
\end{array}\right. \label{b}
\end{equation}
Obviously, both $p$ and $q$ can be projection maps (in any algebra), so we need to add some more identities from \eqref{jedan}. This can be done in a few ways: we can equalize the terms of the two identities given, or equalize either of the terms $q(x,y,x)$, $q(y,x,x)$ with the terms of either of the identities given, or add the identity $q(x,y,x) \approx q(y,x,x)$ to the given two. Let us discuss each possibility:
\begin{itemize}
\item by equalizing the terms of the two identities given, we obtain the following:\begin{equation}
\begin{array}{c}
p(x,y,x) \approx q(x,x,y)\approx
p(x,x,y) \approx p(x,y,y) \label{a}
\end{array}
\end{equation}
Still both $p$ and $q$ can be defined as projections, so more identities from \eqref{jedan} are needed here. We shall discuss each way of adding an identity from \eqref{jedan} to this system: 

If we equalize $q(x,y,x)$ with the terms in  \eqref{a}, we shall obtain \begin{equation*}
\begin{array}{c}
p(x,y,x) \approx q(x,x,y)\approx
p(x,x,y) \approx p(x,y,y)\approx q(x,y,x) 
\end{array}
\end{equation*}
but this still allows both projections, and adding more identities from \eqref{jedan} actually gives us the whole system \eqref{jedan}. 

If we equalize $q(y,x,x)$ with the terms in \eqref{a}, we shall obtain \begin{equation*}
\begin{array}{c}
p(x,y,x) \approx q(x,x,y)\approx
p(x,x,y) \approx p(x,y,y)\approx q(y,x,x) 
\end{array}
\end{equation*}
and for this system holds the same as for the previous. 

If we add the identity $q(x,y,x) \approx q(y,x,x)$ to the system \eqref{a}, we shall  obtain the following: \begin{equation*}
\left\{
\begin{array}{c}
p(x,y,x) \approx q(x,x,y)\approx
p(x,x,y) \approx p(x,y,y)\\
q(x,y,x) \approx q(y,x,x) 
\end{array}\right.
\end{equation*}
This system holds in a full idempotent reduct of a module over $\mathbb{Z}_5$ for $p(x,y,z)$ being the first projection map, and $q(x,y,z)$ being the term $3x+3y$. Adding any new identities from \eqref{jedan} gives us the whole system \eqref{jedan}. 

\noindent By this we have finished analyzing the system \eqref{a}. By adding identities to this system we can obtain either a system that holds in some full idempotent reduct of a module over a finite ring, or the whole system \eqref{jedan}.
\item equalizing the term $q(x,y,x)$ with the terms of the first identity of \eqref{b} gives us the following system:\begin{equation}
\left\{
\begin{array}{c}
p(x,y,x) \approx q(x,x,y)\approx q(x,y,x)\\
p(x,x,y) \approx p(x,y,y) \label{c}
\end{array}\right. 
\end{equation}
Once again both terms can be projections, so more identities from \eqref{jedan} are needed here. We shall discuss each way of adding an identity from \eqref{jedan} to this system:

If we equalize the terms of the identities given  we shall obtain the system \begin{equation*}
\begin{array}{c}
p(x,y,x) \approx q(x,x,y)\approx q(x,y,x)\approx
p(x,x,y) \approx p(x,y,y) 
\end{array}
\end{equation*}
but this still allows both projections, and adding any new identities from \eqref{jedan} gives us \eqref{jedan}.

If we equalize the term $q(y,x,x)$ with the terms of the first identity of the system \eqref{c}, we shall obtain the following\begin{equation*}
\left\{
\begin{array}{c}
p(x,y,x) \approx q(x,x,y)\approx q(x,y,x)\approx q(y,x,x)\\
p(x,x,y) \approx p(x,y,y) 
\end{array}\right. 
\end{equation*} 
but this is exactly the system \eqref{dva}.

If we equalize the term $q(y,x,x)$ with the terms of the second identity of the system \eqref{c}, we shall  obtain the following\begin{equation*}
\left\{
\begin{array}{c}
p(x,y,x) \approx q(x,x,y)\approx q(x,y,x)\\
p(x,x,y) \approx p(x,y,y)\approx q(y,x,x) 
\end{array}\right. 
\end{equation*}
but this system allows both projection maps ($p$ and $q$ being $\pi_3$, $\pi_1$ respectively). Adding any new identities from \eqref{jedan} gives us \eqref{jedan}.

\noindent By this we have finished analyzing the system \eqref{c}. By adding identities to this system we can obtain a system that holds in some full idempotent reduct of a module over a finite ring, or the system \eqref{dva}, or the whole system \eqref{jedan}.
\item equalizing the term $q(x,y,x)$ with the terms of the second identity of \eqref{b} gives us the following system:\begin{equation}
\left\{
\begin{array}{c}
p(x,y,x) \approx q(x,x,y)\\
p(x,x,y) \approx p(x,y,y)\approx q(x,y,x)\label{d}
\end{array}\right. 
\end{equation}
The system allows both projections, so more identities are needed. Let us discuss each way of adding an identity from \eqref{jedan} to this system: 

If we equalize the terms of the identities  given, we shall obtain the system \begin{equation*}
\begin{array}{c}
p(x,y,x) \approx q(x,x,y)\approx
p(x,x,y) \approx p(x,y,y)\approx q(x,y,x) 
\end{array}
\end{equation*}
Still both projections allowed, and by adding any new identity from \eqref{jedan} we obtain \eqref{jedan}. 

If we equalize the term $q(y,x,x)$ with the terms of the first identity of the system \eqref{d}, we shall obtain the following \begin{equation*}
\left\{
\begin{array}{c}
p(x,y,x) \approx q(x,x,y)\approx q(y,x,x)\\
p(x,x,y) \approx p(x,y,y)\approx q(x,y,x)
\end{array}\right. 
\end{equation*}
but this system allows both projections, and by adding any identity we obtain \eqref{jedan}. 

If we equalize the term $q(y,x,x)$ with the terms of the second identity of the system \eqref{d}, we shall  obtain the following \begin{equation*}
\left\{
\begin{array}{c}
p(x,y,x) \approx q(x,x,y)\\
p(x,x,y) \approx p(x,y,y)\approx q(x,y,x)\approx q(y,x,x)
\end{array}\right. 
\end{equation*} 
this system holds in a full idempotent reduct of a module over $\mathbb{Z}_5$, for we can define $p$ and $q$ to be $3x+3z$ and $3x+3y$ respectively in this reduct. By adding identities we can only obtain the whole system \eqref{jedan}.
\noindent By this we have finished analyzing the system \eqref{d}. By adding identities to this system we can obtain either a system that holds in some full idempotent reduct of a module over a finite ring, or the whole system \eqref{jedan}. 
\item equalizing the term $q(y,x,x)$ with the terms of the first identity of \eqref{b} gives us the following system:\begin{equation}
\left\{
\begin{array}{c}
p(x,y,x) \approx q(x,x,y)\approx q(y,x,x)\\
p(x,x,y) \approx p(x,y,y) \\
\end{array}\right. \label{e}
\end{equation}
This system allows both projections, so we need to add more identities from \eqref{jedan}. Let us discuss each way of adding an identity from \eqref{jedan} to this system: 

If we equalize the terms of the  identities given, we shall obtain the system\begin{equation*}
\begin{array}{c}
p(x,y,x) \approx q(x,x,y)\approx q(y,x,x)\approx
p(x,x,y) \approx p(x,y,y) \\
\end{array} 
\end{equation*}
This system also allows both projections, but adding any new identity gives us \eqref{jedan}.

If we equalize the term $q(x,y,x)$ with the terms of the first identity of the system \eqref{e}, we shall obtain the system \eqref{dva}.

If we equalize the term $q(x,y,x)$ with the terms of the second identity of the system \eqref{e}, we shall  obtain the system\begin{equation*}
\left\{
\begin{array}{c}
p(x,y,x) \approx q(x,x,y)\approx q(y,x,x)\\
p(x,x,y) \approx p(x,y,y)\approx q(x,y,x) \\
\end{array}\right. 
\end{equation*} 
which also allows both projections, and by adding identities from \eqref{jedan} we can only obtain \eqref{jedan}.
\noindent By this we have finished analyzing the system \eqref{e}. By adding identities to this system we can obtain either the system \eqref{dva}, or the whole system \eqref{jedan}. 
 
\item  equalizing the term $q(y,x,x)$ with the terms of the second identity of \eqref{b} gives us the following system:\begin{equation}
\left\{
\begin{array}{c}
p(x,y,x) \approx q(x,x,y)\\
p(x,x,y) \approx p(x,y,y)\approx q(y,x,x) \\
\end{array}\right. \label{f}
\end{equation}
The system allows both projections, so more identities are needed. As before, we shall discuss each way of adding an identity from \eqref{jedan} to this system: 

If we equalize the terms of the identities given,  we shall obtain the system \begin{equation*}
\begin{array}{c}
p(x,y,x) \approx q(x,x,y)\approx
p(x,x,y) \approx p(x,y,y)\approx q(y,x,x) \\
\end{array}
\end{equation*}
but this system still allows both projections, and by adding any new identities from \eqref{jedan} we can only obtain the whole system \eqref{jedan}.

If we equalize the term $q(x,y,x)$ with the terms of the first identity of the system \eqref{f}, we shall obtain the system\begin{equation*}
\left\{
\begin{array}{c}
p(x,y,x) \approx q(x,x,y)\approx q(x,y,x)\\
p(x,x,y) \approx p(x,y,y)\approx q(y,x,x) 
\end{array}\right. 
\end{equation*}
which allows both projections, but adding identities from \eqref{jedan} can only give us the whole system \eqref{jedan}.

If we equalize the term $q(x,y,x)$ with the terms of the second identity of the system \eqref{f}, we shall obtain the system\begin{equation*}
\left\{
\begin{array}{c}
p(x,y,x) \approx q(x,x,y)\\
p(x,x,y) \approx p(x,y,y)\approx q(y,x,x)\approx q(x,y,x) 
\end{array}\right. 
\end{equation*}
This system holds in a full idempotent reduct of a module over $\mathbb{Z}_5$, for we can define $p$ and $q$ to be $3x+3z$ and $3x+3y$ respectively in this reduct. By adding identities we can only obtain the whole system \eqref{jedan}.
\noindent By this we have finished analyzing the system \eqref{f}. By adding identities to this system we can obtain either a system that holds in some full idempotent reduct of a module over a finite ring, or the whole system \eqref{jedan}. 

\item adding the identity $q(x,y,x)\approx q(y,x,x)$ to the system \eqref{b} gives us the following system:\begin{equation}
\left\{
\begin{array}{c}
p(x,y,x) \approx q(x,x,y)\\
p(x,x,y) \approx p(x,y,y) \\
q(x,y,x) \approx q(y,x,x)
\end{array}\right. \label{g}
\end{equation} 
This system holds in a full idempotent reduct of a module over $\mathbb{Z}_5$, for we can define $p$ and $q$ to be $3x+3z$ and $3x+3y$ respectively in this reduct. Therefore we need to add some new identities from \eqref{jedan}, and this is done only by equalizing the terms of any two of the three identities given. We shall discuss all possible cases:

If we equalize the terms of the first and the second identity from \eqref{g} we shall obtain the following system:
\begin{equation*}
\left\{
\begin{array}{c}
p(x,y,x) \approx q(x,x,y)\approx 
p(x,x,y) \approx p(x,y,y) \\
q(x,y,x) \approx q(y,x,x)
\end{array}\right. 
\end{equation*}
This system holds in a full idempotent reduct of a module over $\mathbb{Z}_5$, for we can define $p$ and $q$ to be $\pi_1$ and $3x+3y$ respectively in this reduct. This means we have to add more identities from \eqref{jedan}, but that can only give us the whole system \eqref{jedan}.

If we equalize the terms of the first and the third identity from \eqref{g}, we shall obtain the following system:\begin{equation*}
\left\{
\begin{array}{c}
p(x,y,x) \approx q(x,x,y)\approx q(x,y,x) \approx q(y,x,x) \\
p(x,x,y) \approx p(x,y,y) \\
\end{array}\right. 
\end{equation*}
but this is exactly the system \eqref{dva}.

If we equalize the terms of the second and the third identity from \eqref{g} we shall obtain the following system:\begin{equation*}
\left\{
\begin{array}{c}
p(x,y,x) \approx q(x,x,y)\\
p(x,x,y) \approx p(x,y,y) \approx
q(x,y,x) \approx q(y,x,x)
\end{array}\right. 
\end{equation*}
This system holds in a full idempotent reduct of a module over $\mathbb{Z}_5$, for we can define $p$ and $q$ to be $3x+3z$ and $3x+3y$ respectively in this reduct. By adding any identity from \eqref{jedan} we obtain \eqref{jedan}.
\noindent By this we have finished analyzing the system \eqref{g}. By adding identities to this system we can obtain a system that holds in some full idempotent reduct of a module over a finite ring, or the system \eqref{dva}, or the whole system \eqref{jedan}. 
\end{itemize}
Up to this point we have examined all the systems that include identities \eqref{b}, and we have obtained nothing else  but the system \eqref{dva} implying omitting types 1 and 2. We shall now replace the second identity in \eqref{b} for another one on $p$.
\item the second identity is $p(x,x,y) \approx p(x,y,x)$, i.e. we have these identities:\begin{equation}
\begin{array}{c}
p(x,y,x) \approx q(x,x,y)\approx p(x,x,y) 
\end{array}\label{h}
\end{equation} 
The system \eqref{h} obviously allows both projections and we need to add more identities from \eqref{jedan}. This can be done in the following ways: we can equalize each of the terms  $p(x,y,y)$,  $q(x,y,x)$,  $q(y,x,x)$  with the terms of the existing identities, or add a new identity (i.e. one of these three $p(x,y,y) \approx q(x,y,x)$, $p(x,y,y) \approx q(y,x,x)$, $q(x,y,x)\approx q(y,x,x)$) to the existing ones. As before we shall discuss each of these cases:
\begin{itemize}
\item equalizing the term $p(x,y,y)$ with the terms of the identities \eqref{h} gives us the following:\begin{equation}
\begin{array}{c}
p(x,y,x) \approx q(x,x,y)\approx p(x,x,y)\approx p(x,y,y) 
\end{array}\label{i}
\end{equation}
obviously this system allows both projections, so we need to add some more identities from \eqref{jedan}. As before, we shall discuss each way of doing this:

If we equalize the term $q(x,y,x)$ with the terms of \eqref{i}, we shall obtain the system \begin{equation*}
\begin{array}{c}
p(x,y,x) \approx q(x,x,y)\approx p(x,x,y)\approx p(x,y,y)\approx q(x,y,x) 
\end{array}
\end{equation*}
which still allows both projections, but adding any new identity from \eqref{jedan} gives us \eqref{jedan}.

If we equalize the term $q(y,x,x)$ with the terms of \eqref{i}, we shall obtain the system \begin{equation*}
\begin{array}{c}
p(x,y,x) \approx q(x,x,y)\approx p(x,x,y)\approx p(x,y,y)\approx q(y,x,x) 
\end{array}
\end{equation*}
and for this system holds the same as in the previous case.

If we add the identity $q(x,y,x) \approx q(y,x,x)$ to the system \eqref{i}, we shall obtain the system \begin{equation*}
\left\{
\begin{array}{c}
p(x,y,x) \approx q(x,x,y)\approx p(x,x,y)\approx p(x,y,y)\\
q(x,y,x) \approx q(y,x,x) 
\end{array}\right. 
\end{equation*}
This system holds in a full idempotent reduct of a module over $\mathbb{Z}_5$, for we can define $p$ and $q$ to be $\pi_1$ and $3x+3y$ respectively in this reduct. By adding any identity from \eqref{jedan} we obtain \eqref{jedan}.
\noindent By this we have finished analyzing the system \eqref{i}. By adding identities to this system we can obtain either a system that holds in some full idempotent reduct of a module over a finite ring, or the whole system \eqref{jedan}. 

\item equalizing the term $q(x,y,x)$ with the terms of the identities \eqref{h} gives us the following:\begin{equation}
\begin{array}{c}
p(x,y,x) \approx q(x,x,y)\approx p(x,x,y)\approx q(x,y,x) 
\end{array}\label{j}
\end{equation}
obviously this system allows both projections, so we need to add some more identities from \eqref{jedan}. We shall discuss each way of adding an identity to this system: 

If we equalize $p(x,y,y)$ with the terms of the identities \eqref{j}, we shall obtain the system:\begin{equation*}
\begin{array}{c}
p(x,y,x) \approx q(x,x,y)\approx p(x,x,y)\approx q(x,y,x)\approx p(x,y,y) 
\end{array}
\end{equation*}
both projections are still allowed, and adding more identities from \eqref{jedan} can only give us \eqref{jedan}.

If we equalize $q(y,x,x)$ with the terms of the identities \eqref{j}, we shall obtain the system:\begin{equation*}
\begin{array}{c}
p(x,y,x) \approx q(x,x,y)\approx p(x,x,y)\approx q(x,y,x)\approx q(y,x,x) 
\end{array}
\end{equation*}
which holds in a full idempotent reduct of a module over $\mathbb{Z}_5$, for we can define $p$ and $q$ both to be $2x+2y+2z$ in this reduct. As before adding identities can only give us the whole system \eqref{jedan}.

If we add the identity $p(x,y,y) \approx q(y,x,x)$ to the system \eqref{j}, we shall obtain the following system:\begin{equation*}
\left\{
\begin{array}{c}
p(x,y,x) \approx q(x,x,y)\approx p(x,x,y)\approx q(x,y,x)\\
p(x,y,y) \approx q(y,x,x) 
\end{array}\right. 
\end{equation*}
which holds in a full idempotent reduct of a module over $\mathbb{Z}_3$, for we can define $p$ and $q$ both to be $2x+y+z$ in this reduct. As before adding identities can only give us the whole system \eqref{jedan}.
\noindent By this we have finished analyzing the system \eqref{j}. By adding identities to this system we can obtain either a system that holds in some full idempotent reduct of a module over a finite ring, or the whole system \eqref{jedan}. 
\item equalizing the term $q(y,x,x)$ with the terms of the  identities \eqref{h} gives us the following:\begin{equation}
\begin{array}{c}
p(x,y,x) \approx q(x,x,y)\approx p(x,x,y)\approx q(y,x,x) 
\end{array}\label{k}
\end{equation}
obviously this system allows both projections, so we need to add some more identities from \eqref{jedan}. We shall discuss each way of adding an identity to this system:

If we equalize the term $p(x,y,y)$ with the terms of the  identities \eqref{k}, we shall obtain the following system:\begin{equation*}
\begin{array}{c}
p(x,y,x) \approx q(x,x,y)\approx p(x,x,y)\approx q(y,x,x)\approx p(x,y,y)
\end{array}
\end{equation*}
both projections are allowed, and adding identities from \eqref{jedan} can only give us \eqref{jedan}.

If we equalize the term $q(x,y,x)$ with the terms of the identities \eqref{k}, we shall obtain the following system:\begin{equation*}
\begin{array}{c}
p(x,y,x) \approx q(x,x,y)\approx p(x,x,y)\approx q(y,x,x)\approx q(x,y,x)
\end{array}
\end{equation*}
which holds in a full idempotent reduct of a module over $\mathbb{Z}_5$, for we can define $p$ and $q$ both to be $2x+2y+2z$ in this reduct. As before adding identities can only give us the whole system \eqref{jedan}.

If we add the identity $p(x,y,y) \approx q(x,y,x)$ to the system \eqref{k} we shall obtain the following: \begin{equation*}
\left\{
\begin{array}{c}
p(x,y,x) \approx q(x,x,y)\approx p(x,x,y)\approx q(y,x,x)\\
p(x,y,y) \approx q(x,y,x) 
\end{array}\right. 
\end{equation*}
which holds in a full idempotent reduct of a module over $\mathbb{Z}_3$, for we can define $p$ and $q$ to be $2x+y+z$ and $x+2y+z$ respectively in this reduct. As before adding identities can only give us the whole system \eqref{jedan}.
\noindent By this we have finished analyzing the system \eqref{k}. By adding identities to this system we can obtain either a system that holds in some full idempotent reduct of a module over a finite ring, or the whole system \eqref{jedan}. 
\item if we add the identity $p(x,y,y) \approx q(x,y,x)$ to the system \eqref{h} we shall obtain the following system:\begin{equation}
\left\{
\begin{array}{c}
p(x,y,x) \approx q(x,x,y)\approx p(x,x,y)\\
p(x,y,y) \approx q(x,y,x) 
\end{array}\right. \label{l}
\end{equation}
both $p$ and $q$ can be defined as projection maps in any algebra, so we need to add more identities from \eqref{jedan}. We shall discuss each way of adding an identity to this system:

If we equalize the terms of the identities \eqref{l}, we shall obtain the following:\begin{equation*}
\begin{array}{c}
p(x,y,x) \approx q(x,x,y)\approx p(x,x,y)\approx
p(x,y,y) \approx q(x,y,x) 
\end{array} 
\end{equation*}
still both projections are allowed and adding any more identities from \eqref{jedan} gives us \eqref{jedan}.

If we equalize $q(y,x,x)$ with the terms of the first identity of \eqref{l} we shall obtain the following: \begin{equation*}
\left\{
\begin{array}{c}
p(x,y,x) \approx q(x,x,y)\approx p(x,x,y)\approx q(y,x,x)\\
p(x,y,y) \approx q(x,y,x) 
\end{array}\right. 
\end{equation*}
This system holds in a full idempotent reduct of a module over $\mathbb{Z}_3$, for we can define $p$ and $q$ to be $2x+y+z$ and $x+2y+z$ respectively in this reduct. As before adding identities can only give us the whole system \eqref{jedan}.

If we equalize $q(y,x,x)$ with the  terms of the second identity of \eqref{l}, we shall obtain the following: \begin{equation*}
\left\{
\begin{array}{c}
p(x,y,x) \approx q(x,x,y)\approx p(x,x,y)\\
p(x,y,y) \approx q(x,y,x)\approx q(y,x,x) 
\end{array}\right. 
\end{equation*}
This system holds in a full idempotent reduct of a module over $\mathbb{Z}_4$, for we can define $p$ and $q$ to be $3x+y+z$ and $2x+2y+z$ respectively in this reduct. As before adding identities can only give us the whole system \eqref{jedan}.
\noindent By this we have finished analyzing the system \eqref{l}. By adding identities to this system we can obtain either a system that holds in some full idempotent reduct of a module over a finite ring, or the whole system \eqref{jedan}. 
\item if we add the identity $p(x,y,y) \approx q(y,x,x)$ to the system \eqref{h} we shall obtain the following system:\begin{equation}
\left\{
\begin{array}{c}
p(x,y,x) \approx q(x,x,y)\approx p(x,x,y)\\
p(x,y,y) \approx q(y,x,x) 
\end{array}\right. \label{m}
\end{equation}
This system allows both $p$ and $q$ to be projection maps in any algebra, which means we have to add some more identities from \eqref{jedan}. We shall discuss each way of adding an identity to this system:

If we equalize the terms of the identities  \eqref{m}, we shall obtain the following: \begin{equation*}
\begin{array}{c}
p(x,y,x) \approx q(x,x,y)\approx p(x,x,y)\approx
p(x,y,y) \approx q(y,x,x) 
\end{array}
\end{equation*}
Still both terms can be defined as projections, and adding a new identity from \eqref{jedan} gives us the whole system \eqref{jedan}.

If we equalize the term $q(x,y,x)$ with the the terms of the first identity of \eqref{m} we shall obtain the following:\begin{equation*}
\left\{
\begin{array}{c}
p(x,y,x) \approx q(x,x,y)\approx p(x,x,y)\approx q(x,y,x)\\
p(x,y,y) \approx q(y,x,x) 
\end{array}\right. 
\end{equation*}
This system holds in a full idempotent reduct of a module over $\mathbb{Z}_3$, for we can define both $p$ and $q$ to be $2x+y+z$ in this reduct. As before adding identities can only give us the whole system \eqref{jedan}.

If we equalize the term $q(x,y,x)$ with the terms of the second identity of \eqref{m}, we shall obtain the following:\begin{equation*}
\left\{
\begin{array}{c}
p(x,y,x) \approx q(x,x,y)\approx p(x,x,y)\\
p(x,y,y) \approx q(y,x,x)\approx q(x,y,x) 
\end{array}\right. 
\end{equation*}
This system holds in a full idempotent reduct of a module over $\mathbb{Z}_4$, for we can define $p$ and $q$ to be $3x+y+z$ and $2x+2y+z$ respectively in this reduct. As before adding identities can only give us the whole system \eqref{jedan}.
\noindent By this we have finished analyzing the system \eqref{m}. By adding identities to this system we can obtain either a system that holds in some full idempotent reduct of a module over a finite ring, or the whole system \eqref{jedan}. 
\item if we add the identity $q(x,y,x)\approx q(y,x,x)$ to the system \eqref{h} we shall obtain the following system:\begin{equation}
\left\{
\begin{array}{c}
p(x,y,x) \approx q(x,x,y)\approx p(x,x,y)\\
q(x,y,x)\approx q(y,x,x) 
\end{array}\right. \label{n}
\end{equation}
This system holds in a full idempotent reduct of a module over $\mathbb{Z}_5$, for we can define $p$ and $q$ to be $\pi_1$ and $3x+3y$ respectively in this reduct. Therefore we need to add more identities from \eqref{jedan}, so let us discuss each way of doing this:

If we equalize the terms of the identities  \eqref{n}, we shall obtain the following:\begin{equation*}
\begin{array}{c}
p(x,y,x) \approx q(x,x,y)\approx p(x,x,y)\approx
q(x,y,x)\approx q(y,x,x) 
\end{array}
\end{equation*}
This system holds in a full idempotent reduct of a module over $\mathbb{Z}_5$, for we can define both $p$ and $q$ to be $2x+2y+2z$ respectively in this reduct. Adding identities from \eqref{jedan} can only give us the whole system \eqref{jedan}.

If we equalize the term $p(x,y,y)$ with the terms of the first identity of \eqref{n}, we shall obtain the following:\begin{equation*}
\left\{
\begin{array}{c}
p(x,y,x) \approx q(x,x,y)\approx p(x,x,y)\approx p(x,y,y)\\
q(x,y,x)\approx q(y,x,x) 
\end{array}\right. 
\end{equation*}
This system holds in a full idempotent reduct of a module over $\mathbb{Z}_5$, for we can define $p$ and $q$ to be $\pi_1$ and $3x+3y$ respectively in this reduct. Adding any identities from \eqref{jedan} can only give us \eqref{jedan}.

If we equalize the term $p(x,y,y)$ with the terms of the second identity of \eqref{n}, we shall obtain the following:\begin{equation*}
\left\{
\begin{array}{c}
p(x,y,x) \approx q(x,x,y)\approx p(x,x,y)\\
q(x,y,x)\approx q(y,x,x)\approx p(x,y,y) 
\end{array}\right. 
\end{equation*}
This system holds in a full idempotent reduct of a module over $\mathbb{Z}_4$, for we can define $p$ and $q$ to be $3x+y+z$ and $2x+2y+z$ respectively in this reduct. Adding any identities from \eqref{jedan} can only give us \eqref{jedan}.
\noindent By this we have finished analyzing the system \eqref{n}. By adding identities to this system we can obtain either a system that holds in some full idempotent reduct of a module over a finite ring, or the whole system \eqref{jedan}. 
\end{itemize}
By this we have examined all the systems that include identities \eqref{h} and we have obtained no other system but \eqref{dva} implying omitting types 1 and 2. We shall now replace the second identity in \eqref{h} for the remaining one (the third one) on $p$.
\item the second identity is $p(x,y,y) \approx p(x,y,x)$, i.e. we have these identities:\begin{equation}
\begin{array}{c}
p(x,y,x) \approx q(x,x,y)\approx p(x,y,y)
\end{array} \label{o}
\end{equation}
Identities \eqref{o} obviously allow both projection maps, so more identities from the system \eqref{jedan} are needed. For this purpose we can equalize one of the terms $p(x,x,y)$, $q(x,y,x)$, $q(y,x,x)$ with the terms of \eqref{o}, or we can add a new identity to \eqref{o} (i.e. one of these $p(x,x,y) \approx q(x,y,x)$, $p(x,x,y) \approx q(y,x,x)$, $q(x,y,x) \approx q(y,x,x)$). As before we shall discuss each of these cases:
\begin{itemize}
\item if we equalize the term $p(x,x,y)$ with the terms of \eqref{o}, we shall obtain the following system:\begin{equation}
\begin{array}{c}
p(x,y,x) \approx q(x,x,y)\approx p(x,y,y)\approx p(x,x,y)
\end{array}\label{p} 
\end{equation}
Both projections are still allowed, so we need to add some more identities from \eqref{jedan}. Let us discuss each way of doing this:

If we equalize the term $q(x,y,x)$ with the terms of  \eqref{p}, we shall obtain the following:\begin{equation*}
\begin{array}{c}
p(x,y,x) \approx q(x,x,y)\approx p(x,y,y)\approx p(x,x,y)\approx q(x,y,x)
\end{array}
\end{equation*}
This system still allows both projection maps in any algebra, and adding more identities from \eqref{jedan} can only give us the whole system \eqref{jedan}.

If we equalize the term $q(y,x,x)$ with the terms of \eqref{p}, we shall obtain the following:\begin{equation*}
\begin{array}{c}
p(x,y,x) \approx q(x,x,y)\approx p(x,y,y)\approx p(x,x,y)\approx q(y,x,x)
\end{array}
\end{equation*}
For this system holds the same as for the previous.

If we add the identity $q(x,y,x)\approx q(y,x,x)$ to the system \eqref{p}, we shall obtain the following:\begin{equation*}
\left\{
\begin{array}{c}
p(x,y,x) \approx q(x,x,y)\approx p(x,y,y)\approx p(x,x,y)\\
q(x,y,x)\approx q(y,x,x)
\end{array}\right.
\end{equation*}
This system holds in a full idempotent reduct of a module over $\mathbb{Z}_5$, for we can define $p$ and $q$ to be $\pi_1$ and $3x+3y$ respectively in this reduct. Adding any identities from \eqref{jedan} can only give us \eqref{jedan}.
\noindent By this we have finished analyzing the system \eqref{p}. By adding identities to this system we can obtain either a system that holds in some full idempotent reduct of a module over a finite ring, or the whole system \eqref{jedan}. 
\item if we equalize the term $q(x,y,x)$ with the terms of \eqref{o}, we shall obtain the following system:\begin{equation}
\begin{array}{c}
p(x,y,x) \approx q(x,x,y)\approx p(x,y,y)\approx q(x,y,x)
\end{array} \label{r}
\end{equation}
This system allows both $p$ and $q$ to be defined as projection maps in any algebra, so we need to add some more identities from \eqref{jedan}. Let us discuss each way of doing this:

If we equalize the term $p(x,x,y)$ with the terms of \eqref{r}, we shall obtain the following:\begin{equation*}
\begin{array}{c}
p(x,y,x) \approx q(x,x,y)\approx p(x,y,y)\approx q(x,y,x)\approx p(x,x,y)
\end{array} 
\end{equation*}
This system also allows both projections, and adding more identities from \eqref{jedan} may only give us the whole system \eqref{jedan}.

If we equalize the term $q(y,x,x)$ with the terms of \eqref{r}, we shall obtain the following: \begin{equation*}
\begin{array}{c}
p(x,y,x) \approx q(x,x,y)\approx p(x,y,y)\approx q(x,y,x)\approx q(y,x,x)
\end{array} 
\end{equation*}
This system holds in a full idempotent reduct of a module over $\mathbb{Z}_5$, for we can define $p$ and $q$ to be $4x+2y$ and $2x+2y+2z$ respectively in this reduct. Adding any identities from \eqref{jedan} can only give us \eqref{jedan}.

If we add the identity $p(x,x,y) \approx q(y,x,x)$ to the system \eqref{r} we shall obtain the following system:\begin{equation*}
\left\{
\begin{array}{c}
p(x,y,x) \approx q(x,x,y)\approx p(x,y,y)\approx q(x,y,x)\\
p(x,x,y) \approx q(y,x,x)
\end{array}\right. 
\end{equation*}
This system holds in a full idempotent reduct of a module over $\mathbb{Z}_5$, for we can define $p$ and $q$ to be $3x+3y$ and $3y+3z$ respectively in this reduct. Adding any identities from \eqref{jedan} can only give us \eqref{jedan}.
\noindent By this we have finished analyzing the system \eqref{r}. By adding identities to this system we can obtain either a system that holds in some full idempotent reduct of a module over a finite ring, or the whole system \eqref{jedan}. 
\item if we equalize the term $q(y,x,x)$ with the terms of  \eqref{o}, we shall obtain the following system:\begin{equation}
\begin{array}{c}
p(x,y,x) \approx q(x,x,y)\approx p(x,y,y)\approx q(y,x,x)
\end{array} \label{s}
\end{equation}
This system allows both terms to be defined as projections in any algebra, so we shall add some more identities from \eqref{jedan}. Let us discuss each way of doing this: 

If we equalize the term $p(x,x,y)$ with the terms of \eqref{s}, we shall obtain the following:\begin{equation*}
\begin{array}{c}
p(x,y,x) \approx q(x,x,y)\approx p(x,y,y)\approx q(y,x,x)\approx p(x,x,y)
\end{array} 
\end{equation*}
Still both $p$ and $q$ can be projection maps, and adding new identities from \eqref{jedan} can only give us \eqref{jedan}.

If we equalize the term $q(x,y,x)$ with the terms of  \eqref{s}, we shall obtain the following:\begin{equation*}
\begin{array}{c}
p(x,y,x) \approx q(x,x,y)\approx p(x,y,y)\approx q(y,x,x)\approx q(x,y,x)
\end{array} 
\end{equation*}
This system holds in a full idempotent reduct of a module over $\mathbb{Z}_5$, for we can define $p$ and $q$ to be $4x+2y$ and $2x+2y+3z$ respectively in this reduct. Adding any identities from \eqref{jedan} can only give us \eqref{jedan}.

If we add the identity $p(x,x,y)\approx q(x,y,x)$ to the system \eqref{s}, we shall obtain the following:\begin{equation*}
\left\{
\begin{array}{c}
p(x,y,x) \approx q(x,x,y)\approx p(x,y,y)\approx q(y,x,x)\\
p(x,x,y)\approx q(x,y,x)
\end{array}\right.  
\end{equation*}
This system holds in a full idempotent reduct of a module over $\mathbb{Z}_5$, for we can define $p$ and $q$ to be $3x+3y$ and $3x+3z$ respectively in this reduct. Adding any identities from \eqref{jedan} can only give us \eqref{jedan}.
\noindent By this we have finished analyzing the system \eqref{s}. By adding identities to this system we can obtain either a system that holds in some full idempotent reduct of a module over a finite ring, or the whole system \eqref{jedan}. 
\item if we add the identity $p(x,x,y) \approx q(x,y,x)$ to the system \eqref{o} we shall obtain the following system:\begin{equation}
\left\{
\begin{array}{c}
p(x,y,x) \approx q(x,x,y)\approx p(x,y,y)\\
p(x,x,y) \approx q(x,y,x)
\end{array}\right.  \label{t}
\end{equation}
This system allows both projection maps in any algebra, so we need to add some more identities from \eqref{jedan}. Let us discuss each way of doing this:

If we equalize the terms of the identities  \eqref{t}, we shall obtain the following system:\begin{equation*}
\begin{array}{c}
p(x,y,x) \approx q(x,x,y)\approx p(x,y,y)\approx
p(x,x,y) \approx q(x,y,x)
\end{array}  
\end{equation*}
This still allows both projection maps in any algebra, and adding identities from \eqref{jedan} can only give us \eqref{jedan}.

If we equalize the term $q(y,x,x)$ with the terms of the first identity of \eqref{t}, we shall obtain the following:\begin{equation*}
\left\{
\begin{array}{c}
p(x,y,x) \approx q(x,x,y)\approx p(x,y,y)\approx q(y,x,x) \\
p(x,x,y) \approx q(x,y,x)
\end{array}\right. 
\end{equation*}
This system holds in a full idempotent reduct of a module over $\mathbb{Z}_5$, for we can define $p$ and $q$ to be $3x+3y$ and $3x+3z$ respectively in this reduct. Adding any identities from \eqref{jedan} can only give us \eqref{jedan}.

If we equalize the term $q(y,x,x)$ with the terms of the second identity of \eqref{t}, we shall obtain the following:\begin{equation*}
\left\{
\begin{array}{c}
p(x,y,x) \approx q(x,x,y)\approx p(x,y,y)\\
p(x,x,y) \approx q(x,y,x)\approx q(y,x,x)
\end{array}\right.  
\end{equation*}
This system allows both $p$  and $q$ to be defined as projection maps in any algebra ($\pi_2$ and $\pi_3$ respectively), and adding an identity from \eqref{jedan} can only give us \eqref{jedan}. 
\noindent By this we have finished analyzing the system \eqref{t}. By adding identities to this system we can obtain either a system that holds in some full idempotent reduct of a module over a finite ring, or the whole system \eqref{jedan}. 
\item if we add the identity $p(x,x,y) \approx q(y,x,x)$ to the system \eqref{o} we shall obtain the following system:\begin{equation}
\left\{
\begin{array}{c}
p(x,y,x) \approx q(x,x,y)\approx p(x,y,y)\\
p(x,x,y) \approx q(y,x,x)
\end{array} \right. \label{u}
\end{equation}
This system allows both projection maps in any algebra, so we need to add more identities from \eqref{jedan}. Let us discuss each way of doing this:

If we equalize the terms of the identities \eqref{u}, we shall obtain the following system:\begin{equation*}
\left\{
\begin{array}{c}
p(x,y,x) \approx q(x,x,y)\approx p(x,y,y)\approx
p(x,x,y) \approx q(y,x,x)
\end{array} \right. 
\end{equation*}
Still both terms can be defined as projection maps in any algebra, and adding identities from \eqref{jedan} can only give us the whole system \eqref{jedan}.

If we equalize the term $q(x,y,x)$ with the terms of the first identity of \eqref{u}, we shall obtain the following system:\begin{equation*}
\left\{
\begin{array}{c}
p(x,y,x) \approx q(x,x,y)\approx p(x,y,y)\approx q(x,y,x)\\
p(x,x,y) \approx q(y,x,x)
\end{array} \right. 
\end{equation*}
This system holds in a full idempotent reduct of a module over $\mathbb{Z}_5$, for we can define $p$ and $q$ to be $3x+3y$ and $3y+3z$ respectively in this reduct. Adding any identities from \eqref{jedan} can only give us \eqref{jedan}.

If we equalize the term $q(x,y,x)$ with the terms of the second identity of \eqref{u}, we shall obtain the following system:\begin{equation*}
\left\{
\begin{array}{c}
p(x,y,x) \approx q(x,x,y)\approx p(x,y,y)\\
p(x,x,y) \approx q(y,x,x)\approx q(x,y,x)
\end{array} \right. 
\end{equation*}
This system allows both terms to be defined as projection maps ($p$ and $q$ being $\pi_2$ and $\pi_3$ respectively),and adding any more identities from \eqref{jedan} can only give us the whole system \eqref{jedan}.
\noindent By this we have finished analyzing the system \eqref{u}. By adding identities to this system we can obtain either a system that holds in some full idempotent reduct of a module over a finite ring, or the whole system \eqref{jedan}.  

\item if we add the identity $q(x,y,x) \approx q(y,x,x)$ to the system \eqref{o}, we shall obtain the following:\begin{equation}
\left\{
\begin{array}{c}
p(x,y,x) \approx q(x,x,y)\approx p(x,y,y)\\
q(x,y,x) \approx q(y,x,x)
\end{array} \right.   \label{v}
\end{equation}
This system holds in a full idempotent reduct of a module over $\mathbb{Z}_5$, for we can define $p$ and $q$ to be $\pi_1$ and $3x+3y$ respectively in this reduct. We need to add more identities from \eqref{jedan}. Let us discuss each way of doing this:

If we equalize the terms of the identities  \eqref{v}, we shall obtain the following:\begin{equation*}
\begin{array}{c}
p(x,y,x) \approx q(x,x,y)\approx p(x,y,y)\approx
q(x,y,x) \approx q(y,x,x)
\end{array}  
\end{equation*}
This system holds in a full idempotent reduct of a module over $\mathbb{Z}_5$, for we can define $p$ and $q$ to be $4x+2y$ and $2x+2y+2z$ respectively in this reduct. Adding identities from \eqref{jedan} gives us nothing but \eqref{jedan}. 

If we equalize the term $p(x,x,y)$ with the terms of the first identity of \eqref{v}, we shall obtain the following: \begin{equation*}
\left\{
\begin{array}{c}
p(x,y,x) \approx q(x,x,y)\approx p(x,y,y)\approx p(x,x,y)\\
q(x,y,x) \approx q(y,x,x)
\end{array} \right.   
\end{equation*}
This system holds in a full idempotent reduct of a module over $\mathbb{Z}_5$, for we can define $p$ and $q$ to be $\pi_1$ and $3x+3y$ respectively in this reduct. Adding identities from \eqref{jedan} gives us nothing but \eqref{jedan}. 

If we equalize the term $p(x,x,y)$ with the terms of the second identity of \eqref{v}, we shall obtain the following:\begin{equation*}
\left\{
\begin{array}{c}
p(x,y,x) \approx q(x,x,y)\approx p(x,y,y)\\
q(x,y,x) \approx q(y,x,x)\approx p(x,x,y)
\end{array} \right.   
\end{equation*}
This system allows both terms to be defined as projection maps ($p$ and $q$ being $\pi_2$ and $\pi_3$ respectively), and adding any more identities from \eqref{jedan} can only give us the whole system \eqref{jedan}.
\noindent By this we have finished analyzing the system \eqref{v}. By adding identities to this system we can obtain either a system that holds in some full idempotent reduct of a module over a finite ring, or the whole system \eqref{jedan}.  
\end{itemize}
By this we have examined all the systems that include identities \eqref{o} and we have obtained no other systems but \eqref{dva} implying omitting types 1 and 2. 
\end{itemize}
\noindent Let us notice here that, by this point, we have examined all the proper subsets of \eqref{jedan} having the first identity $p(x,y,x) \approx q(x,x,y)$ and the second identity only on $p$. 

\vspace{0.3 cm}

\noindent Before we start analyzing systems having the first identity $p(x,y,x) \approx q(x,x,y)$ and the second on $p$ and $q$ let us notice the following facts:
\begin{enumerate}
\item If the left hand side of the second identity is the same as one of the first (i.e. $p(x,y,x)$), we could substitute the second identity for an identity only on $q$ (simply by equalizing right hand sides), and that would make for an equivalent system. As said before, the term $p$ has to occur in at least one more identity besides the first one in the system considered (i.e. a proper subset of \eqref{jedan} possibly describing omitting types 1 and 2), which means there has to be an identity including $p(x,x,y)$ or $p(x,y,y)$ in the system. That identity may be considered the second, instead of the identity with $p(x,y,x)$ on the left, so we may assume that left hand side of the second identity is not $p(x,y,x)$.
\item If the right hand side of the second identity is the same as one of the first (i.e. $q(x,x,y)$), we could substitute the second identity with the one only on $p$ (by equalizing left hand sides), obtaining an equivalent system. All such cases are already examined, so we may assume that the second identity does not have $q(x,x,y)$ on the right.
\item According to the previous two items, it is sufficient to examine second identities on $p$ and $q$ that include neither $p(x,y,x)$ on the left, nor $q(x,x,y)$ on the right. 
\end{enumerate}
\begin{itemize}
\item the second identity is $p(x,x,y)\approx q(x,y,x)$, i.e. we have these two identities:\begin{equation}
\left\{
\begin{array}{c}
p(x,y,x) \approx q(x,x,y)\\
p(x,x,y)\approx q(x,y,x)
\end{array} \right.\label{w}   
\end{equation}
Identities \eqref{w} obviously allow both projection maps, so we need to add more identities from \eqref{jedan}. This can be done in the following ways: we can equalize the terms of the existing identities, or equalize each of the terms $p(x,y,y)$, $q(y,x,x)$ with the terms of each of the two identities, or add the identity $p(x,y,y) \approx q(y,x,x)$ to the system \eqref{w}. Let us discuss all these cases:
\begin{itemize}
\item if we equalize the terms of the  identities \eqref{w}, we shall obtain the following:\begin{equation}
\begin{array}{c}
p(x,y,x) \approx q(x,x,y)\approx
p(x,x,y)\approx q(x,y,x)\label{y}
\end{array}    
\end{equation}
Both projection maps are still allowed by his system, so we need to add some more identities from \eqref{jedan}. Let us discuss each way of doing this:

If we equalize the term $p(x,y,y)$ with the terms of \eqref{y}, we shall obtain the following:\begin{equation*}
\begin{array}{c}
p(x,y,x) \approx q(x,x,y)\approx
p(x,x,y)\approx q(x,y,x)\approx p(x,y,y)
\end{array}    
\end{equation*}
This system still allows both projection maps in any algebra, and adding any new identities from \eqref{jedan} can only give us the whole system \eqref{jedan}. 

If we equalize the term $q(y,x,x)$ with the terms of  \eqref{y}, we shall obtain the following:\begin{equation*}
\begin{array}{c}
p(x,y,x) \approx q(x,x,y)\approx
p(x,x,y)\approx q(x,y,x)\approx q(y,x,x)
\end{array}    
\end{equation*}
This system holds in a full idempotent reduct of a module over $\mathbb{Z}_5$, for we can define both $p$ and $q$ to be $2x+2y+2z$ in this reduct. Adding identities from \eqref{jedan} can only give us \eqref{jedan}.

If we add the identity $p(x,y,y) \approx q(y,x,x)$ to the system \eqref{y} we shall obtain the following:\begin{equation*}
\left\{
\begin{array}{c}
p(x,y,x) \approx q(x,x,y)\approx
p(x,x,y)\approx q(x,y,x)\\
p(x,y,y) \approx q(y,x,x)
\end{array} \right.    
\end{equation*}
This system holds in a full idempotent reduct of a module over $\mathbb{Z}_3$, for we can define both $p$ and $q$ to be $2x+y+z$ in this reduct. Adding identities from \eqref{jedan} can only give us \eqref{jedan}.
\noindent By this we have finished analyzing the system \eqref{y}. By adding identities to this system we can obtain either a system that holds in some full idempotent reduct of a module over a finite ring, or the whole system \eqref{jedan}.  
\item if we equalize the term $p(x,y,y)$ with the terms of the first identity of \eqref{w}, we shall obtain the following system:\begin{equation}
\left\{
\begin{array}{c}
p(x,y,x) \approx q(x,x,y)\approx p(x,y,y)\\
p(x,x,y)\approx q(x,y,x)
\end{array} \right.\label{x}   
\end{equation}
This system allows both terms to be defined as projections in any algebra, so we need to add more identities from \eqref{jedan}. Let us discuss each way of doing this:

If we equalize the terms of the identities  \eqref{x}, we shall obtain the following:\begin{equation*}
\begin{array}{c}
p(x,y,x) \approx q(x,x,y)\approx p(x,y,y)\approx
p(x,x,y)\approx q(x,y,x)
\end{array}   
\end{equation*}
This still allows both projections, and adding identities from \eqref{jedan} gives us nothing but \eqref{jedan}.

If we equalize the term $q(y,x,x)$ with the terms of the first identity of \eqref{x}, we shall obtain the following:\begin{equation*}
\left\{
\begin{array}{c}
p(x,y,x) \approx q(x,x,y)\approx p(x,y,y)\approx q(y,x,x)\\
p(x,x,y)\approx q(x,y,x)
\end{array} \right.  
\end{equation*}
This system holds in a full idempotent reduct of a module over $\mathbb{Z}_5$, for we can define $p$ and $q$ to be $3x+3y$ and $3x+3z$ in this reduct. Adding identities from \eqref{jedan} can only give us \eqref{jedan}.

If we equalize the term $q(y,x,x)$ with the terms of the second identity of \eqref{x}, we shall obtain the following:\begin{equation*}
\left\{
\begin{array}{c}
p(x,y,x) \approx q(x,x,y)\approx p(x,y,y)\\
p(x,x,y)\approx q(x,y,x)\approx q(y,x,x)
\end{array} \right.  
\end{equation*}
This system allows both $p$ and $q$ to be defined as projection maps in any algebra, and adding identities from \eqref{jedan} gives us only the whole system \eqref{jedan}.
\noindent By this we have finished analyzing the system \eqref{x}. By adding identities to this system we can obtain either a system that holds in some full idempotent reduct of a module over a finite ring, or the whole system \eqref{jedan}.  
\item if we equalize the term $p(x,y,y)$ with the terms of the  second identity of \eqref{w}, we shall obtain the following system:\begin{equation}
\left\{
\begin{array}{c}
p(x,y,x) \approx q(x,x,y)\\
p(x,x,y)\approx q(x,y,x)\approx p(x,y,y)
\end{array} \right.\label{q}   
\end{equation}
This system allows both projections, so we need to add more identities from \eqref{jedan}. Let us discuss each way of doing this:

If we equalize the terms of the identities  \eqref{q}, we shall obtain the following:\begin{equation*}
\begin{array}{c}
p(x,y,x) \approx q(x,x,y)\approx
p(x,x,y)\approx q(x,y,x)\approx p(x,y,y)
\end{array}  
\end{equation*}
Both projections are still allowed, and adding identities from \eqref{jedan} can only give us \eqref{jedan}.

If we equalize the term $q(y,x,x)$ with the terms of the first identity of \eqref{q}, we shall obtain the following system: \begin{equation*}
\left\{
\begin{array}{c}
p(x,y,x) \approx q(x,x,y)\approx q(y,x,x)\\
p(x,x,y)\approx q(x,y,x)\approx p(x,y,y)
\end{array} \right.  
\end{equation*}
This allows both projection maps in any algebra ($p$, $q$ being $\pi_3$, $\pi_2$ respectively), and adding identities from \eqref{jedan} can only give us \eqref{jedan}.

If we equalize the term $q(y,x,x)$ with the terms of the second identity of \eqref{q}, we shall obtain the following system:\begin{equation*}
\left\{
\begin{array}{c}
p(x,y,x) \approx q(x,x,y)\\
p(x,x,y)\approx q(x,y,x)\approx p(x,y,y)\approx q(y,x,x)
\end{array} \right.  
\end{equation*}
This system holds in a full idempotent reduct of a module over $\mathbb{Z}_5$, for we can define $p$ and $q$ to be $3x+3z$ and $3x+3y$ respectively in this reduct. Adding identities from \eqref{jedan} can only give us \eqref{jedan}.
\noindent By this we have finished analyzing the system \eqref{q}. By adding identities to this system we can obtain either a system that holds in some full idempotent reduct of a module over a finite ring, or the whole system \eqref{jedan}.
\item if we equalize the term $q(y,x,x)$ with the terms of the first identity of \eqref{w}, we shall obtain the following system:\begin{equation}
\left\{
\begin{array}{c}
p(x,y,x) \approx q(x,x,y)\approx q(y,x,x)\\
p(x,x,y)\approx q(x,y,x)
\end{array} \right.\label{z}   
\end{equation}
This system allows both projection maps in any algebra, so we need to add more identities from \eqref{jedan}. Let us discuss each way of doing this: 

If we equalize the terms of the identities  \eqref{z}, we shall obtain the following:\begin{equation*}
\begin{array}{c}
p(x,y,x) \approx q(x,x,y)\approx q(y,x,x)\approx
p(x,x,y)\approx q(x,y,x)
\end{array}   
\end{equation*}
This system holds in a full idempotent reduct of a module over $\mathbb{Z}_5$, for we can define both $p$ and $q$ to be $2x+2y+2z$ in this reduct. Adding identities from \eqref{jedan} can only give us \eqref{jedan}.

If we equalize the term $p(x,y,y)$ with the terms of the  first identity of \eqref{z}, we shall obtain the following: \begin{equation*}
\left\{
\begin{array}{c}
p(x,y,x) \approx q(x,x,y)\approx q(y,x,x) \approx p(x,y,y)\\
p(x,x,y)\approx q(x,y,x)
\end{array} \right. 
\end{equation*}
This system holds in a full idempotent reduct of a module over $\mathbb{Z}_5$, for we can define  $p$ and $q$ to be $3x+3y$, $3x+3z$ respectively in this reduct. Adding identities from \eqref{jedan} can only give us \eqref{jedan}.

If we equalize the term $p(x,y,y)$ with the terms of the  second identity of \eqref{z}, we shall obtain the following:\begin{equation*}
\left\{
\begin{array}{c}
p(x,y,x) \approx q(x,x,y)\approx q(y,x,x)\\
p(x,x,y)\approx q(x,y,x) \approx p(x,y,y)
\end{array} \right.   
\end{equation*}
This system allows both $p$ and $q$ to be defined as projection maps in any algebra ($p$, $q$ being $\pi_3$, $\pi_2$ respectively), and adding identities from \eqref{jedan} can only give us \eqref{jedan}.
\noindent By this we have finished analyzing the system \eqref{z}. By adding identities to this system we can obtain either a system that holds in some full idempotent reduct of a module over a finite ring, or the whole system \eqref{jedan}.
\item if we equalize the term $q(y,x,x)$ with the terms of the second identity of \eqref{w}, we shall obtain the following system:\begin{equation}
\left\{
\begin{array}{c}
p(x,y,x) \approx q(x,x,y)\\
p(x,x,y)\approx q(x,y,x)\approx q(y,x,x)
\end{array} \right.\label{aa}   
\end{equation}
This system allows both $p$ and $q$ to be defined as projection maps in any algebra, so we need to add some more identities from \eqref{jedan}. Let us discuss each way of doing this:  

If we equalize the terms of the identities \eqref{aa}, we shall obtain the following:\begin{equation*}
\begin{array}{c}
p(x,y,x) \approx q(x,x,y) \approx 
p(x,x,y)\approx q(x,y,x)\approx q(y,x,x)
\end{array} 
\end{equation*}
This system holds in a full idempotent reduct of a module over $\mathbb{Z}_5$, for we can define  both $p$ and $q$ to be $2x+2y+2z$ in this reduct. Adding identities from \eqref{jedan} can only give us \eqref{jedan}.

If we equalize the term $p(x,y,y)$ with the terms of the first identity of  \eqref{aa}, we shall obtain the following system:\begin{equation*}
\left\{
\begin{array}{c}
p(x,y,x) \approx q(x,x,y) \approx p(x,y,y)\\
p(x,x,y)\approx q(x,y,x)\approx q(y,x,x)
\end{array} \right.  
\end{equation*}
This system allows both projections, and adding identities from \eqref{jedan} can only give us \eqref{jedan}. 

If we equalize the term $p(x,y,y)$ with the terms of the second identity of  \eqref{aa}, we shall obtain the following system:\begin{equation*}
\left\{
\begin{array}{c}
p(x,y,x) \approx q(x,x,y) \\
p(x,x,y)\approx q(x,y,x)\approx q(y,x,x) \approx p(x,y,y)
\end{array} \right.  
\end{equation*}
This system holds in a full idempotent reduct of a module over $\mathbb{Z}_5$, for we can define  $p$ and $q$ to be $3x+3z$, $3x+3y$ respectively in this reduct. Adding identities from \eqref{jedan} can only give us \eqref{jedan}.
\noindent By this we have finished analyzing the system \eqref{aa}. By adding identities to this system we can obtain either a system that holds in some full idempotent reduct of a module over a finite ring, or the whole system \eqref{jedan}.
\item if we add the identity $p(x,y,y) \approx q(y,x,x)$ to the system \eqref{w} we shall obtain the following system:\begin{equation}
\left\{
\begin{array}{c}
p(x,y,x) \approx q(x,x,y)\\
p(x,x,y)\approx q(x,y,x)\\
p(x,y,y) \approx q(y,x,x)
\end{array} \right.\label{ab}   
\end{equation}
This system holds in a full idempotent reduct of a module over $\mathbb{Z}_5$, for we can define  $p$ and $q$ to be $3x+3z$, $3x+3y$ respectively in this reduct. We need to add more identities from \eqref{jedan}, but in this case that can only mean equalizing terms of any two of the three identities in the system. Let us discuss each way of doing this: 

If we equalize the  terms of the first and the second identity of \eqref{ab}, we shall obtain the following system:\begin{equation*}
\left\{
\begin{array}{c}
p(x,y,x) \approx q(x,x,y)\approx
p(x,x,y)\approx q(x,y,x)\\
p(x,y,y) \approx q(y,x,x)
\end{array} \right.   
\end{equation*}
This system holds in a full idempotent reduct of a module over $\mathbb{Z}_3$, for we can define  both $p$ and $q$ to be $2x+y+z$ in this reduct. Adding identities from \eqref{jedan} gives us nothing but the whole system \eqref{jedan}. 

If we equalize the terms of the second and the third identity of \eqref{ab}, we shall obtain the following system:\begin{equation*}
\left\{
\begin{array}{c}
p(x,y,x) \approx q(x,x,y)\\
p(x,x,y)\approx q(x,y,x)\approx
p(x,y,y) \approx q(y,x,x)
\end{array} \right. 
\end{equation*}
This system holds in a full idempotent reduct of a module over $\mathbb{Z}_5$, for we can define  $p$ and $q$ to be $3x+3z$, $3x+3y$ respectively in this reduct. Adding identities from \eqref{jedan} gives us nothing but the whole system \eqref{jedan}. 

If we equalize the terms of the first and the third identity of \eqref{ab}, we shall obtain the following system:\begin{equation*}
\left\{
\begin{array}{c}
p(x,y,x) \approx q(x,x,y)\approx
p(x,y,y) \approx q(y,x,x)\\
p(x,x,y)\approx q(x,y,x)
\end{array} \right.  
\end{equation*}
This system holds in a full idempotent reduct of a module over $\mathbb{Z}_5$, for we can define  $p$ and $q$ to be $3x+3y$, $3x+3z$ respectively in this reduct. Adding identities from \eqref{jedan} gives us nothing but the whole system \eqref{jedan}.
\noindent By this we have finished analyzing the system \eqref{ab}. By adding identities to this system we can obtain either a system that holds in some full idempotent reduct of a module over a finite ring, or the whole system \eqref{jedan}. 
\end{itemize}
\item the second identity is $p(x,x,y)\approx q(y,x,x)$, i.e.we have these two identities:\begin{equation}
\left\{
\begin{array}{c}
p(x,y,x) \approx q(x,x,y)\\
p(x,x,y)\approx q(y,x,x)
\end{array} \right.\label{ac}   
\end{equation}
We shall not examine subsets of \eqref{jedan} including \eqref{ac}, but instead notice the following: if there is a subset $\sigma_1$ of \eqref{jedan} that includes \eqref{ac} and implies omitting types 1 and 2, then there is a subset $\sigma_2$ of \eqref{jedan} that includes \eqref{w} and implies omitting types 1 and 2 ($\sigma_2$ is obtained simply  by substituting $q(y,x,z)$ for $q(x,y,z)$ in $\sigma_1$). We have already proven that $\sigma_2$ does not exist, which means $\sigma_1$  also does not exist, so there is no need to discuss this case at all. 
\item the second identity is $p(x,y,y)\approx q(x,y,x)$, i.e.we have these two identities:
\begin{equation}
\left\{
\begin{array}{c}
p(x,y,x) \approx q(x,x,y)\\
p(x,y,y)\approx q(x,y,x)
\end{array} \right.\label{ad}   
\end{equation}
This system obviously allows both $p$ and $q$ to be defined as projection maps in any algebra, so we need to add some more identities from \eqref{jedan}. This can be done in several ways: we can equalize the terms of the identities \eqref{ad}, or equalize each of the terms $p(x,x,y)$, $q(y,x,x)$ with  the terms in each of the identities, or add the identity $p(x,x,y)\approx q(y,x,x)$ to \eqref{ad}. Let us discuss each of these cases:
\begin{itemize}
\item if we equalize the terms of the identities \eqref{ad}, we shall obtain the following:\begin{equation}
\begin{array}{c}
p(x,y,x) \approx q(x,x,y)\approx
p(x,y,y)\approx q(x,y,x)
\end{array} \label{ae}   
\end{equation}
This system still allows both projections, so we must add more identities from \eqref{jedan}. Let us discuss each way of doing this: 

If we equalize the term $p(x,x,y)$ with the terms of  \eqref{ae}, we shall obtain the following:\begin{equation*}
\begin{array}{c}
p(x,y,x) \approx q(x,x,y)\approx
p(x,y,y)\approx q(x,y,x)\approx p(x,x,y)
\end{array}   
\end{equation*}
This still allows both terms to be defined as projections, and adding more identities from \eqref{jedan} can only give us \eqref{jedan}.

If we equalize the term $q(y,x,x)$ with  the terms of \eqref{ae}, we shall obtain the following:\begin{equation*}
\begin{array}{c}
p(x,y,x) \approx q(x,x,y)\approx
p(x,y,y)\approx q(x,y,x) \approx q(y,x,x)
\end{array}  
\end{equation*}
This system holds in a full idempotent reduct of a module over $\mathbb{Z}_5$, for we can define  $p$ and $q$ to be $4x+2y$, $2x+2y+2z$ respectively in this reduct. Adding identities from \eqref{jedan} gives us nothing but the whole system \eqref{jedan}.

If we add the identity $p(x,x,y)\approx q(y,x,x)$ to the system \eqref{ae} we shall obtain the following system:\begin{equation*}
\left\{
\begin{array}{c}
p(x,y,x) \approx q(x,x,y)\approx
p(x,y,y)\approx q(x,y,x)\\
p(x,x,y)\approx q(y,x,x)
\end{array}\right.   
\end{equation*}
This system holds in a full idempotent reduct of a module over $\mathbb{Z}_5$, for we can define  $p$ and $q$ to be $3x+3y$, $3y+3z$ respectively in this reduct. Adding identities from \eqref{jedan} gives us nothing but the whole system \eqref{jedan}.
\noindent By this we have finished analyzing the system \eqref{ae}. By adding identities to this system we can obtain either a system that holds in some full idempotent reduct of a module over a finite ring, or the whole system \eqref{jedan}.
\item if we equalize the term $p(x,x,y)$ with the terms in the first identity of \eqref{ad}, we shall obtain the following:\begin{equation}
\left\{
\begin{array}{c}
p(x,y,x) \approx q(x,x,y)\approx p(x,x,y)\\
p(x,y,y)\approx q(x,y,x)
\end{array} \right.\label{af}   
\end{equation}
This system allows both $p$ and $q$ to be defined as projection maps in any algebra, so we need to add more identities from \eqref{jedan}. Let us discuss each way of doing this: 

If we equalize the terms of the  identities  \eqref{af}, we shall obtain the following:\begin{equation*}
\begin{array}{c}
p(x,y,x) \approx q(x,x,y)\approx p(x,x,y)\approx 
p(x,y,y)\approx q(x,y,x)
\end{array}  
\end{equation*}
This system still allows both projections, and adding identities from \eqref{jedan} can only give us \eqref{jedan}.

If we equalize the term $q(y,x,x)$ with the terms of the  first identity of \eqref{af}, we shall obtain the following:\begin{equation*}
\left\{
\begin{array}{c}
p(x,y,x) \approx q(x,x,y)\approx p(x,x,y)\approx q(y,x,x)\\
p(x,y,y)\approx q(x,y,x)
\end{array} \right.   
\end{equation*}
This system holds in a full idempotent reduct of a module over $\mathbb{Z}_3$, for we can define  $p$ and $q$ to be $2x+y+z$, $x+2y+z$ respectively in this reduct. Adding identities from \eqref{jedan} gives us nothing but the whole system \eqref{jedan}.

If we equalize the term $q(y,x,x)$ with the terms of the  second identity of \eqref{af}, we shall obtain the following:\begin{equation*}
\left\{
\begin{array}{c}
p(x,y,x) \approx q(x,x,y)\approx p(x,x,y)\\
p(x,y,y)\approx q(x,y,x)\approx q(y,x,x)
\end{array} \right.  
\end{equation*}
This system holds in a full idempotent reduct of a module over $\mathbb{Z}_4$, for we can define  $p$ and $q$ to be $3x+y+z$, $2x+2y+z$ respectively in this reduct. Adding identities from \eqref{jedan} gives us nothing but the whole system \eqref{jedan}.
\noindent By this we have finished analyzing the system \eqref{af}. By adding identities to this system we can obtain either a system that holds in some full idempotent reduct of a module over a finite ring, or the whole system \eqref{jedan}.
\item if we equalize the term $p(x,x,y)$ with the terms of the second identity of \eqref{ad}, we shall obtain the following:\begin{equation}
\left\{
\begin{array}{c}
p(x,y,x) \approx q(x,x,y)\\
p(x,y,y)\approx q(x,y,x)\approx p(x,x,y)
\end{array} \right.\label{ag}
\end{equation}
This system allows both $p$ and $q$ to be defined as projection maps in any algebra, so we need to add more identities from \eqref{jedan}. Let us discuss each way of doing this:  

If we equalize the terms of the  identities  \eqref{ag}, we shall obtain the following system:\begin{equation*}
\begin{array}{c}
p(x,y,x) \approx q(x,x,y)\approx
p(x,y,y)\approx q(x,y,x)\approx p(x,x,y)
\end{array} 
\end{equation*}
This still allows both projections, and adding identities from \eqref{jedan} can only give us the whole system \eqref{jedan}.

If we equalize the term $q(y,x,x)$ with the terms of the first identity of \eqref{ag}, we shall obtain the following: \begin{equation*}
\left\{
\begin{array}{c}
p(x,y,x) \approx q(x,x,y)\approx q(y,x,x)\\
p(x,y,y)\approx q(x,y,x)\approx p(x,x,y)
\end{array} \right.
\end{equation*}
This system allows both projection maps ($p$ and $q$ being $\pi_3$ and $\pi_2$ respectively), and adding identities from \eqref{jedan} can only give us the whole system \eqref{jedan}.

If we equalize the term $q(y,x,x)$ with the terms of the second identity of \eqref{ag}, we shall obtain the following: \begin{equation*}
\left\{
\begin{array}{c}
p(x,y,x) \approx q(x,x,y)\\
p(x,y,y)\approx q(x,y,x)\approx p(x,x,y)\approx q(y,x,x)
\end{array} \right.
\end{equation*}
This system holds in a full idempotent reduct of a module over $\mathbb{Z}_5$, for we can define  $p$ and $q$ to be $3x+3z$, $3x+3y$ respectively in this reduct. Adding identities from \eqref{jedan} gives us nothing but the whole system \eqref{jedan}.
\noindent By this we have finished analyzing the system \eqref{ag}. By adding identities to this system we can obtain either a system that holds in some full idempotent reduct of a module over a finite ring, or the whole system \eqref{jedan}.
\item if we equalize the term $q(y,x,x)$ with the terms of the  first identity of \eqref{ad}, we shall obtain the following:\begin{equation}
\left\{
\begin{array}{c}
p(x,y,x) \approx q(x,x,y) \approx q(y,x,x)\\
p(x,y,y)\approx q(x,y,x)
\end{array} \right.\label{ah}   
\end{equation}
This system allows both projection maps, so we need to add some more identities from \eqref{jedan}. Let us discuss each way of doing this:  

If we equalize the terms of the  identities \eqref{ah}, we shall obtain the following system:\begin{equation*}
\begin{array}{c}
p(x,y,x) \approx q(x,x,y) \approx q(y,x,x)\approx
p(x,y,y)\approx q(x,y,x)
\end{array}   
\end{equation*}
This system holds in a full idempotent reduct of a module over $\mathbb{Z}_5$, for we can define  $p$ and $q$ to be $4x+2y$, $2x+2y+2z$ respectively in this reduct. Adding identities from \eqref{jedan} gives us nothing but the whole system \eqref{jedan}.

If we equalize the term $p(x,x,y)$ with the terms of the first identity of the system \eqref{ah}, we shall obtain the following system:\begin{equation*}
\left\{
\begin{array}{c}
p(x,y,x) \approx q(x,x,y) \approx q(y,x,x)\approx $p(x,x,y)$ \\
p(x,y,y)\approx q(x,y,x)
\end{array} \right.   
\end{equation*}
This system holds in a full idempotent reduct of a module over $\mathbb{Z}_3$, for we can define  $p$ and $q$ to be $2x+y+z$, $x+2y+z$ respectively in this reduct. Adding identities from \eqref{jedan} gives us nothing but the whole system \eqref{jedan}.

If we equalize the term $p(x,x,y)$ with the terms of the  second identity of the system \eqref{ah}, we shall obtain the following system:\begin{equation*}
\left\{
\begin{array}{c}
p(x,y,x) \approx q(x,x,y) \approx q(y,x,x)\\
p(x,y,y)\approx q(x,y,x)\approx $p(x,x,y)$
\end{array} \right. 
\end{equation*}
This system allows both projection maps, and adding identities from \eqref{jedan} gives us nothing but the whole system \eqref{jedan}.
\noindent By this we have finished analyzing the system \eqref{ah}. By adding identities to this system we can obtain either a system that holds in some full idempotent reduct of a module over a finite ring, or the whole system \eqref{jedan}.
\item if we equalize the term $q(y,x,x)$ with the terms of the second identity of \eqref{ad}, we shall obtain the following:\begin{equation}
\left\{
\begin{array}{c}
p(x,y,x) \approx q(x,x,y) \\
p(x,y,y)\approx q(x,y,x)\approx q(y,x,x)
\end{array} \right.\label{ai}   
\end{equation}
This system holds in a full idempotent reduct of a module over $\mathbb{Z}_5$, for we can define  $p$ and $q$ to be $4x+2y$, $2x+2y+2z$ respectively in this reduct. We need to add some more identities from \eqref{jedan}. Let us discuss each way of doing this:

If we equalize the terms of the identities  \eqref{ai}, we shall obtain the following:\begin{equation*}
\begin{array}{c}
p(x,y,x) \approx q(x,x,y) \approx
p(x,y,y)\approx q(x,y,x)\approx q(y,x,x)
\end{array}  
\end{equation*}
This still holds in a full idempotent reduct of a module over $\mathbb{Z}_5$, for we can define  $p$ and $q$ to be $4x+2y$, $2x+2y+2z$ respectively in this reduct, as in the previous case. Adding identities from the system \eqref{jedan} can only give us the whole system \eqref{jedan}. 

If we equalize the term $p(x,x,y)$ with the terms of the  first identity of the system \eqref{ai}, we shall obtain the following:\begin{equation*}
\left\{
\begin{array}{c}
p(x,y,x) \approx q(x,x,y)\approx p(x,x,y) \\
p(x,y,y)\approx q(x,y,x)\approx q(y,x,x)
\end{array} \right.
\end{equation*}
This system holds in a full idempotent reduct of a module over $\mathbb{Z}_4$, for we can define  $p$ and $q$ to be $3x+y+z$, $2x+2y+z$ respectively in this reduct. Adding identities from the system \eqref{jedan} can only give us the whole system \eqref{jedan}. 

If we equalize the term $p(x,x,y)$ with the terms of the second identity of the system \eqref{ai}, we shall obtain the following:\begin{equation*}
\left\{
\begin{array}{c}
p(x,y,x) \approx q(x,x,y) \\
p(x,y,y)\approx q(x,y,x)\approx q(y,x,x)\approx p(x,x,y)
\end{array} \right.
\end{equation*}
This system holds in a full idempotent reduct of a module over $\mathbb{Z}_5$, for we can define  $p$ and $q$ to be $3x+3z$, $3x+3y$ respectively in this reduct. Adding identities from the system \eqref{jedan} can only give us the whole system \eqref{jedan}.
\noindent By this we have finished analyzing the system \eqref{ai}. By adding identities to this system we can obtain either a system that holds in some full idempotent reduct of a module over a finite ring, or the whole system \eqref{jedan}. 
\item if we add the identity $p(x,x,y)\approx q(y,x,x)$ to the system \eqref{ad} we shall obtain the following system:\begin{equation}
\left\{
\begin{array}{c}
p(x,y,x) \approx q(x,x,y)\\
p(x,y,y)\approx q(x,y,x)\\
p(x,x,y)\approx q(y,x,x)
\end{array} \right.\label{aj}   
\end{equation}
This system holds in a full idempotent reduct of a module over $\mathbb{Z}_5$, for we can define  $p$ and $q$ to be $3x+3y$, $3y+3z$ respectively in this reduct. We need to add more identities from the system \eqref{jedan}, but in this case that can only be done by equalizing terms of any two of the three identities given above. Let us discuss each way of doing this: 

If we equalize the terms of the first and the second identity of the system \eqref{aj}, we shall obtain the following:\begin{equation*}
\left\{
\begin{array}{c}
p(x,y,x) \approx q(x,x,y)\approx
p(x,y,y)\approx q(x,y,x)\\
p(x,x,y)\approx q(y,x,x)
\end{array} \right.  
\end{equation*}
This still holds in a full idempotent reduct of a module over $\mathbb{Z}_5$, for we can define  $p$ and $q$ to be $3x+3y$, $3y+3z$ respectively in this reduct, as in the previous case. Adding identities from the system \eqref{jedan} can only give us the whole system \eqref{jedan}. 

If we equalize the terms of the second and the third identity of the system \eqref{aj}, we shall obtain the following:\begin{equation*}
\left\{
\begin{array}{c}
p(x,y,x) \approx q(x,x,y)\\
p(x,y,y)\approx q(x,y,x)\approx
p(x,x,y)\approx q(y,x,x)
\end{array} \right. 
\end{equation*}
This system holds in a full idempotent reduct of a module over $\mathbb{Z}_5$, for we can define  $p$ and $q$ to be $3x+3z$, $3x+3y$ respectively in this reduct. Adding identities from the system \eqref{jedan} can only give us the whole system \eqref{jedan}. 

If we equalize the terms of the  first and the third identity of the system \eqref{aj}, we shall obtain the following:\begin{equation*}
\left\{
\begin{array}{c}
p(x,y,x) \approx q(x,x,y)\approx
p(x,x,y)\approx q(y,x,x)\\
p(x,y,y)\approx q(x,y,x)
\end{array} \right.  
\end{equation*}
This system holds in a full idempotent reduct of a module over $\mathbb{Z}_3$, for we can define  $p$ and $q$ to be $2x+y+z$, $x+2y+z$ respectively in this reduct. Adding identities from the system \eqref{jedan} can only give us the whole system \eqref{jedan}.
\noindent By this we have finished analyzing the system \eqref{aj}. By adding identities to this system we can obtain either a system that holds in some full idempotent reduct of a module over a finite ring, or the whole system \eqref{jedan}.  
\end{itemize}

By this we have examined all proper subsets of \eqref{jedan} that include the identities \eqref{ad}. Each of them holds in a reduct of a module over some finite ring, therefore does not describe omitting types 1 and 2.  
 
\item the second identity is $p(x,y,y)\approx q(y,x,x)$, i.e.we have these two identities:
\begin{equation}
\left\{
\begin{array}{c}
p(x,y,x) \approx q(x,x,y)\\
p(x,y,y)\approx q(y,x,x)
\end{array} \right.\label{ak}   
\end{equation}
We shall not examine subsets of \eqref{jedan} including \eqref{ak}, but instead notice the following: if there is a subset $\sigma_1$ of \eqref{jedan} that includes \eqref{ak} and implies omitting types 1 and 2, then there is a subset $\sigma_2$ of \eqref{jedan} that includes \eqref{ad} and implies omitting types 1 and 2 ($\sigma_2$ is obtained simply  by substituting $q(y,x,z)$ for $q(x,y,z)$ in $\sigma_1$). We have already proven that $\sigma_2$ does not exist, which means $\sigma_1$  also does not exist, so there is no need to discuss this case at all. 
\end{itemize}

\vspace{0.2 cm}

\noindent By this point we have examined all proper subsets of the system \eqref{jedan} including   the identity $p(x,y,x) \approx q(x,x,y)$ (this identity is considered the first in the examined systems). We have obtained no other system but \eqref{dva} implying omitting types 1 and 2.

\subsection{}\label{sub2}
In this subsection we shall examine proper subsets of the system \eqref{jedan} having one of the following identities for the first one: $p(x,y,x) \approx q(x,y,x)$, $p(x,y,x) \approx q(y,x,x)$, $p(x,x,y) \approx q(x,x,y)$ , $p(x,x,y) \approx q(x,y,x)$, $p(x,x,y) \approx q(y,x,x)$. Let us discuss each of these cases:
\begin{itemize}
\item the first identity is $p(x,y,x) \approx q(x,y,x)$ -- suppose there is a subset of \eqref{jedan} having  $p(x,y,x) \approx q(x,y,x)$ for the first identity and implying omitting types 1 and 2. We shall denote this subset by  $\sigma_1$ . Now we can substitute $q(x,y,z)$ for $q(x,z,y)$ in $\sigma_1$ and obtain an equivalent system $\sigma_2$ (also a subset of \eqref{jedan}). The system $\sigma_2$ has $p(x,y,x) \approx q(x,x,y)$ for the first identity, and of course, implies omitting types 1 and 2. We have proved in the previous subsection that there is no other system but \eqref{dva} having these properties, so $\sigma_2$ is actually the system \eqref{dva}, and $\sigma_1$ is equivalent to \eqref{dva} (obtained by the permutation of variables). This means we can obtain nothing new with the first identity being $p(x,y,x) \approx q(x,y,x)$.
\item the first identity is $p(x,y,x) \approx q(y,x,x)$ -- everything is the same as in the previous item, except for the permutation needed, for in this case we need to substitute $q(x,y,z)$ for $q(z,y,x)$. 
\item the first identity is $p(x,x,y) \approx q(x,x,y)$ -- we shall substitute $p(x,y,z)$ for $p(x,z,y)$ here.
\item $p(x,x,y) \approx q(x,y,x)$ -- we shall substitute $p(x,y,z)$ for $p(x,z,y)$ here and $q(x,y,z)$ for $q(x,z,y)$.
\item $p(x,x,y) \approx q(y,x,x)$ -- we shall substitute $p(x,y,z)$ for $p(x,z,y)$ and $q(x,y,z)$ for $q(z,y,x)$.
\end{itemize}

\vspace{0.2 cm}

\noindent We can now conclude the following: in each of these cases we can obtain either a system equivalent to \eqref{dva}, obtained by a permutation of variables, or a system that holds in some idempotent reduct of a module over a finite ring. 

\subsection{}
In this subsection we shall examine proper subsets of the system \eqref{jedan} having one of the remaining three identities for the first one: $p(x,y,y) \approx q(x,y,x)$, $p(x,y,y) \approx q(x,x,y)$, $p(x,y,y) \approx q(y,x,x)$. 

\vspace{0.2 cm}

\noindent There is no need to go through all of these cases -- namely, if the first identity is  $p(x,y,y) \approx q(x,x,y)$, then by substituting $q(x,y,z)$ for $q(x,z,y)$ in the whole system (which is a subset of \eqref{jedan}, of course), we obtain an equivalent system, also a subset of \eqref{jedan}, with the first identity being $p(x,y,y) \approx q(x,y,x)$. If the first identity is $p(x,y,y) \approx q(y,x,x)$, we shall substitute $q(x,y,z)$ for $q(y,x,z)$ and obtain the same as in the previous case. Therefore \emph{it will be sufficient to examine proper subsets of \eqref{jedan} having $p(x,y,y) \approx q(x,y,x)$ for the first identity}. As done in the subsection \ref{sub1}, we shall vary the second identity, first considering ones only on $p$ and then ones on $p$ and $q$. 

\begin{itemize}
\item the second identity is $p(x,x,y) \approx p(x,y,x)$, i.e. so far we have this system:
\begin{equation}
\left\{
\begin{array}{c}
p(x,y,y) \approx q(x,y,x)\\
p(x,x,y) \approx p(x,y,x) \\
\end{array}\right. \label{al}
\end{equation}
Obviously both $p$ and $q$ can be defined as projection maps in any algebra, so we need to add some more identities from \eqref{jedan}. This can be done in the following ways: we can equalize the  terms of the two identities given, or equalize each of the terms $q(x,x,y)$, $q(y,x,x)$ with terms of each of the identities, or add the identity $q(x,x,y) \approx q(y,x,x)$ to the system. Let us discuss each possibility:
\begin{itemize}
\item by equalizing the terms of the identities  \eqref{al}, we obtain the following:\begin{equation}
\begin{array}{c}
p(x,y,y) \approx q(x,y,x)\approx
p(x,x,y) \approx p(x,y,x) 
\end{array} \label{am}
\end{equation}
Still both terms can be defined as projection maps, so more identities from \eqref{jedan} are needed. Let us discuss each way of adding an identity: 

If we equalize the term $q(x,x,y)$ with the terms of \eqref{am}, we shall obtain the following:\begin{equation*}
\begin{array}{c}
p(x,y,y) \approx q(x,y,x)\approx
p(x,x,y) \approx p(x,y,x)\approx q(x,x,y) 
\end{array} 
\end{equation*}
This still allows both projections, and adding identities from \eqref{jedan} can only give us the whole system \eqref{jedan}.

If we equalize the term $q(y,x,x)$ with the terms of \eqref{am}, we shall obtain the following:\begin{equation*}
\begin{array}{c}
p(x,y,y) \approx q(x,y,x)\approx
p(x,x,y) \approx p(x,y,x)\approx q(y,x,x) 
\end{array} 
\end{equation*}
This still allows both projections, and adding identities from \eqref{jedan} can only give us the whole system \eqref{jedan}.

If we add the identity $q(x,x,y) \approx q(y,x,x)$ to the system \eqref{am} we shall obtain the following system:\begin{equation*}
\begin{array}{c}
p(x,y,y) \approx q(x,y,x)\approx
p(x,x,y) \approx p(x,y,x)\\
q(x,x,y) \approx q(y,x,x) 
\end{array} 
\end{equation*}
This system holds in a full idempotent reduct of a module over $\mathbb{Z}_5$, for we can define  $p$ and $q$ to be $\pi_1$, $3x+3z$ respectively in this reduct. Adding identities from the system \eqref{jedan} can only give us the whole system \eqref{jedan}. 
\noindent By this we have finished analyzing the system \eqref{am}. By adding identities to this system we can obtain either a system that holds in some full idempotent reduct of a module over a finite ring, or the whole system \eqref{jedan}. 
\item if we equalize the term $q(x,x,y)$ with the terms of the  first identity of the system \eqref{al}, we shall obtain the following:\begin{equation}
\left\{
\begin{array}{c}
p(x,y,y) \approx q(x,y,x)\approx q(x,x,y) \\
p(x,x,y) \approx p(x,y,x) 
\end{array}\right. \label{an}
\end{equation}
This system allows both projection maps in any algebra, so we need to add more identities from \eqref{jedan}. Let us discuss each way of adding an identity: 

If we equalize the terms of the identities \eqref{an}, we shall obtain the following: \begin{equation*}
\begin{array}{c}
p(x,y,y) \approx q(x,y,x)\approx q(x,x,y) \approx
p(x,x,y) \approx p(x,y,x) 
\end{array} 
\end{equation*}
Still both terms can be projection maps, and adding any new identities from \eqref{jedan} can only give us \eqref{jedan}.

If we equalize the term $q(y,x,x)$ with the terms of the first identity of \eqref{an}, we shall obtain the following:\begin{equation*}
\left\{
\begin{array}{c}
p(x,y,y) \approx q(x,y,x)\approx q(x,x,y)\approx q(y,x,x) \\
p(x,x,y) \approx p(x,y,x) 
\end{array}\right.
\end{equation*}
This system holds in a full idempotent reduct of a module over $\mathbb{Z}_5$, for we can define  $p$ and $q$ to be $4x+y+z$, $2x+2y+2z$ respectively in this reduct. Adding identities from the system \eqref{jedan} can only give us the whole system \eqref{jedan}.

If we equalize the term $q(y,x,x)$ with the terms of the second identity of \eqref{an}, we shall obtain the following: \begin{equation*}
\left\{
\begin{array}{c}
p(x,y,y) \approx q(x,y,x)\approx q(x,x,y) \\
p(x,x,y) \approx p(x,y,x)\approx q(y,x,x) 
\end{array}\right.
\end{equation*}
This system holds in a full idempotent reduct of a module over $\mathbb{Z}_4$, for we can define  $p$ and $q$ to be $3x+y+z$, $x+2y+2z$ respectively in this reduct. Adding identities from the system \eqref{jedan} can only give us the whole system \eqref{jedan}.
\noindent By this we have finished analyzing the system \eqref{an}. By adding identities to this system we can obtain either a system that holds in some full idempotent reduct of a module over a finite ring, or the whole system \eqref{jedan}. 
\item if we equalize the term $q(x,x,y)$ with the terms of the  second identity of the system \eqref{al}, we shall obtain the following:\begin{equation}
\left\{
\begin{array}{c}
p(x,y,y) \approx q(x,y,x) \\
p(x,x,y) \approx p(x,y,x)\approx q(x,x,y) 
\end{array}\right. \label{ao}
\end{equation}
This system allows both projection maps, so we need to add more identities from \eqref{jedan}. Let us discuss each way of adding an identity: 

If we equalize the terms of the identities \eqref{ao}, we shall obtain the following: \begin{equation*}
\begin{array}{c}
p(x,y,y) \approx q(x,y,x) \approx
p(x,x,y) \approx p(x,y,x)\approx q(x,x,y) 
\end{array}
\end{equation*}
Still both terms can be defined as projection maps in any algebra, and adding identities from the system \eqref{jedan} can only give us the whole system \eqref{jedan}.

If we equalize the term $q(y,x,x)$ with the terms of the first identity of \eqref{ao}, we shall obtain the following: \begin{equation*}
\left\{
\begin{array}{c}
p(x,y,y) \approx q(x,y,x) \approx q(y,x,x) \\
p(x,x,y) \approx p(x,y,x)\approx q(x,x,y) 
\end{array}\right. 
\end{equation*}
This system holds in a full idempotent reduct of a module over $\mathbb{Z}_4$, for we can define  $p$ and $q$ to be $3x+y+z$, $2x+2y+z$ respectively in this reduct. Adding identities from the system \eqref{jedan} can only give us the whole system \eqref{jedan}. 

If we equalize the term $q(y,x,x)$ with the terms of the second identity of \eqref{ao}, we shall obtain the following:\begin{equation*}
\left\{
\begin{array}{c}
p(x,y,y) \approx q(x,y,x)  \\
p(x,x,y) \approx p(x,y,x)\approx q(x,x,y)\approx q(y,x,x) 
\end{array}\right. 
\end{equation*}
This system holds in a full idempotent reduct of a module over $\mathbb{Z}_3$, for we can define  $p$ and $q$ to be $2x+y+z$, $x+2y+z$ respectively in this reduct. Adding identities from the system \eqref{jedan} can only give us the whole system \eqref{jedan}.
\noindent By this we have finished analyzing the system \eqref{ao}. By adding identities to this system we can obtain either a system that holds in some full idempotent reduct of a module over a finite ring, or the whole system \eqref{jedan}. 
\item if we equalize the term $q(y,x,x)$ with the terms of the  first identity of the system \eqref{al}, we shall obtain the following: \begin{equation}
\left\{
\begin{array}{c}
p(x,y,y) \approx q(x,y,x)\approx q(y,x,x)\\
p(x,x,y) \approx p(x,y,x) \\
\end{array}\right. \label{ap}
\end{equation}
This system allows both $p$ and $q$ to be defined as projection maps in any algebra, so we need to add some more identities from \eqref{jedan}. Let us discuss each way of adding an identity: 

If we equalize the terms of the identities \eqref{ap}, we shall obtain the following:\begin{equation*}
\begin{array}{c}
p(x,y,y) \approx q(x,y,x)\approx q(y,x,x)\approx
p(x,x,y) \approx p(x,y,x) 
\end{array}
\end{equation*}
Still both terms can be projection maps, and adding identities from \eqref{jedan} can only give us \eqref{jedan}.

If we equalize the term $q(x,x,y)$ with the terms of the  first identity of \eqref{ap}, we shall obtain he following: \begin{equation*}
\left\{
\begin{array}{c}
p(x,y,y) \approx q(x,y,x)\approx q(y,x,x)\approx q(x,x,y) \\
p(x,x,y) \approx p(x,y,x) \\
\end{array}\right. 
\end{equation*}
This system holds in a full idempotent reduct of a module over $\mathbb{Z}_5$, for we can define  $p$ and $q$ to be $4x+y+z$, $2x+2y+2z$ respectively in this reduct. Adding identities from the system \eqref{jedan} can only give us the whole system \eqref{jedan}.

If we equalize the term $q(x,x,y)$ with the terms of the second identity of \eqref{ap}, we shall obtain he following: \begin{equation*}
\left\{
\begin{array}{c}
p(x,y,y) \approx q(x,y,x)\approx q(y,x,x)\\
p(x,x,y) \approx p(x,y,x) \approx q(x,x,y)  
\end{array}\right. 
\end{equation*}
This system holds in a full idempotent reduct of a module over $\mathbb{Z}_4$, for we can define  $p$ and $q$ to be $3x+y+z$, $2x+2y+z$ respectively in this reduct. Adding identities from the system \eqref{jedan} can only give us the whole system \eqref{jedan}. 
\noindent By this we have finished analyzing the system \eqref{ap}. By adding identities to this system we can obtain either a system that holds in some full idempotent reduct of a module over a finite ring, or the whole system \eqref{jedan}. 
\item if we equalize the term $q(y,x,x)$ with the terms of the  second identity of the system \eqref{al}, we shall obtain the following:\begin{equation}
\left\{
\begin{array}{c}
p(x,y,y) \approx q(x,y,x)\\
p(x,x,y) \approx p(x,y,x) \approx q(y,x,x) 
\end{array}\right. \label{ar}
\end{equation}
Both $p$ and $q$ can be defined as projection maps in any algebra, so we need to add some more identities from \eqref{jedan}. Let us discuss each way of adding an identity: 

If we equalize the terms of the  identities  \eqref{ar}, we shall obtain the following: \begin{equation*}
\begin{array}{c}
p(x,y,y) \approx q(x,y,x)\approx
p(x,x,y) \approx p(x,y,x) \approx q(y,x,x) 
\end{array}
\end{equation*}
Still both terms can be defined as projection maps, and adding identities from \eqref{jedan} can only give us the whole system \eqref{jedan}.

If we equalize the term $q(x,x,y)$ with the terms of the first identity of \eqref{ar}, we shall obtain the following:\begin{equation*}
\left\{
\begin{array}{c}
p(x,y,y) \approx q(x,y,x) \approx q(x,x,y)\\
p(x,x,y) \approx p(x,y,x) \approx q(y,x,x) 
\end{array}\right. 
\end{equation*}
This system holds in a full idempotent reduct of a module over $\mathbb{Z}_4$, for we can define  $p$ and $q$ to be $3x+y+z$, $x+2y+2z$ respectively in this reduct. Adding identities from the system \eqref{jedan} can only give us the whole system \eqref{jedan}. 

If we equalize the term $q(x,x,y)$ with the terms of the second identity of \eqref{ar}, we shall obtain the following:\begin{equation*}
\left\{
\begin{array}{c}
p(x,y,y) \approx q(x,y,x) \\
p(x,x,y) \approx p(x,y,x) \approx q(y,x,x) \approx q(x,x,y) 
\end{array}\right. 
\end{equation*}
This system holds in a full idempotent reduct of a module over $\mathbb{Z}_3$, for we can define  $p$ and $q$ to be $2x+y+z$, $x+2y+z$ respectively in this reduct. Adding identities from the system \eqref{jedan} can only give us the whole system \eqref{jedan}.
\noindent By this we have finished analyzing the system \eqref{ar}. By adding identities to this system we can obtain either a system that holds in some full idempotent reduct of a module over a finite ring, or the whole system \eqref{jedan}. 
\item if we add the identity $q(x,x,y) \approx q(y,x,x)$ to the system \eqref{al} we shall obtain the following:\begin{equation}
\left\{
\begin{array}{c}
p(x,y,y) \approx q(x,y,x)\\
p(x,x,y) \approx p(x,y,x) \\
q(x,x,y) \approx q(y,x,x)
\end{array}\right. \label{as}
\end{equation}
This system holds in a full idempotent reduct of a module over $\mathbb{Z}_5$, for we can define  $p$ and $q$ to be $3y+3z$, $\pi_2$ respectively in this reduct. We need to add some more identities from \eqref{jedan}, and is this case this is done by equalizing  terms of any two of the three identities given above. Let us discuss each way of doing this:  

If we equalize the terms of the first and the second identity of \eqref{as}, we shall obtain the following:\begin{equation*}
\left\{
\begin{array}{c}
p(x,y,y) \approx q(x,y,x)\approx
p(x,x,y) \approx p(x,y,x) \\
q(x,x,y) \approx q(y,x,x)
\end{array}\right. 
\end{equation*}
This system holds in a full idempotent reduct of a module over $\mathbb{Z}_5$, for we can define  $p$ and $q$ to be $\pi_1$, $3x+3z$ respectively in this reduct. Adding any new identities from \eqref{jedan}can only give us \eqref{jedan}. 

If we equalize the terms of the second and the third identity of  \eqref{as}, we shall obtain the following: \begin{equation*}
\left\{
\begin{array}{c}
p(x,y,y) \approx q(x,y,x)\\
p(x,x,y) \approx p(x,y,x) \approx
q(x,x,y) \approx q(y,x,x)
\end{array}\right. 
\end{equation*}
This system holds in a full idempotent reduct of a module over $\mathbb{Z}_3$, for we can define  $p$ and $q$ to be $2x+y+z$, $x+2y+z$ respectively in this reduct. Adding any new identities from \eqref{jedan}can only give us \eqref{jedan}.

If we equalize the terms of the  first and the third identity of \eqref{as}, we shall obtain  the following system:\begin{equation*}
\left\{
\begin{array}{c}
p(x,y,y) \approx q(x,y,x)\approx
q(x,x,y) \approx q(y,x,x)\\
p(x,x,y) \approx p(x,y,x)
\end{array}\right. 
\end{equation*}
This system holds in a full idempotent reduct of a module over $\mathbb{Z}_5$, for we can define  $p$ and $q$ to be $4x+y+z$, $2x+2y+2z$ respectively in this reduct. Adding any new identities from \eqref{jedan} can only give us \eqref{jedan}.
\noindent By this we have finished analyzing the system \eqref{as}. By adding identities to this system we can obtain either a system that holds in some full idempotent reduct of a module over a finite ring, or the whole system \eqref{jedan}. 
\end{itemize}
\noindent By this we have examined all proper subsets of \eqref{jedan} that include the identities \eqref{al}. Each of them holds in a reduct of a module over some finite ring, therefore does not describe omitting types 1 and 2. 
\item the second identity is $p(x,y,y) \approx p(x,x,y)$, i.e. we have the following system:\begin{equation}
\begin{array}{c}
p(x,y,y) \approx q(x,y,x) \approx p(x,x,y)
\end{array}\label{at}
\end{equation} 
This system obviously allows both $p$ and $q$ to be defined as projection maps in any algebra, so we need to add more identities from the system \eqref{jedan}. This can be done in several ways: we can equalize one of the terms $p(x,y,x)$, $q(x,x,y)$, $q(y,x,x)$ with the terms of the identities given, or add a new identity to the system, i.e. one of the identities $p(x,y,x) \approx q(x,x,y)$, $p(x,y,x) \approx q(y,x,x)$, $q(x,x,y) \approx q(y,x,x)$. We shall go through all these cases:
\begin{itemize}
\item by equalizing the term $p(x,y,x)$ with the terms of the  identities \eqref{at}, we obtain the following system:\begin{equation}
\begin{array}{c}
p(x,y,y) \approx q(x,y,x) \approx p(x,x,y)\approx p(x,y,x)
\end{array}\label{au}
\end{equation} 
This still allows both projection maps, so we need to add more identities from \eqref{jedan}. Let us discuss each way of doing this:

If we equalize the term $q(x,x,y)$ with the terms of the  identities \eqref{au}, we shall obtain the following:\begin{equation*}
\begin{array}{c}
p(x,y,y) \approx q(x,y,x) \approx p(x,x,y)\approx p(x,y,x)\approx q(x,x,y)
\end{array}
\end{equation*} 
Still both terms can be defined as projection maps in any algebra, and adding more identities from \eqref{jedan} can only give us the whole system \eqref{jedan}.

If we equalize the term $q(y,x,x)$ with the terms of the  identities \eqref{au}, we shall obtain the following: \begin{equation*}
\begin{array}{c}
p(x,y,y) \approx q(x,y,x) \approx p(x,x,y)\approx p(x,y,x)\approx q(y,x,x)
\end{array}
\end{equation*} 
Still both terms can be defined as projection maps in any algebra, and adding more identities from \eqref{jedan} can only give us the whole system \eqref{jedan}.

If we add the identity $q(x,x,y) \approx q(y,x,x)$ to the identities \eqref{au} we shall obtain the following:\begin{equation*}
\left\{
\begin{array}{c}
p(x,y,y) \approx q(x,y,x) \approx p(x,x,y)\approx p(x,y,x)\\
q(x,x,y) \approx q(y,x,x)
\end{array}\right. 
\end{equation*}
This system holds in a full idempotent reduct of a module over $\mathbb{Z}_5$, for we can define  $p$ and $q$ to be $\pi_1$, $3x+3z$ respectively in this reduct. Adding any new identities from \eqref{jedan} can only give us \eqref{jedan}.
\noindent By this we have finished analyzing the system \eqref{au}. By adding identities to this system we can obtain either a system that holds in some full idempotent reduct of a module over a finite ring, or the whole system \eqref{jedan}. 
\item by equalizing the term $q(x,x,y)$ with the terms of the identities \eqref{at}, we obtain the following system:\begin{equation}
\begin{array}{c}
p(x,y,y) \approx q(x,y,x) \approx p(x,x,y) \approx q(x,x,y)
\end{array}\label{av}
\end{equation}
This system obviously allows both $p$ and $q$ to be defined as projection maps in any algebra, so we need to add more identities from the system \eqref{jedan}. Let us discuss each way of adding an identity: 

If we equalize the term $p(x,y,x)$ with the terms of the identities \eqref{av}, we shall obtain the following:\begin{equation*}
\begin{array}{c}
p(x,y,y) \approx q(x,y,x) \approx p(x,x,y) \approx q(x,x,y)\approx p(x,y,x)
\end{array}
\end{equation*}
Still both terms can be defined as projection maps in any algebra, and adding more identities from \eqref{jedan} can only give us the whole system \eqref{jedan}.

If we equalize the term $q(y,x,x)$ with the terms of the  identities \eqref{av}, we shall obtain the following:\begin{equation*}
\begin{array}{c}
p(x,y,y) \approx q(x,y,x) \approx p(x,x,y) \approx q(x,x,y)\approx q(y,x,x)
\end{array}
\end{equation*}
This system holds in a full idempotent reduct of a module over $\mathbb{Z}_5$, for we can define  $p$ and $q$ to be $4x+2z$, $2x+2y+2z$ respectively in this reduct. Adding any new identities from \eqref{jedan} can only give us \eqref{jedan}.

If we add the identity $p(x,y,x)\approx q(y,x,x)$ to the identities \eqref{av} we shall obtain the following: \begin{equation*}
\left\{
\begin{array}{c}
p(x,y,y) \approx q(x,y,x) \approx p(x,x,y) \approx q(x,x,y)\\
p(x,y,x)\approx q(y,x,x)
\end{array}\right. 
\end{equation*}
This system holds in a full idempotent reduct of a module over $\mathbb{Z}_5$, for we can define  $p$ and $q$ to be $3x+3z$, $3y+3z$ respectively in this reduct. Adding any new identities from \eqref{jedan} can only give us \eqref{jedan}.
\noindent By this we have finished analyzing the system \eqref{av}. By adding identities to this system we can obtain either a system that holds in some full idempotent reduct of a module over a finite ring, or the whole system \eqref{jedan}. 
\item by equalizing the term $q(y,x,x)$ with the terms of the  identities \eqref{at}, we obtain the following system:\begin{equation}
\begin{array}{c}
p(x,y,y) \approx q(x,y,x) \approx p(x,x,y) \approx q(y,x,x)
\end{array}\label{aw}
\end{equation}
This system allows both $p$ and $q$ to be defined as projection maps in any algebra, so we need to add more identities from the system \eqref{jedan}. Let us discuss each way of adding an identity: 

If we equalize the term $p(x,y,x)$ with the terms of the identities \eqref{aw}, we shall obtain the following: \begin{equation*}
\begin{array}{c}
p(x,y,y) \approx q(x,y,x) \approx p(x,x,y) \approx q(y,x,x)\approx p(x,y,x)
\end{array}
\end{equation*}
Still both terms can be defined as projection maps in any algebra, and adding more identities from \eqref{jedan} can only give us the whole system \eqref{jedan}.

If we equalize the term $q(x,x,y)$ with the terms of the identities \eqref{aw}, we shall obtain the following: \begin{equation*}
\begin{array}{c}
p(x,y,y) \approx q(x,y,x) \approx p(x,x,y) \approx q(y,x,x)\approx q(x,x,y)
\end{array}
\end{equation*}
This system holds in a full idempotent reduct of a module over $\mathbb{Z}_5$, for we can define  $p$ and $q$ to be $4x+2z$, $2x+2y+2z$ respectively in this reduct. Adding any new identities from \eqref{jedan} can only give us \eqref{jedan}.

If we add the identity $p(x,y,x)\approx q(x,x,y)$ to the identities \eqref{aw} we shall obtain the following:\begin{equation*}
\left\{
\begin{array}{c}
p(x,y,y) \approx q(x,y,x) \approx p(x,x,y) \approx q(y,x,x)\\
p(x,y,x)\approx q(x,x,y)
\end{array}\right.
\end{equation*}
This system holds in a full idempotent reduct of a module over $\mathbb{Z}_5$, for we can define  $p$ and $q$ to be $3x+3z$, $3x+3y$ respectively in this reduct. Adding any new identities from \eqref{jedan} can only give us \eqref{jedan}.
\noindent By this we have finished analyzing the system \eqref{aw}. By adding identities to this system we can obtain either a system that holds in some full idempotent reduct of a module over a finite ring, or the whole system \eqref{jedan}. 
\item if we add the identity $p(x,y,x) \approx q(x,x,y)$ to the system \eqref{at} we shall obtain the following:\begin{equation}
\left\{
\begin{array}{c}
p(x,y,y) \approx q(x,y,x) \approx p(x,x,y)\\
p(x,y,x) \approx q(x,x,y)
\end{array} \right. \label{ax}
\end{equation}
This system allows both $p$ and $q$ to be defined as projection maps in any algebra, so we need to add more identities from \eqref{jedan}. Let us discuss each way of adding an identity: 

If we equalize the terms of the identities \eqref{ax}, we shall obtain the following: \begin{equation*}
\begin{array}{c}
p(x,y,y) \approx q(x,y,x) \approx p(x,x,y)\approx
p(x,y,x) \approx q(x,x,y)
\end{array} 
\end{equation*} 
Still both projections are allowed, and adding identities from \eqref{jedan} can only give us the whole system \eqref{jedan}. 

If we equalize the term $q(y,x,x)$ with the terms of the first identity of the system \eqref{ax},  we shall obtain the following:\begin{equation*}
\left\{
\begin{array}{c}
p(x,y,y) \approx q(x,y,x) \approx p(x,x,y)\approx q(y,x,x)\\
p(x,y,x) \approx q(x,x,y)
\end{array} \right. 
\end{equation*}
This system holds in a full idempotent reduct of a module over $\mathbb{Z}_5$, for we can define  $p$ and $q$ to be $3x+3z$, $3x+3y$ respectively in this reduct. Adding any new identities from \eqref{jedan} can only give us \eqref{jedan}.

If we equalize the term $q(y,x,x)$ with the terms of the  second identity of the system \eqref{ax}, we shall obtain the following:\begin{equation*}
\left\{
\begin{array}{c}
p(x,y,y) \approx q(x,y,x) \approx p(x,x,y)\\
p(x,y,x) \approx q(x,x,y)\approx q(y,x,x)
\end{array} \right. 
\end{equation*}
This system allows both projection maps in any algebra ($p$ and $q$ can be defined as $\pi_3$, $\pi_2$ respectively), and adding any more identities from \eqref{jedan} can only give us \eqref{jedan}.
\noindent By this we have finished analyzing the system \eqref{ax}. By adding identities to this system we can obtain either a system that holds in some full idempotent reduct of a module over a finite ring, or the whole system \eqref{jedan}. 
\item if we add the identity $p(x,y,x) \approx q(y,x,x)$ to the system \eqref{at} we shall obtain the following:
\begin{equation}
\left\{
\begin{array}{c}
p(x,y,y) \approx q(x,y,x) \approx p(x,x,y)\\
p(x,y,x) \approx q(y,x,x)
\end{array} \right. \label{ay}
\end{equation}
This system allows both terms to be defined as projection maps in any algebra, so we need to add more identities from \eqref{jedan}. Let us discuss each way of adding an identity:  

If we equalize the terms of the identities  \eqref{ay}, we shall obtain the following: \begin{equation*}
\begin{array}{c}
p(x,y,y) \approx q(x,y,x) \approx p(x,x,y)\approx
p(x,y,x) \approx q(y,x,x)
\end{array} 
\end{equation*}
This still allows both projection maps, and adding any new identities from \eqref{jedan} can only give us \eqref{jedan}.

If we equalize the term $q(x,x,y)$ with the terms of the  first identity of \eqref{ay}, we shall obtain the following:\begin{equation*}
\left\{
\begin{array}{c}
p(x,y,y) \approx q(x,y,x) \approx p(x,x,y)\approx q(x,x,y)\\
p(x,y,x) \approx q(y,x,x)
\end{array} \right. 
\end{equation*}
This system holds in a full idempotent reduct of a module over $\mathbb{Z}_5$, for we can define  $p$ and $q$ to be $3x+3z$, $3y+3z$ respectively in this reduct. Adding any new identities from \eqref{jedan} can only give us \eqref{jedan}.

If we equalize the term $q(x,x,y)$ with the terms of the second identity of \eqref{ay}, we shall obtain the following:\begin{equation*}
\left\{
\begin{array}{c}
p(x,y,y) \approx q(x,y,x) \approx p(x,x,y)\\
p(x,y,x) \approx q(y,x,x) \approx q(x,x,y)
\end{array} \right. 
\end{equation*}
This system allows both projection maps in any algebra ($p$ and $q$ can be defined as $\pi_3$, $\pi_2$ respectively), and adding any more identities from \eqref{jedan} can only give us \eqref{jedan}.
\noindent By this we have finished analyzing the system \eqref{ay}. By adding identities to this system we can obtain either a system that holds in some full idempotent reduct of a module over a finite ring, or the whole system \eqref{jedan}. 
\item if we add the identity $q(x,x,y) \approx q(y,x,x)$ to the system \eqref{at} we shall obtain the following:
\begin{equation}
\left\{
\begin{array}{c}
p(x,y,y) \approx q(x,y,x) \approx p(x,x,y)\\
q(x,x,y) \approx q(y,x,x)
\end{array} \right. \label{az}
\end{equation}
This system allows both projection maps in any algebra ($p$ and $q$ can be defined as $\pi_3$, $\pi_2$ respectively), so we need to add more identities from \eqref{jedan}. Let us discuss each way of adding an identity:  

By equalizing  the terms of the identities \eqref{az}, we obtain the following system: \begin{equation*}
\begin{array}{c}
p(x,y,y) \approx q(x,y,x) \approx p(x,x,y)\approx
q(x,x,y) \approx q(y,x,x)
\end{array} 
\end{equation*}
This system holds in a full idempotent reduct of a module over $\mathbb{Z}_5$, for we can define  $p$ and $q$ to be $4x+2z$, $2x+2y+2z$ respectively in this reduct. Adding any new identities from \eqref{jedan} can only give us \eqref{jedan}.

If we equalize the term $p(x,y,x)$ with the terms of the  first identity of  \eqref{az}, we shall obtain the following:\begin{equation*}
\left\{
\begin{array}{c}
p(x,y,y) \approx q(x,y,x) \approx p(x,x,y) \approx p(x,y,x) \\
q(x,x,y) \approx q(y,x,x)
\end{array} \right. 
\end{equation*}
This system holds in a full idempotent reduct of a module over $\mathbb{Z}_5$, for we can define  $p$ and $q$ to be $\pi_1$, $3x+3z$ respectively in this reduct. Adding any new identities from \eqref{jedan} can only give us \eqref{jedan}.

If we equalize the term $p(x,y,x)$ with the terms of the  second identity of  \eqref{az}, we shall obtain the following:\begin{equation*}
\left\{
\begin{array}{c}
p(x,y,y) \approx q(x,y,x) \approx p(x,x,y) \\
q(x,x,y) \approx q(y,x,x) \approx p(x,y,x)
\end{array} \right. 
\end{equation*}
This system allows both projection maps in any algebra ($p$ and $q$ can be defined as $\pi_3$, $\pi_2$ respectively), and adding any identities from \eqref{jedan} can only give us \eqref{jedan}. 
\noindent By this we have finished analyzing the system \eqref{az}. By adding identities to this system we can obtain either a system that holds in some full idempotent reduct of a module over a finite ring, or the whole system \eqref{jedan}. 
\end{itemize}

\vspace{0.2 cm}

\noindent By this we have examined all proper subsets of \eqref{jedan} that include the identities \eqref{at}. Each of them holds in some full idempotent reduct of a module over a finite ring, therefore does not imply omitting types 1 and 2. 
\item the second identity is $p(x,y,y) \approx p(x,y,x)$, i.e. we have the following system:\begin{equation}
\begin{array}{c}
p(x,y,y) \approx q(x,y,x) \approx p(x,y,x)
\end{array}\label{aq}
\end{equation} 
This system obviously allows both $p$ and $q$ to be defined as projection maps, so we need to add more identities from \eqref{jedan}. This can be done in the following ways: we can equalize one of the terms $p(x,x,y)$, $q(x,x,y)$, $q(y,x,x)$ with the terms in the identities given, or add to them one of the identities $p(x,x,y) \approx q(x,x,y)$, $q(x,x,y) \approx q(y,x,x)$ , $p(x,x,y) \approx q(y,x,x)$. As before we shall discuss each case:
\begin{itemize}
\item if we equalize the term $p(x,x,y)$ with the terms of the identities \eqref{aq}, we shall obtain the following:\begin{equation}
\begin{array}{c}
p(x,y,y) \approx q(x,y,x) \approx p(x,y,x)\approx p(x,x,y)
\end{array}\label{ba}
\end{equation}
This system still allows both terms to be defined as projection maps in any algebra, so we need to add more identities from \eqref{jedan}. Let us discuss each way of adding an identity:  

If we equalize the term $q(x,x,y)$ with the terms of the identities \eqref{ba}, we shall obtain the following: \begin{equation*}
\begin{array}{c}
p(x,y,y) \approx q(x,y,x) \approx p(x,y,x)\approx p(x,x,y) \approx q(x,x,y) 
\end{array}
\end{equation*}
This still allows both projections, and adding identities from \eqref{jedan} can only give us \eqref{jedan}.

If we equalize the term $q(y,x,x)$ with the terms of the identities \eqref{ba}, we shall obtain the following: \begin{equation*}
\begin{array}{c}
p(x,y,y) \approx q(x,y,x) \approx p(x,y,x)\approx p(x,x,y) \approx q(y,x,x) 
\end{array}
\end{equation*}
This still allows both projections, and adding identities from \eqref{jedan} can only give us \eqref{jedan}.

If we add the identity  $q(x,x,y) \approx q(y,x,x)$ to the system \eqref{ba} we shall obtain the following: \begin{equation*}
\left\{
\begin{array}{c}
p(x,y,y) \approx q(x,y,x) \approx p(x,y,x)\approx p(x,x,y)\\
q(x,x,y) \approx q(y,x,x)
\end{array}\right.
\end{equation*}
This system holds in a full idempotent reduct of a module over $\mathbb{Z}_5$, for we can define  $p$ and $q$ to be $\pi_1$, $3x+3z$ respectively in this reduct. Adding any new identities from \eqref{jedan} can only give us \eqref{jedan}.
\noindent By this we have finished analyzing the system \eqref{ba}. By adding identities to this system we can obtain either a system that holds in some full idempotent reduct of a module over a finite ring, or the whole system \eqref{jedan}. 
\item if we equalize the term $q(x,x,y)$ with the terms of the  identities \eqref{aq}, we shall obtain the following:\begin{equation}
\begin{array}{c}
p(x,y,y) \approx q(x,y,x) \approx p(x,y,x)\approx q(x,x,y)
\end{array}\label{bb}
\end{equation}
This system obviously allows both terms to be projection maps in any algebra, so more identities from \eqref{jedan} are needed here. Let us discuss each way of adding an identity:

If we equalize the term $p(x,x,y)$ with the terms of the identities \eqref{bb}, we shall obtain the following:\begin{equation*}
\begin{array}{c}
p(x,y,y) \approx q(x,y,x) \approx p(x,y,x)\approx q(x,x,y) \approx p(x,x,y)
\end{array}
\end{equation*}
Still both projection maps are allowed, and adding any new identities from \eqref{jedan} can only give us \eqref{jedan}.

If we equalize the term $q(y,x,x)$ with the terms of the identities \eqref{bb}, we shall obtain the following:\begin{equation*}
\begin{array}{c}
p(x,y,y) \approx q(x,y,x) \approx p(x,y,x)\approx q(x,x,y) \approx q(y,x,x)
\end{array}
\end{equation*}
This system holds in a full idempotent reduct of a module over $\mathbb{Z}_5$, for we can define  $p$ and $q$ to be $4x+2y$, $2x+2y+2z$ respectively in this reduct. Adding any new identities from \eqref{jedan} can only give us \eqref{jedan}.

If we add the identity $p(x,x,y) \approx q(y,x,x)$ to the identities \eqref{bb} we shall obtain the following:\begin{equation*}
\left\{
\begin{array}{c}
p(x,y,y) \approx q(x,y,x) \approx p(x,y,x)\approx q(x,x,y)\\
p(x,x,y) \approx q(y,x,x)
\end{array}\right.
\end{equation*}
This system holds in a full idempotent reduct of a module over $\mathbb{Z}_5$, for we can define  $p$ and $q$ to be $3x+3y$, $3y+3z$ respectively in this reduct. Adding any new identities from \eqref{jedan} can only give us \eqref{jedan}.
\noindent By this we have finished analyzing the system \eqref{bb}. By adding identities to this system we can obtain either a system that holds in some full idempotent reduct of a module over a finite ring, or the whole system \eqref{jedan}. 
\item if we equalize the term $q(y,x,x)$ with the terms of the  identities \eqref{aq}, we shall obtain the following:\begin{equation}
\begin{array}{c}
p(x,y,y) \approx q(x,y,x) \approx p(x,y,x)\approx q(y,x,x)
\end{array}\label{bc}
\end{equation}
This allows both $p$ and $q$ to be defined as projection maps in any algebra, so more identities from \eqref{jedan} are needed here. Let us discuss each way of adding an identity:  

If we equalize the term $p(x,x,y)$ with the terms of the identities \eqref{bc}, we shall obtain the following:\begin{equation*}
\begin{array}{c}
p(x,y,y) \approx q(x,y,x) \approx p(x,y,x)\approx q(y,x,x)\approx p(x,x,y)
\end{array}
\end{equation*}
Still both projections are allowed, and adding any new identities from \eqref{jedan} can only give us \eqref{jedan}. 

If we equalize the term $q(x,x,y)$ with the terms of the  identities \eqref{bc}, we shall obtain the following:\begin{equation*}
\begin{array}{c}
p(x,y,y) \approx q(x,y,x) \approx p(x,y,x)\approx q(y,x,x)\approx q(x,x,y)
\end{array}
\end{equation*}
This system holds in a full idempotent reduct of a module over $\mathbb{Z}_5$, for we can define  $p$ and $q$ to be $4x+2y$, $2x+2y+2z$ respectively in this reduct. Adding any new identities from \eqref{jedan} can only give us \eqref{jedan}.

If we add the identity $p(x,x,y) \approx q(x,x,y)$ to the identities \eqref{bc} we shall obtain the following:\begin{equation*}
\left\{
\begin{array}{c}
p(x,y,y) \approx q(x,y,x) \approx p(x,y,x)\approx q(y,x,x)\\
p(x,x,y) \approx q(x,x,y)
\end{array} \right.
\end{equation*}
This system holds in a full idempotent reduct of a module over $\mathbb{Z}_5$, for we can define  $p$ and $q$ to be $3x+3y$, $3x+3y$ respectively in this reduct. Adding any new identities from \eqref{jedan} can only give us \eqref{jedan}.
\noindent By this we have finished analyzing the system \eqref{bc}. By adding identities to this system we can obtain either a system that holds in some full idempotent reduct of a module over a finite ring, or the whole system \eqref{jedan}. 
\item if we add the identity $p(x,x,y) \approx q(x,x,y)$ to the identities \eqref{aq} we shall obtain the following:\begin{equation}
\left\{
\begin{array}{c}
p(x,y,y) \approx q(x,y,x) \approx p(x,y,x)\\
p(x,x,y) \approx q(x,x,y)
\end{array} \right. \label{bd}
\end{equation}
This system allows both terms to be defined as projection maps in any algebra, so more identities from \eqref{jedan} are needed. Let us discuss each way of adding an identity:  

If we equalize the terms of the identities \eqref{bd}, we shall obtain the following: \begin{equation*}
\begin{array}{c}
p(x,y,y) \approx q(x,y,x) \approx p(x,y,x)\approx
p(x,x,y) \approx q(x,x,y)
\end{array} 
\end{equation*}
Still both terms can be defined as projections, and adding any new identity from \eqref{jedan} gives us only the whole system \eqref{jedan}.

If we equalize the term $q(y,x,x)$ with the terms of the first identity of \eqref{bd}, we shall obtain the following: \begin{equation*}
\left\{
\begin{array}{c}
p(x,y,y) \approx q(x,y,x) \approx p(x,y,x) \approx q(y,x,x)\\
p(x,x,y) \approx q(x,x,y)
\end{array} \right. 
\end{equation*}
This system holds in a full idempotent reduct of a module over $\mathbb{Z}_5$, for we can define  $p$ and $q$ to be $3x+3y$, $3x+3y$ respectively in this reduct. Adding any new identities from \eqref{jedan} can only give us \eqref{jedan}.

If we equalize the term $q(y,x,x)$ with the terms of the second identity of \eqref{bd}, we shall obtain the following: \begin{equation*}
\left\{
\begin{array}{c}
p(x,y,y) \approx q(x,y,x) \approx p(x,y,x) \\
p(x,x,y) \approx q(x,x,y)\approx q(y,x,x)
\end{array} \right. 
\end{equation*}
Both terms can now be defined as projection maps in any algebra, and adding any new identities from \eqref{jedan} can only give us \eqref{jedan}.
\noindent By this we have finished analyzing the system \eqref{bd}. By adding identities to this system we can obtain either a system that holds in some full idempotent reduct of a module over a finite ring, or the whole system \eqref{jedan}. 
\item if we add the identity $q(x,x,y) \approx q(y,x,x)$ to the identities \eqref{aq} we shall obtain the following:\begin{equation}
\left\{
\begin{array}{c}
p(x,y,y) \approx q(x,y,x) \approx p(x,y,x)\\
q(x,x,y) \approx q(y,x,x)
\end{array} \right. \label{be}
\end{equation}
This system holds in a full idempotent reduct of a module over $\mathbb{Z}_5$, for we can define  $p$ and $q$ to be $\pi_1$, $3x+3z$ respectively in this reduct. Therefore we need to add some more identities from \eqref{jedan}. Let us discuss each way of adding an identity: 

If we equalize the terms of the  identities \eqref{be}, we shall obtain the following: \begin{equation*}
\begin{array}{c}
p(x,y,y) \approx q(x,y,x) \approx p(x,y,x)\approx
q(x,x,y) \approx q(y,x,x)
\end{array} 
\end{equation*}
This system holds in a full idempotent reduct of a module over $\mathbb{Z}_5$, for we can define  $p$ and $q$ to be $4x+2y$, $2x+2y+2z$ respectively in this reduct. Adding any new identities from \eqref{jedan} can only give us \eqref{jedan}.

If we equalize the term $p(x,x,y)$ with the terms of the first identity of \eqref{be}, we shall obtain the following: \begin{equation*}
\left\{
\begin{array}{c}
p(x,y,y) \approx q(x,y,x) \approx p(x,y,x) \approx p(x,x,y)\\
q(x,x,y) \approx q(y,x,x)
\end{array} \right. 
\end{equation*}
This system holds in a full idempotent reduct of a module over $\mathbb{Z}_5$, for we can define  $p$ and $q$ to be $\pi_1$, $3x+3z$ respectively in this reduct. Adding more identities from \eqref{jedan} can only give us \eqref{jedan}.

If we equalize the term $p(x,x,y)$ with the terms of the  second identity of \eqref{be}, we shall obtain the following: \begin{equation*}
\left\{
\begin{array}{c}
p(x,y,y) \approx q(x,y,x) \approx p(x,y,x)\\
q(x,x,y) \approx q(y,x,x)\approx p(x,x,y)
\end{array} \right. 
\end{equation*}
This system allows both projection maps in any algebra ($p$ and $q$ can both be defined as $\pi_2$ ), and adding any identities from \eqref{jedan} can only give us \eqref{jedan}. 
\noindent By this we have finished analyzing the system \eqref{be}. By adding identities to this system we can obtain either a system that holds in some full idempotent reduct of a module over a finite ring, or the whole system \eqref{jedan}.
\item if we add the identity $p(x,x,y) \approx q(y,x,x)$ to the identities \eqref{aq} we shall obtain the following:\begin{equation}
\left\{
\begin{array}{c}
p(x,y,y) \approx q(x,y,x) \approx p(x,y,x)\\
p(x,x,y) \approx q(y,x,x)
\end{array} \right. \label{bf}
\end{equation}
This system allows both $p$ and $q$ to be defined as projection maps in any algebra ($p$, $q$  being $\pi_1$, $\pi_3$ respectively), so we need to add more identities from \eqref{jedan}. Let us discuss each way of adding an identity:

If we equalize the terms of the identities  \eqref{bf}, we shall obtain the following: \begin{equation*}
\begin{array}{c}
p(x,y,y) \approx q(x,y,x) \approx p(x,y,x)\approx
p(x,x,y) \approx q(y,x,x)
\end{array} 
\end{equation*}
Still both terms can be defined as projection maps, and adding identities from \eqref{jedan} can only give us \eqref{jedan}.

If we equalize the term $q(x,x,y)$ with the terms of the first identity of \eqref{bf}, we shall obtain the following:\begin{equation*}
\left\{
\begin{array}{c}
p(x,y,y) \approx q(x,y,x) \approx p(x,y,x) \approx q(x,x,y)\\
p(x,x,y) \approx q(y,x,x)
\end{array} \right.
\end{equation*}
This system holds in a full idempotent reduct of a module over $\mathbb{Z}_5$, for we can define  $p$ and $q$ to be $3x+3y$, $3y+3z$ respectively in this reduct. Adding more identities from \eqref{jedan} can only give us \eqref{jedan}.

If we equalize the term $q(x,x,y)$ with the terms of the second identity of \eqref{bf},  we shall obtain the following:\begin{equation*}
\left\{
\begin{array}{c}
p(x,y,y) \approx q(x,y,x) \approx p(x,y,x) \\
p(x,x,y) \approx q(y,x,x) \approx q(x,x,y)
\end{array} \right.
\end{equation*}
This system allows both $p$ and $q$ to be defined as projection maps in any algebra ($p$, $q$  both being $\pi_2$), and adding more identities from \eqref{jedan} can only give us \eqref{jedan}.
\noindent By this we have finished analyzing the system \eqref{bf}. By adding identities to this system we can obtain either a system that holds in some full idempotent reduct of a module over a finite ring, or the whole system \eqref{jedan}.
\end{itemize} 

\vspace{0.2 cm} 

\noindent We have examined all proper subsets of \eqref{jedan} that include the identities \eqref{aq}. Each of them holds in some full idempotent reduct of a module over a finite ring, therefore does not imply omitting types 1 and 2.

\vspace{0.2 cm}

\noindent Up to this point, we have analyzed all proper subsets of \eqref{jedan} that have the identity $p(x,y,y) \approx q(x,y,x)$ as the first one and the second identity only on $p$. None of them implies omitting types 1 and 2. Let us now discuss proper subsets of \eqref{jedan} that have the identity $p(x,y,y) \approx q(x,y,x)$ as the first one, and the second identity on $p$ and $q$. As explained before (in the subsection \ref{sub1}), it is sufficient to examine second identities on $p$ and $q$ that include neither $p(x,y,y)$ on the left, nor $q(x,y,x)$ on the right.

\item the second identity is $p(x,x,y) \approx q(x,x,y)$, i.e. we have the system:\begin{equation}
\left\{
\begin{array}{c}
p(x,y,y) \approx q(x,y,x) \\
p(x,x,y) \approx q(x,x,y)
\end{array} \right. \label{bg}
\end{equation}
If we permute the identities of this system, we shall obtain the system with the first identity being $p(x,x,y) \approx q(x,x,y)$. This system has already been discussed ( in the subsection \ref{sub2}), and it is proven that adding identities from \eqref{jedan} can only give us a system equivalent to \eqref{dva} or the whole system  \eqref{jedan}.
\item the same holds for the second identity being any of the following: $p(x,x,y) \approx q(y,x,x)$, $p(x,y,x) \approx q(y,x,x)$, $p(x,y,x) \approx q(x,x,y)$ -- we can consider the second identity to be the first one, and all these systems have already been discussed or examined in the previous two subsections (\ref{sub1} and \ref{sub2}).
\end{itemize} 

\vspace{0.3 cm}

\noindent By this we have examined all proper subsets of the system \eqref{jedan}, and we can conclude the following: from the system \eqref{jedan} we can obtain a single system, up to a permutation of variables, that implies omitting types 1 and 2 ( and a proper subset of \eqref{jedan}), and this is the system denoted by \eqref{dva}. 

\section {} \label{sesta}

In this section we shall prove the following (this proof is excluded from the subsection \ref{cases12}):

\vspace{0.5 cm}

\noindent from the system \begin{equation}
\begin{array}{c}
x \approx p(x,y,y) \approx p(x,x,y)\approx p(x,y,x) \approx q(x,y,x)\approx q(x,x,y) \approx q(y,x,x) 
\end{array} \label{bh}
\end{equation}
(also denoted by (2) in the subsection \ref{cases12}), we can obtain no more but two systems (up to a permutation of variables) implying omitting types 1 and 2: the first one is the system \eqref{dva} (also denoted by (4) in the subsection \ref{cases12}), and the second one is the system \begin{equation}
\left\{
\begin{array}{r}
x \approx q(x,y,x)\\
p(x,y,y)\approx p(x,y,x)\\ 
p(x,x,y)\approx q(x,x,y) \approx q(y,x,x)\label{bi}
\end{array}\right.
\end{equation}
(denoted by (5) in the subsection \ref{cases12}). 

\vspace{0.2 cm}

\noindent Let us notice the following facts:
\begin{enumerate}
\item the system \eqref{bh} does not imply omitting types 1 and 2, as explained in the subsection \ref{cases12}, so we need to eliminate some of its identities, i.e. we have to form a proper subset of \eqref{bh}.
\item if we have a subset of \eqref{bh} in which no identity has $x$ alone on either side, then it is in fact a subset of \eqref{jedan}, and all of its subsets have already been examined in the section \ref{peta}.
\end{enumerate}
\noindent Therefore it is sufficient to examine proper subsets of \eqref{bh} that have at least one identity with $x$ alone on one side. We shall consider this identity the first in each one of these systems.
\subsection{}\label{sub4}
Proper subsets of the system \eqref{bh} with the first identity being one of the following three: $x \approx p(x,x,y)$, $x \approx p(x,y,x)$, $x \approx p(x,y,y)$ are already examined in the subsection 4.1 of this paper, and it has been proven that none of them implies omitting types 1 and 2. 
\subsection{} \label{sub5}
In this subsection we shall examine proper subsets of the system \eqref{bh} with the first identity being one of the following three: $x \approx q(x,x,y)$, $x \approx q(x,y,x)$, $x \approx q(y,x,x)$.

\vspace{0.2 cm}

\noindent It is easily seen there is no need to examine all these cases: suppose there is a subset $\sigma_1$ of \eqref{bh} that implies omitting types 1 and 2, and its first identity is $x \approx q(x,y,x)$  or  $x \approx q(y,x,x)$. If we substitute $q(x,y,z)$ for  $q(x,z,y)$ or $q(z,y,x)$ respectively in the whole system $\sigma_1$, we shall obtain a new subset $\sigma_2$ of \eqref{bh} which is equivalent to $\sigma_1$, and whose first identity is $x \approx q(x,x,y)$. This means it is sufficient to examine proper subsets of \eqref{bh} whose first identity is $x \approx q(x,x,y)$. 

\vspace{0.2 cm}

\noindent In the following proof we shall make use of the example given in the section 2 of this paper (denoted by example 1 there) -- it is $\mathbf{B}=\langle \ \{\ \!0 \ ,\ 1 \} \ ,\ \land  \  \rangle$, the semilattice with two elements, and this algebra omits types 1 and 2. 

\vspace{0.2 cm}

\noindent Suppose we have a subset $\theta$ of \eqref{bh}, that describes omitting types 1 and 2, with the first identity being $x \approx q(x,x,y)$. The system $\theta$ has to hold in algebra $\mathbf{B}$, so let us notice some important facts:
\begin{enumerate}
\item if there is the identity  $x \approx p(x,x,y)$ in $\theta$, then we have $x \approx q(x,x,y) \approx p(x,x,y)$. The system $\theta$ holds in $\mathbf{B}$, so both $p$ and $q$ can be defined as at most binary terms in this algebra, but this allows $p$ and $q$ to be defined as at most binary terms in a full idempotent reduct of a module over $\mathbb{Z}_5$. So $\theta$ would have to hold in $\mathbb{Z}_5$, and this is impossible. Therefore  $\theta$ cannot include the identity $x \approx p(x,x,y)$ . The same observation holds for identities $x \approx p(x,y,x)$, $x \approx p(x,y,y)$, i.e.  $\theta$ cannot include either of them.
\item if there is no identity on both $p$ and $q$ in the system $\theta$, i.e. all the identities are either only on $p$ or only on $q$, then $q$ is at most a binary term in $\mathbf{B}$, and $p$ can be defined as a projection map $\pi_1$. This also allows both $p$ and $q$ to be defined in a full idempotent reduct of a module over $\mathbb{Z}_5$, so this case is impossible too. In other words, $\theta$ has to include at least one identity on both $p$ and $q$. Moreover, the term $q(x,x,y)$ cannot occur in the identity mentioned, for the reasons explained in the previous item.
\end{enumerate}

\vspace{0.2 cm}

\noindent We shall vary the second identity now:
\begin{itemize}
\item the second identity is $p(x,x,y) \approx q(x,y,x)$ , i.e. we have the identities:\begin{equation}
\left\{
\begin{array}{c}
x \approx q(x,x,y) \\
p(x,x,y) \approx q(x,y,x)
\end{array} \right. \label{bj}
\end{equation}
This system obviously allows both $p$ and $q$ to be defined as projection maps in any algebra, so more identities from \eqref{bh} are needed here. Adding identities from \eqref{bh} can be done in several ways: we can equalize the terms of the identities \eqref{bj} (but we shall not do this for the reasons explained in the item 1 above), we can equalize the term $q(y,x,x)$ with the terms in  each of the identities, or each of the terms $p(x,y,y)$, $p(x,y,x)$ with the terms of the  second identity only (again for the reasons explained in the item 1), or we can add each of the identities $q(y,x,x) \approx p(x,y,y)$, $q(y,x,x) \approx p(x,y,x)$, $p(x,y,y) \approx p(x,y,x)$ to the system \eqref{bj}. Let us go through all these cases:
\begin{itemize}
\item if we equalize the term $q(y,x,x)$ with the terms of the first identity of the system \eqref{bj}, we shall obtain the following: \begin{equation}
\left\{
\begin{array}{c}
x \approx q(x,x,y) \approx q(y,x,x) \\
p(x,x,y) \approx q(x,y,x)
\end{array} \right. \label{bk}
\end{equation}
This still allows both projection maps in any algebra, so more identities from \eqref{bh} are needed. Let us discuss each way of adding an identity: 

If we equalize the term $p(x,y,x)$ with the terms of the  second identity of the system \eqref{bk}, we shall obtain the following system:\begin{equation*}
\left\{
\begin{array}{c}
x \approx q(x,x,y) \approx q(y,x,x) \\
p(x,x,y) \approx q(x,y,x)\approx p(x,y,x)
\end{array} \right. 
\end{equation*}
These terms do not exist in algebra $\mathbf{B}$ mentioned above, so this system does not describe omitting types 1 and 2. Adding identities from \eqref{bh} does not change a thing.

If we equalize the term $p(x,y,y)$ with the terms of the  second identity of the system \eqref{bk}, we shall obtain the following system:\begin{equation*}
\left\{
\begin{array}{c}
x \approx q(x,x,y) \approx q(y,x,x) \\
p(x,x,y) \approx q(x,y,x)\approx p(x,y,y)
\end{array} \right. 
\end{equation*}
This system allows both projection maps, so we need to add more identities from \eqref{bk}. By equalizing the terms in  these two identities, or by equalizing the term $p(x,y,x)$ with the terms in the first of them, we cannot obtain anything useful (this is explained in the item 1 above), so we can only equalize the term $p(x,y,x)$ with the terms of the second identity. We obtain the following:\begin{equation*}
\left\{
\begin{array}{c}
x \approx q(x,x,y) \approx q(y,x,x) \\
p(x,x,y) \approx q(x,y,x)\approx p(x,y,y) \approx p(x,y,x)
\end{array} \right. 
\end{equation*}
This system does not hold in algebra $\mathbf{B}$, and adding identities from \eqref{bh} cannot change that fact. 

If we add the identity $p(x,y,x) \approx p(x,y,y)$ to the system \eqref{bk} we shall obtain the following:\begin{equation*}
\left\{
\begin{array}{c}
x \approx q(x,x,y) \approx q(y,x,x) \\
p(x,x,y) \approx q(x,y,x)\\
p(x,y,x) \approx p(x,y,y)
\end{array} \right. 
\end{equation*}
This system does not hold in algebra $\mathbf{B}$, and adding identities from \eqref{bh} cannot change that fact.
\noindent By this we have finished analyzing the system \eqref{bk}. By adding identities to this system we can obtain either a system that holds in some full idempotent reduct of a module over a finite ring, or a system that does not hold in algebra $\mathbf{B}$.
\item if we equalize the term $q(y,x,x)$ with the terms of the  second identity of the system \eqref{bj}, we shall obtain the following: \begin{equation}
\left\{
\begin{array}{c}
x \approx q(x,x,y) \\
p(x,x,y) \approx q(x,y,x) \approx q(y,x,x)
\end{array} \right. \label{bl}
\end{equation}
This system holds in a full idempotent reduct of a module over $\mathbb{Z}_5$, for we can define  $p$ and $q$ to be $3x+3z$, $3x+3y$ respectively in this reduct. We need to add more identities from \eqref{bh}. By equalizing the terms of the  two identities given we cannot obtain anything useful, as already explained above, and the same holds for equalizing either of the terms $p(x,y,y)$, $p(x,y,x)$ with the terms of the  first identity. So let us discuss the rest of the cases: 

If we equalize the term $p(x,y,y)$ with the terms of the  second identity of \eqref{bl}, we shall obtain the following:\begin{equation*}
\left\{
\begin{array}{c}
x \approx q(x,x,y) \\
p(x,x,y) \approx q(x,y,x) \approx q(y,x,x)\approx p(x,y,y)
\end{array} \right. 
\end{equation*} 
This system also holds in a full idempotent reduct of a module over $\mathbb{Z}_5$, for we can define  $p$ and $q$ to be $3x+3z$, $3x+3y$ respectively in this reduct. We need more identities from \eqref{bh}, and in this case it can only mean equalizing the term $p(x,y,x)$ with the terms of the  second identity, so we obtain the following system: \begin{equation}
\left\{
\begin{array}{c}
x \approx q(x,x,y) \\
p(x,x,y) \approx q(x,y,x) \approx q(y,x,x)\approx p(x,y,y)\approx p(x,y,x)
\end{array} \right. \label{bm} 
\end{equation}
It is easily proved that the system \eqref{bm} actually implies omitting types 1 and 2, for it cannot hold in any full idempotent reduct of a module over a finite ring:
\begin{proof}
Let $p$ be the term $\alpha x + \beta y + \gamma z$ for $\alpha + \beta + \gamma =1$ (in a reduct). Now, from the second identity we obtain $\alpha + \beta  = \alpha = \alpha+ \gamma$, which gives us  $\beta =  \gamma = 0$, i.e. $p$ has to be the first projection map. Then, from the second identity we see that $q$ has to be  the third projection map in this reduct, which is impossible because of the first identity. Therefore $p$ and $q$ do not exist in any reduct of a module over a finite ring. 
\end{proof}
So we have obtained the system \eqref{bm} that implies omitting types 1 and 2, but it is not minimal in that respect -- namely if we substitute $q(x,y,z)$ for $q(x,z,y)$ in this system, we shall obtain an equivalent system: \begin{equation}
\left\{
\begin{array}{c}
x \approx q(x,y,x) \\
p(x,x,y) \approx q(x,x,y) \approx q(y,x,x)\approx p(x,y,y)\approx p(x,y,x)
\end{array} \right. \label{bn} 
\end{equation}
This system \eqref{bn} also implies omitting types 1 and 2, of course, but it contains the system \eqref{bi}, i.e. it includes all the identities of \eqref{bi} and some more. Therefore \eqref{bn} is not a minimal system having this property.
 
If we equalize the term $p(x,y,x)$ with the terms of the  second identity of \eqref{bl}, we shall obtain the following:\begin{equation*}
\left\{
\begin{array}{c}
x \approx q(x,x,y) \\
p(x,x,y) \approx q(x,y,x) \approx q(y,x,x)\approx p(x,y,x)
\end{array} \right. 
\end{equation*}
This system holds in a full idempotent reduct of a module over $\mathbb{Z}_5$, for we can define  $p$ and $q$ to be $3y+3z$, $3x+3y$ respectively in this reduct. Adding more identities from \eqref{bh} can only mean equalizing the term $p(x,y,y)$ with the terms of the second identity, but that would give us the system \eqref{bm}, which is already examined above. 

If we add the identity $p(x,y,y) \approx p(x,y,x)$ to the system \eqref{bl}, we shall obtain the following:\begin{equation*}
\left\{
\begin{array}{c}
x \approx q(x,x,y) \\
p(x,x,y) \approx q(x,y,x) \approx q(y,x,x)\\
p(x,y,y) \approx p(x,y,x)
\end{array} \right. 
\end{equation*}
This system is equivalent to the system \eqref{bi} (it can be obtained from \eqref{bi} by substituting $q(x,y,z)$ for $q(x,z,y)$).
\noindent By this we have finished analyzing the system \eqref{bl}. By adding identities to this system, we can obtain either a system that holds in some full idempotent reduct of a module over a finite ring, or a system that contains the system \eqref{bi}, up to a permutation of variables, or a system equivalent to the system \eqref{bi}, again up to a permutation of variables.  
\item if we equalize the term $p(x,y,y)$ with the terms of the second identity of the system \eqref{bj}, we shall obtain the following: \begin{equation}
\left\{
\begin{array}{c}
x \approx q(x,x,y) \\
p(x,x,y) \approx q(x,y,x) \approx p(x,y,y)
\end{array} \right. \label{bo}
\end{equation}
This system allows both $p$ and $q$ to be projection maps in any algebra, so we need to add more identities from \eqref{bh}. Let us discuss each way of adding an identity: 

If we equalize the term $p(x,y,x)$ with the terms of the second identity of \eqref{bo}, we shall obtain the following system: \begin{equation}
\left\{
\begin{array}{c}
x \approx q(x,x,y) \\
p(x,x,y) \approx q(x,y,x) \approx p(x,y,y) \approx p(x,y,x)
\end{array} \right.\label{bp}
\end{equation} 
This still allows both projection maps in any algebra, so let us add more identities here:

If we equalize the term $q(y,x,x)$ with the terms of the first identity of the system \eqref{bp}, we shall obtain the following:\begin{equation*}
\left\{
\begin{array}{c}
x \approx q(x,x,y) \approx q(y,x,x) \\
p(x,x,y) \approx q(x,y,x) \approx p(x,y,y) \approx p(x,y,x)
\end{array} \right.
\end{equation*}
This system does not hold in algebra $\mathbf{B}$ (example 1, section 2), so it does not describe omitting types 1 and 2.

If we equalize the term $q(y,x,x)$ with the terms of the second identity of the system \eqref{bp}, we shall obtain the following:\begin{equation*}
\left\{
\begin{array}{c}
x \approx q(x,x,y)  \\
p(x,x,y) \approx q(x,y,x) \approx p(x,y,y) \approx p(x,y,x) \approx q(y,x,x)
\end{array} \right.
\end{equation*}
We have already discussed this system, it is denoted by \eqref{bm} above (it  properly contains the system \eqref{bi}, up to a permutation of variables).
\noindent By this we have finished analyzing the system \eqref{bp}. By adding identities to this system, we can obtain either a system that does not hold in $\mathbf{B}$, or a system that properly contains the system \eqref{bi}, up to a permutation of variables.

If we equalize the term $q(y,x,x)$ with the terms of the first identity of \eqref{bo}, we shall obtain the following system: \begin{equation}
\left\{
\begin{array}{c}
x \approx q(x,x,y) \approx q(y,x,x)\\
p(x,x,y) \approx q(x,y,x) \approx p(x,y,y)
\end{array} \right. \label{br} 
\end{equation}
This system allows both projection maps in any algebra, so let us add more identities from \eqref{bh} (it only make sense to equalize the term $p(x,y,x)$ with the terms of the second identity above).

If we equalize the term $p(x,y,x)$ with the terms of the second identity of \eqref{br}, we shall obtain the following:\begin{equation*}
\left\{
\begin{array}{c}
x \approx q(x,x,y) \approx q(y,x,x)\\
p(x,x,y) \approx q(x,y,x) \approx p(x,y,y) \approx p(x,y,x)
\end{array} \right.  
\end{equation*}
This system does not hold in algebra $\mathbf{B}$ (example 1, section 2), so it does not describe omitting types 1 and 2.
By this we have finished analyzing the system \eqref{br}. By adding identities to this system we can obtain only a system that does not hold in $\mathbf{B}$.

If we equalize the term $q(y,x,x)$ with the terms of the second identity of \eqref{bo}, we shall obtain the following system: \begin{equation*}
\left\{
\begin{array}{c}
x \approx q(x,x,y) \\
p(x,x,y) \approx q(x,y,x) \approx p(x,y,y) \approx q(y,x,x)
\end{array} \right.  
\end{equation*}
This system holds in a full idempotent reduct of a module over $\mathbb{Z}_5$, for we can define  $p$ and $q$ to be $3x+3z$, $3x+3y$ respectively in this reduct. We need to add more identities from \eqref{bh}, i.e. to equalize the term $p(x,y,x)$ with the terms of the second identity of the system \eqref{bs}, but it would give us the system \eqref{bm} which is already discussed above. 

If we add the identity $p(x,y,x) \approx q(y,x,x)$ to the system \eqref{bo} we shall obtain the following system: \begin{equation}
\left\{
\begin{array}{c}
x \approx q(x,x,y) \\
p(x,x,y) \approx q(x,y,x) \approx p(x,y,y)\\
p(x,y,x) \approx q(y,x,x)
\end{array} \right. \label{bs}
\end{equation}
This system allows both $p$ and $q$ to be defined as projection maps in any algebra ($p$, $q$ being $\pi_3$, $\pi_2$ respectively), so more identities from \eqref{bh} are needed here. Adding identities can only be done by equalizing  terms of any two of the three identities given above, so let us go through the cases:

If we equalize the terms of the first and the second identity of the system \eqref{bs}, we shall obtain the following:\begin{equation*}
\left\{
\begin{array}{c}
x \approx q(x,x,y) \approx
p(x,x,y) \approx q(x,y,x) \approx p(x,y,y)\\
p(x,y,x) \approx q(y,x,x)
\end{array} \right. 
\end{equation*}
It is easily seen that this system does not hold in algebra $\mathbf{B}$ (example 1, section 2).

If we equalize the terms of the second and the third identity of the system \eqref{bs},  we shall obtain the following: \begin{equation*}
\left\{
\begin{array}{c}
x \approx q(x,x,y) \\
p(x,x,y) \approx q(x,y,x) \approx p(x,y,y) \approx
p(x,y,x) \approx q(y,x,x)
\end{array} \right. 
\end{equation*}
We have already discussed this system, it is denoted by \eqref{bm} above (it  properly contains the system \eqref{bi}, up to a permutation of variables). 

If we equalize the terms of the  first and the third identity of the system \eqref{bs},  we shall obtain the following:\begin{equation*}
\left\{
\begin{array}{c}
x \approx q(x,x,y) \approx p(x,y,x) \approx q(y,x,x)\\
p(x,x,y) \approx q(x,y,x) \approx p(x,y,y)
\end{array} \right. 
\end{equation*}
This allows both projection maps in any algebra, and adding identities from \eqref{bh} can only give us the whole system \eqref{bh}. 
\noindent By this we have finished analyzing the system \eqref{bo}. By adding identities to this system, we can obtain one of the following: a system that holds in some full idempotent reduct of a module over a finite ring, a system that does not hold in $\mathbf{B}$, a system that properly contains the system \eqref{bi}, up to a permutation of variables, or the whole system \eqref{bh}.
\item if we equalize the term $p(x,y,x)$ with the terms of the second identity of the system \eqref{bj}, we shall obtain the following:\begin{equation}
\left\{
\begin{array}{c}
x \approx q(x,x,y) \\
p(x,x,y) \approx q(x,y,x)\approx p(x,y,x)
\end{array} \right. \label{bt}
\end{equation}
This system allows both $p$ and $q$ to be defined as projection maps, so we need more identities from \eqref{bh}. Let us discuss each way of adding an identity: 

If we equalize the term $p(x,y,y)$ with the terms of the second identity of \eqref{bt}, we shall obtain the following:\begin{equation*}
\left\{
\begin{array}{c}
x \approx q(x,x,y) \\
p(x,x,y) \approx q(x,y,x)\approx p(x,y,x) \approx p(x,y,y)
\end{array} \right.
\end{equation*}
This system has already been examined above (denoted by \eqref{bp}). 

If we equalize the term $q(y,x,x)$ with the terms of the  first identity of \eqref{bt}, we shall obtain the following:\begin{equation*}
\left\{
\begin{array}{c}
x \approx q(x,x,y) \approx q(y,x,x) \\
p(x,x,y) \approx q(x,y,x)\approx p(x,y,x) 
\end{array} \right.
\end{equation*}
It is easily seen that this system does not hold in algebra $\mathbf{B}$ (example 1, section 2).

If we equalize the term $q(y,x,x)$ with the terms of the  second identity of \eqref{bt}, we shall obtain the following:\begin{equation*}
\left\{
\begin{array}{c}
x \approx q(x,x,y)  \\
p(x,x,y) \approx q(x,y,x)\approx p(x,y,x) \approx q(y,x,x)
\end{array} \right.
\end{equation*}
This system holds in a full idempotent reduct of a module over $\mathbb{Z}_5$, for we can define  $p$ and $q$ to be $3y+3z$, $3x+3y$ respectively in this reduct. We need to add more identities from \eqref{bh}, but this can only be done by equalizing the term $p(x,y,y)$ with the terms of teh second identity. We obtain the following system:\begin{equation*}
\left\{
\begin{array}{c}
x \approx q(x,x,y)  \\
p(x,x,y) \approx q(x,y,x)\approx p(x,y,x) \approx q(y,x,x)\approx p(x,y,y)
\end{array} \right.
\end{equation*}
We have already discussed this system, it is denoted by \eqref{bm} above (it  properly contains the system \eqref{bi}, up to a permutation of variables).

If we add the identity $p(x,y,y)\approx q(y,x,x)$ to the system \eqref{bt} we shall obtain the following:\begin{equation}
\left\{
\begin{array}{c}
x \approx q(x,x,y) \\
p(x,x,y) \approx q(x,y,x)\approx p(x,y,x)\\
p(x,y,y)\approx q(y,x,x)
\end{array} \right. \label{bu} 
\end{equation}
This system holds in a full idempotent reduct of a module over $\mathbb{Z}_5$, for we can define  $p$ and $q$ to be $2x+2y+2z$, $4x+2y$ respectively in this reduct. We need more identities from \eqref{bh}, but this can only be done by equalizing any terms of  two of the three identities given above, so let us go through the cases: 

If we equalize the terms of the first and the second identity  of the system \eqref{bu}, we shall obtain the following:\begin{equation*}
\left\{
\begin{array}{c}
x \approx q(x,x,y) \approx
p(x,x,y) \approx q(x,y,x)\approx p(x,y,x)\\
p(x,y,y)\approx q(y,x,x)
\end{array} \right. 
\end{equation*}
It is easily seen that this system does not hold in algebra $\mathbf{B}$ (example 1 in the section 2).

If we equalize the terms of the  second and the third identity of the system \eqref{bu}, we shall obtain the following:\begin{equation*}
\left\{
\begin{array}{c}
x \approx q(x,x,y) \\
p(x,x,y) \approx q(x,y,x)\approx p(x,y,x)\approx
p(x,y,y)\approx q(y,x,x)
\end{array} \right. 
\end{equation*}
We have already discussed this system, it is denoted by \eqref{bm} above (it  properly contains the system \eqref{bi}, up to a permutation of variables).

If we equalize the terms of the first and the third identity of the system \eqref{bu},  we shall obtain the following:\begin{equation*}
\left\{
\begin{array}{c}
x \approx q(x,x,y) \approx p(x,y,y)\approx q(y,x,x)\\
p(x,x,y) \approx q(x,y,x)\approx p(x,y,x)
\end{array} \right.
\end{equation*}
It is easily seen that this system does not hold in algebra $\mathbf{B}$ (example 1, section 2).
\noindent By this we have finished analyzing the system \eqref{bt}. By adding identities to this system, we can obtain one of the following: a system that holds in some full idempotent reduct of a module over a finite ring, a system that does not hold in $\mathbf{B}$ or a system that properly contains the system \eqref{bi}, up to a permutation of variables. 
\item if we add the identity $q(y,x,x) \approx p(x,y,y)$ to the system \eqref{bj} we shall obtain the following:\begin{equation}
\left\{
\begin{array}{c}
x \approx q(x,x,y) \\
p(x,x,y) \approx q(x,y,x)\\
q(y,x,x) \approx p(x,y,y)
\end{array} \right. \label{bv}
\end{equation}
This system allows both $p$ and $q$ to be projection maps in any algebra, so we need to add more identities from \eqref{bh}. Let us discuss each way of adding an identity:  

If we equalize the terms of the  second and the third identity of the system \eqref{bv}, we shall obtain the following system: \begin{equation*}
\left\{
\begin{array}{c}
x \approx q(x,x,y) \\
p(x,x,y) \approx q(x,y,x)\approx
q(y,x,x) \approx p(x,y,y)
\end{array} \right. 
\end{equation*}
This system holds in a full idempotent reduct of a module over $\mathbb{Z}_5$, for we can define  $p$ and $q$ to be $3x+3z$, $3x+3y$ respectively in this reduct. We need more identities from \eqref{bh} here, and this can only be done by equalizing the term $p(x,y,x)$ with the terms of the second identity, but then we obtain the system \eqref{bm} from above, which is already discussed.

If we equalize the term  $p(x,y,x)$  with the terms of the  second identity of the system \eqref{bv}, we shall obtain the following: \begin{equation*}
\left\{
\begin{array}{c}
x \approx q(x,x,y) \\
p(x,x,y) \approx q(x,y,x) \approx p(x,y,x)\\
q(y,x,x) \approx p(x,y,y)
\end{array} \right. 
\end{equation*}
This system has already been discussed above (denoted by \eqref{bu}).

If we equalize the term  $p(x,y,x)$  with the terms of the  third identity of the system \eqref{bv}, we shall obtain the following:\begin{equation*}
\left\{
\begin{array}{c}
x \approx q(x,x,y) \\
p(x,x,y) \approx q(x,y,x)\\
q(y,x,x) \approx p(x,y,y)\approx p(x,y,x)
\end{array} \right. 
\end{equation*}
This system allows both $p$ and $q$ to be defined as projection maps in any algebra, so we need to add more identities from \eqref{bh}, but this can only be done by equalizing the terms of the  second and the third identity, and that would give us the system \eqref{bm} which is already discussed.
\noindent By this we have finished analyzing the system \eqref{bv}. By adding identities to this system, we can obtain one of the following: a system that holds in some full idempotent reduct of a module over a finite ring, a system that does not hold in $\mathbf{B}$, or a system that properly contains the system \eqref{bi}, up to a permutation of variables. 
\item if we add the identity $q(y,x,x) \approx p(x,y,x)$ to the system \eqref{bj} we shall obtain the following:\begin{equation}
\left\{
\begin{array}{c}
x \approx q(x,x,y) \\
p(x,x,y) \approx q(x,y,x)\\
q(y,x,x) \approx p(x,y,x)
\end{array} \right. \label{bw}
\end{equation}
This system allows both $p$ and $q$ to be defined as projection maps in any algebra, so we need to add more identities from \eqref{bh}. Let us discuss each way of adding an identity:  

If we equalize the terms of the second and the third identity of the system \eqref{bw}, we shall obtain the following:\begin{equation*}
\left\{
\begin{array}{c}
x \approx q(x,x,y) \\
p(x,x,y) \approx q(x,y,x)\approx
q(y,x,x) \approx p(x,y,x)
\end{array} \right. 
\end{equation*}
This system holds in a full idempotent reduct of a module over $\mathbb{Z}_5$, for we can define  $p$ and $q$ to be $3y+3z$, $3x+3y$ respectively in this reduct. We need to add more identities from \eqref{bh} here, and this can only be done by equalizing the term $p(x,y,y)$ with the terms of the  second identity of the system above, but that would give us the system \eqref{bm}.

If we equalize the term $p(x,y,y)$ with the terms of the second identity of the system \eqref{bw}, we shall obtain the following:\begin{equation*}
\left\{
\begin{array}{c}
x \approx q(x,x,y) \\
p(x,x,y) \approx q(x,y,x) \approx p(x,y,y)\\
q(y,x,x) \approx p(x,y,x)
\end{array} \right. 
\end{equation*}
This system allows both projection maps, so we need to add more identities from \eqref{bh}, but we can only equalize the terms of the second and the third identity here, which gives us the system \eqref{bm}.

If we equalize the term $p(x,y,y)$ with the terms of the third identity of the system \eqref{bw} we shall obtain the following:\begin{equation*}
\left\{
\begin{array}{c}
x \approx q(x,x,y) \\
p(x,x,y) \approx q(x,y,x)\\
q(y,x,x) \approx p(x,y,x)\approx p(x,y,y)
\end{array} \right. 
\end{equation*}
This allows both projection maps, and by adding more identities from \eqref{bh} we can only obtain the system \eqref{bm}.
\noindent By this we have finished analyzing the system \eqref{bw}. By adding identities to this system, we can obtain one of the following: a system that holds in some full idempotent reduct of a module over a finite ring, or a system that properly contains the system \eqref{bi}, up to a permutation of variables. 
\item if we add the identity $p(x,y,y) \approx p(x,y,x)$ to the system \eqref{bj} we shall obtain the following:\begin{equation}
\left\{
\begin{array}{c}
x \approx q(x,x,y) \\
p(x,x,y) \approx q(x,y,x)\\
p(x,y,y) \approx p(x,y,x)
\end{array} \right. \label{bq}
\end{equation}
This system allows both $p$ and $q$ to be defined as projection maps, so we need to add more identities from \eqref{bh}. Let us discuss each way of adding an identity: 

If we equalize the terms of the second and the third identity of the system \eqref{bq}, we shall obtain the following:\begin{equation}
\left\{
\begin{array}{c}
x \approx q(x,x,y) \\
p(x,x,y) \approx q(x,y,x)\approx
p(x,y,y) \approx p(x,y,x)
\end{array} \right. \label{bz}
\end{equation}
Still both projection maps are allowed, so we need more identities from \eqref{bh}. Let us discuss each way of adding an identity: 

If we equalize the term $q(y,x,x)$ with the terms of the first identity of the system \eqref{bz}, we shall obtain the following:\begin{equation*}
\left\{
\begin{array}{c}
x \approx q(x,x,y) \approx q(y,x,x) \\
p(x,x,y) \approx q(x,y,x)\approx
p(x,y,y) \approx p(x,y,x)
\end{array} \right.
\end{equation*}
It is easily seen that this system does not hold in algebra $\mathbf{B}$ (example 1, section 2).

If we equalize the term $q(y,x,x)$ with terms of the second identity of the system \eqref{bz} we shall obtain the following:\begin{equation*}
\left\{
\begin{array}{c}
x \approx q(x,x,y) \\
p(x,x,y) \approx q(x,y,x)\approx
p(x,y,y) \approx p(x,y,x)\approx q(y,x,x)
\end{array} \right. 
\end{equation*}
This is the system \eqref{bm} which has already been discussed. 

If we equalize the term $q(y,x,x)$ with the terms of the first identity of the system \eqref{bq} we shall obtain the following system:\begin{equation*}
\left\{
\begin{array}{c}
x \approx q(x,x,y)\approx q(y,x,x) \\
p(x,x,y) \approx q(x,y,x)\\
p(x,y,y) \approx p(x,y,x)
\end{array} \right. 
\end{equation*}
It is easily seen that this system does not hold in algebra $\mathbf{B}$ (example 1, section 2).

If we equalize the term $q(y,x,x)$ with the terms of the second identity of the system \eqref{bq} we shall obtain the following system:\begin{equation*}
\left\{
\begin{array}{c}
x \approx q(x,x,y) \\
p(x,x,y) \approx q(x,y,x) \approx q(y,x,x)\\
p(x,y,y) \approx p(x,y,x)
\end{array} \right. 
\end{equation*}
This is the system \eqref{bi} up to a permutation of variables. Adding identities from \eqref{bh} would not make any sense, for if we did, we would obtain also a system that implies omitting types 1 and 2 but not a minimal one in that respect.

If we equalize the term $q(y,x,x)$ with the terms of the  third identity of the system \eqref{bq} we shall obtain the following system:\begin{equation*}
\left\{
\begin{array}{c}
x \approx q(x,x,y) \\
p(x,x,y) \approx q(x,y,x)\\
p(x,y,y) \approx p(x,y,x)\approx q(y,x,x)
\end{array} \right. 
\end{equation*}
This allows both $p$ and $q$ to be defined as projection maps in any algebra, so we need to add more identities from \eqref{bh}, but this can only be done by equalizing the terms of the  second and the third identity of this system , which would give us the system \eqref{bm}.
\noindent By this we have finished analyzing the system \eqref{bq}. By adding identities to this system, we can obtain one of the following: a system that holds in some full idempotent reduct of a module over a finite ring, a system that does not hold in $\mathbf{B}$, a system that properly contains the system \eqref{bi}, up to a permutation of variables, or a system equivalent to the system \eqref{bi}, up to a permutation of variables.
\end{itemize}
Up to this point we have examined all proper subsets of the system \eqref{bh} which include the identities \eqref{bj}.
\item the second identity is $p(x,x,y) \approx q(y,x,x)$, i.e. we have the identities:\begin{equation}
\left\{
\begin{array}{c}
x \approx q(x,x,y) \\
p(x,x,y) \approx q(y,x,x)
\end{array} \right.\label{bx}
\end{equation}
There is no need to examine subsets of the system  \eqref{bh} that include the identities \eqref{bx}, for by substituting $q(x,y,z)$ for  $q(y,x,z)$ in any of them we obtain a subset of \eqref{bh} including the identities \eqref{bj} and all these subsets have already been examined. 
\item the second identity is $p(x,y,x) \approx q(x,y,x)$, i.e. we have the identities:\begin{equation}
\left\{
\begin{array}{c}
x \approx q(x,x,y) \\
p(x,y,x) \approx q(x,y,x)
\end{array} \right.\label{by}
\end{equation}
There is no need to examine subsets of the system  \eqref{bh} that include the identities \eqref{by}, for by substituting $p(x,y,z)$ for  $p(x,z,y)$ in any of them we obtain a subset of \eqref{bh} including the identities \eqref{bj} and all these subsets have already been examined.
\item the second identity is $p(x,y,x) \approx q(y,x,x)$, i.e. we have the identities:
\begin{equation}
\left\{
\begin{array}{c}
x \approx q(x,x,y) \\
p(x,y,x) \approx q(y,x,x)
\end{array} \right.\label{ca}
\end{equation} 
There is no need to examine subsets of the system  \eqref{bh} that include the identities \eqref{ca}, for by substituting $p(x,y,z)$ for  $p(x,z,y)$ and $q(x,y,z)$ for  $q(y,x,z)$ in any of them we obtain a subset of \eqref{bh} including the identities \eqref{bj} and all these subsets have already been examined.
\item the second identity is $p(x,y,y) \approx q(x,y,x)$, i.e. we have the identities:\begin{equation}
\left\{
\begin{array}{c}
x \approx q(x,x,y) \\
p(x,y,y) \approx q(x,y,x)
\end{array} \right.\label{cb}
\end{equation}
This system obviously allows both $p$ and $q$ to be defined as projection maps, so we need to add more identities from \eqref{bh}. This can be done in the following ways: we can equalize the term $q(y,x,x)$ with the terms of each of the identities, or each of the terms $p(x,y,x)$, $p(x,x,y)$ with the terms of the second identity, or we can add one of the identities $q(y,x,x) \approx p(x,y,x)$, $q(y,x,x) \approx p(x,x,y)$, $p(x,y,x) \approx p(x,x,y)$ to the existing two. As before we shall analyze each case:
\begin{itemize}
\item if we equalize the term $q(y,x,x)$ with the terms of the  first identity of the system \eqref{cb} we shall obtain the following: \begin{equation}
\left\{
\begin{array}{c}
x \approx q(x,x,y)\approx q(y,x,x) \\
p(x,y,y) \approx q(x,y,x)
\end{array} \right.\label{cc}
\end{equation}
This system allows both projection maps, so more identities from \eqref{bh} are needed. Let us discuss each way of adding an identity:  

If we equalize the term $p(x,y,x)$ with the terms of the second identity of the system \eqref{cc}, we shall obtain the following:\begin{equation*}
\left\{
\begin{array}{c}
x \approx q(x,x,y)\approx q(y,x,x) \\
p(x,y,y) \approx q(x,y,x) \approx p(x,y,x)
\end{array} \right.
\end{equation*}
This system still allows both terms to be defined as projection maps, so we need to add more identities from \eqref{bh}, but it only makes sense to equalize the term $p(x,x,y)$ with the terms  of the second identity, i.e. we have the following system:\begin{equation*}
\left\{
\begin{array}{c}
x \approx q(x,x,y)\approx q(y,x,x) \\
p(x,y,y) \approx q(x,y,x) \approx p(x,y,x)\approx p(x,x,y)
\end{array} \right.
\end{equation*}
It is easily seen that this system does not hold in algebra $\mathbf{B}$ (example 1, section 2).

If we equalize the term $p(x,x,y)$ with the terms of the second identity of the system \eqref{cc} we shall obtain the following:\begin{equation*}
\left\{
\begin{array}{c}
x \approx q(x,x,y)\approx q(y,x,x) \\
p(x,y,y) \approx q(x,y,x) \approx p(x,x,y)
\end{array} \right.
\end{equation*}
This system allows both projection maps in any algebra, so we need more identities from \eqref{bh}, but it only makes sense to equalize the term $p(x,y,x)$ with the terms of the second identity (for the reasons explained at the beginning of the current subsection). So, we shall obtain the system: \begin{equation*}
\left\{
\begin{array}{c}
x \approx q(x,x,y)\approx q(y,x,x) \\
p(x,y,y) \approx q(x,y,x) \approx p(x,x,y) \approx p(x,y,x)
\end{array} \right.
\end{equation*}
This is the same system as the one above, and it does not hold in algebra $\mathbf{B}$.

If we add the identity $p(x,y,x) \approx p(x,x,y)$ to the system \eqref{cc} we shall obtain the following: \begin{equation*}
\left\{
\begin{array}{c}
x \approx q(x,x,y)\approx q(y,x,x) \\
p(x,y,y) \approx q(x,y,x)\\
p(x,y,x) \approx p(x,x,y)
\end{array} \right. 
\end{equation*}
This system holds in a full idempotent reduct of a module over $\mathbb{Z}_5$, for we can define  $p$ and $q$ to be $3y+3z$, $\pi_2$ respectively in this reduct. We need to add more identities from the system \eqref{bh} here, but as we have explained at the beginning of this subsection, it might only be useful to equalize the terms of the second and the third identity. Therefore, we can obtain the system: \begin{equation*}
\left\{
\begin{array}{c}
x \approx q(x,x,y)\approx q(y,x,x) \\
p(x,y,y) \approx q(x,y,x) \approx
p(x,y,x) \approx p(x,x,y)
\end{array} \right. 
\end{equation*} 
This does not hold in algebra $\mathbf{B}$.
\noindent By this we have finished analyzing the system \eqref{cc}. By adding identities to this system, we can obtain one of the following: a system that holds in some full idempotent reduct of a module over a finite ring or a system that does not hold in $\mathbf{B}$.
\item if we equalize the term $q(y,x,x)$ with the terms of the  second identity of the system \eqref{cb}, we shall obtain the following: \begin{equation}
\left\{
\begin{array}{c}
x \approx q(x,x,y) \\
p(x,y,y) \approx q(x,y,x) \approx q(y,x,x)
\end{array} \right.\label{cd}
\end{equation}
This system holds in a full idempotent reduct of a module over $\mathbb{Z}_5$, for we can define  both $p$ and $q$ to be $3x+3y$ in this reduct. We need to add more identities from \eqref{bh}, so let us analyze each case:

If we equalize the term $p(x,y,x)$ with the terms of the  second identity of the system \eqref{cd}, we shall obtain the following: \begin{equation*}
\left\{
\begin{array}{c}
x \approx q(x,x,y) \\
p(x,y,y) \approx q(x,y,x) \approx q(y,x,x) \approx p(x,y,x)
\end{array} \right.
\end{equation*}
This system holds in a full idempotent reduct of a module over $\mathbb{Z}_5$, for we can define  both $p$ and $q$ to be $3x+3y$ in this reduct. Adding more identities from \eqref{bh} can only mean equalizing the term $p(x,x,y)$ with the terms of the second identity (as already explained), i.e. we would obtain the following system: \begin{equation*}
\left\{
\begin{array}{c}
x \approx q(x,x,y) \\
p(x,y,y) \approx q(x,y,x) \approx q(y,x,x) \approx p(x,y,x) \approx p(x,x,y)
\end{array} \right.
\end{equation*} 
We have already discussed this system, it is denoted by \eqref{bm} above (it  properly contains the system \eqref{bi}, up to a permutation of variables).

If we equalize the term $p(x,x,y)$ with the terms of the second identity of the system \eqref{cd},  we shall obtain the following:\begin{equation*}
\left\{
\begin{array}{c}
x \approx q(x,x,y) \\
p(x,y,y) \approx q(x,y,x) \approx q(y,x,x) \approx p(x,x,y)
\end{array} \right.
\end{equation*}
This system holds in a full idempotent reduct of a module over $\mathbb{Z}_5$, for we can define   $p$ and $q$ to be $3x+3z$, $3x+3y$ respectively in this reduct. Adding more identities from \eqref{bh} can only mean equalizing the term $p(x,y,x)$ with the terms of the second identity (as already explained), i.e. we would obtain the following system: \begin{equation*}
\left\{
\begin{array}{c}
x \approx q(x,x,y) \\
p(x,y,y) \approx q(x,y,x) \approx q(y,x,x) \approx p(x,x,y) \approx p(x,y,x)
\end{array} \right.
\end{equation*}
We have already discussed this system, it is denoted by \eqref{bm} above (it  properly contains the system \eqref{bi}, up to a permutation of variables).

If we add the identity  $p(x,y,x) \approx p(x,x,y)$ to the system \eqref{cd}, we shall obtain the following: \begin{equation*}
\left\{
\begin{array}{c}
x \approx q(x,x,y) \\
p(x,y,y) \approx q(x,y,x) \approx q(y,x,x)\\
p(x,y,x) \approx p(x,x,y)
\end{array} \right.
\end{equation*}
This system holds in a full idempotent reduct of a module over $\mathbb{Z}_3$, for we can define   $p$ and $q$ to be $2x+y+z$, $2x+2y$ respectively in this reduct. We need to add more identities from \eqref{bh} here, but it only make sense to equalize the terms of the second and the third identity, which would give us the system \eqref{bm} again.
\noindent By this we have finished analyzing the system \eqref{cd}. By adding identities to this system, we can obtain one of the following: a system that holds in some full idempotent reduct of a module over a finite ring or a system that properly contains the system \eqref{bi}.
\item if we equalize the term $p(x,y,x)$ with the terms of the  second identity of the system \eqref{cb},  we shall obtain the following: \begin{equation}
\left\{
\begin{array}{c}
x \approx q(x,x,y) \\
p(x,y,y) \approx q(x,y,x)\approx p(x,y,x)
\end{array} \right.\label{ce}
\end{equation}
This system allows both $p$ and $q$ to be projection maps, so we need to add more identities form \eqref{bh}. Let us analyze each case:

If we equalize the term $q(y,x,x)$ with the terms of the first identity of \eqref{ce}, we shall obtain the following: \begin{equation*}
\left\{
\begin{array}{c}
x \approx q(x,x,y) \approx  q(y,x,x) \\
p(x,y,y) \approx q(x,y,x)\approx p(x,y,x)
\end{array} \right.
\end{equation*}
This system allows both projection maps, so more identities from \eqref{bh} are needed. We can only equalize the term $p(x,x,y)$ with the terms of the second identity, and we shall obtain the following system:\begin{equation*}
\left\{
\begin{array}{c}
x \approx q(x,x,y) \approx  q(y,x,x) \\
p(x,y,y) \approx q(x,y,x)\approx p(x,y,x) \approx p(x,x,y)
\end{array} \right.
\end{equation*}
This system does not hold in algebra $\mathbf{B}$.

If we equalize the term $q(y,x,x)$ with the terms of the  second identity of \eqref{ce}, we shall obtain the following: \begin{equation*}
\left\{
\begin{array}{c}
x \approx q(x,x,y) \\
p(x,y,y) \approx q(x,y,x)\approx p(x,y,x) \approx q(y,x,x)
\end{array} \right.
\end{equation*}
This system has already been discussed above, it gives us nothing new.

If we equalize the term $p(x,x,y)$ with the terms of the second identity of \eqref{ce}, we shall obtain the following: \begin{equation*}
\left\{
\begin{array}{c}
x \approx q(x,x,y) \\
p(x,y,y) \approx q(x,y,x)\approx p(x,y,x) \approx p(x,x,y)
\end{array} \right. \label{cf}
\end{equation*}
This system allows both projection maps, so we need to add more identities from \eqref{bh}. This can be done by equalizing the term $q(y,x,x)$ with terms of either of the identities, so let us discuss both cases:

If we equalize the term $q(y,x,x)$ with the terms of the  first identity  of \eqref{cf},  we shall obtain the following system:\begin{equation*}
\left\{
\begin{array}{c}
x \approx q(x,x,y) \approx q(y,x,x) \\
p(x,y,y) \approx q(x,y,x)\approx p(x,y,x) \approx p(x,x,y)
\end{array} \right.
\end{equation*}
This system does not hold in algebra $\mathbf{B}$.

If we equalize the term $q(y,x,x)$ with the terms of the second identity  of \eqref{cf},  we shall obtain the following system:\begin{equation*}
\left\{
\begin{array}{c}
x \approx q(x,x,y) \\
p(x,y,y) \approx q(x,y,x)\approx p(x,y,x) \approx p(x,x,y)  \approx q(y,x,x)
\end{array} \right.
\end{equation*}
We have already discussed this system, it is denoted by \eqref{bm} above (it  properly contains the system \eqref{bi}, up to a permutation of variables).

If we add the identity $q(y,x,x) \approx p(x,x,y)$ to the system \eqref{ce}, we shall obtain the following: \begin{equation*}
\left\{
\begin{array}{c}
x \approx q(x,x,y) \\
p(x,y,y) \approx q(x,y,x)\approx p(x,y,x)\\
q(y,x,x) \approx p(x,x,y)
\end{array} \right.
\end{equation*}
This system allows both projection maps, and adding identities from \eqref{bh} can only be done by equalizing the terms of the  second and the third identity, which would give us the system \eqref{bm}.
\noindent By this we have finished analyzing the system \eqref{ce}. By adding identities to this system, we can obtain one of the following: a system that holds in some full idempotent reduct of a module over a finite ring, a system that does not hold in algebra $\mathbf{B}$, or a system that properly contains the system \eqref{bi}.
\item if we equalize the term $p(x,x,y)$ with the terms of the  second identity of the system \eqref{cb},  we shall obtain the following:\begin{equation}
\left\{
\begin{array}{c}
x \approx q(x,x,y) \\
p(x,y,y) \approx q(x,y,x) \approx p(x,x,y)
\end{array} \right.\label{cg}
\end{equation}
This system allows both $p$ and $q$ to be projection maps, so we need to add more identities from \eqref{bh}. Let us analyze each case:

If we equalize the term $q(y,x,x)$ with the terms of the  first identity of the system \eqref{cg},  we shall obtain the following:\begin{equation*}
\left\{
\begin{array}{c}
x \approx q(x,x,y) \approx q(y,x,x) \\
p(x,y,y) \approx q(x,y,x) \approx p(x,x,y)
\end{array} \right.
\end{equation*} 
This system still allows both projection maps, so more identities from \eqref{bh} are needed. This can by equalizing the term $p(x,y,x)$ with the terms of the second identity, so we obtain the following: \begin{equation*}
\left\{
\begin{array}{c}
x \approx q(x,x,y) \approx q(y,x,x) \\
p(x,y,y) \approx q(x,y,x) \approx p(x,x,y) \approx p(x,y,x)
\end{array} \right.
\end{equation*}
This system does not hold in algebra $\mathbf{B}$.

If we equalize the term $q(y,x,x)$ with the terms of the second identity of the system \eqref{cg} , we shall obtain the following:\begin{equation*}
\left\{
\begin{array}{c}
x \approx q(x,x,y)  \\
p(x,y,y) \approx q(x,y,x) \approx p(x,x,y) \approx q(y,x,x)
\end{array} \right.
\end{equation*}
This system holds in a full idempotent reduct of a module over $\mathbb{Z}_3$, for we can define   $p$ and $q$ to be $3x+3z$, $3x+3y$ respectively in this reduct. By adding identities from the system \eqref{bh} we can only obtain the system \eqref{bm}.

If we equalize the term $p(x,y,x)$ with the terms of the second identity of the system \eqref{cg},  we shall obtain the following: \begin{equation*}
\left\{
\begin{array}{c}
x \approx q(x,x,y) \\
p(x,y,y) \approx q(x,y,x) \approx p(x,x,y) \approx p(x,y,x)
\end{array} \right.
\end{equation*}
This system has already been discussed (is denoted by \eqref{cf} above), and it can give us nothing new.

If we add the identity $p(x,y,x) \approx q(y,x,x)$ to the system \eqref{cg} we shall obtain the following:\begin{equation*}
\left\{
\begin{array}{c}
x \approx q(x,x,y) \\
p(x,y,y) \approx q(x,y,x) \approx p(x,x,y)\\
p(x,y,x) \approx q(y,x,x)
\end{array} \right.
\end{equation*}
This system allows both $p$ and $q$ to be defined as projection maps in any algebra, and adding identities from \eqref{bh} can only give us the system \eqref{bm}.
\noindent By this we have finished analyzing the system \eqref{cg}. By adding identities to this system, we can obtain one of the following: a system that holds in some full idempotent reduct of a module over a finite ring, a system that does not hold in algebra $\mathbf{B}$, or a system that properly contains the system \eqref{bi}.
\item if we add the identity $q(y,x,x) \approx p(x,y,x)$ to the system \eqref{cb} we shall obtain the following:\begin{equation}
\left\{
\begin{array}{c}
x \approx q(x,x,y) \\
p(x,y,y) \approx q(x,y,x)\\
q(y,x,x) \approx p(x,y,x)
\end{array} \right.\label{ch}
\end{equation}
This system holds in a full idempotent reduct of a module over $\mathbb{Z}_5$, for we can define   both $p$ and $q$ to be $3x+3y$ in this reduct. We need to add more identities from \eqref{bh} here, so let us go through the cases: 

If we equalize the terms of the second and the third identity of the system \eqref{ch},  we shall obtain the following:\begin{equation*}
\left\{
\begin{array}{c}
x \approx q(x,x,y) \\
p(x,y,y) \approx q(x,y,x)\approx
q(y,x,x) \approx p(x,y,x)
\end{array} \right.
\end{equation*}
This system holds in a full idempotent reduct of a module over $\mathbb{Z}_5$, for we can define   both $p$ and $q$ to be $3x+3y$ in this reduct. Adding identities from \eqref{bh} can only give us the system \eqref{bm}.

If we equalize the term $p(x,x,y)$ with the terms of the  second identity of the system \eqref{ch}, we shall obtain the following: \begin{equation*}
\left\{
\begin{array}{c}
x \approx q(x,x,y) \\
p(x,y,y) \approx q(x,y,x) \approx p(x,x,y)\\
q(y,x,x) \approx p(x,y,x)
\end{array} \right.
\end{equation*}
Both $p$ and $q$ can be defined as projection maps in any algebra, and adding identities from \eqref{bh} can only be done by equalizing the terms of the second and the third identity, which gives us the system \eqref{bm}.

If we equalize the term $p(x,x,y)$ with the terms of the  third identity of the system \eqref{ch}, we shall obtain the following: \begin{equation*}
\left\{
\begin{array}{c}
x \approx q(x,x,y) \\
p(x,y,y) \approx q(x,y,x)\\
q(y,x,x) \approx p(x,y,x) \approx p(x,x,y)
\end{array} \right.
\end{equation*}
This system holds in a full idempotent reduct of a module over $\mathbb{Z}_5$, for we can define   $p$ and $q$ to be $2x+2y+2z$, $2x+4y$ respectively in this reduct. Adding more identities from \eqref{bh} can only be done by equalizing the terms of the  second and the third identity, which would give us the system \eqref{bm}.
\noindent By this we have finished analyzing the system \eqref{ch}. By adding identities to this system, we can obtain one of the following: a system that holds in some full idempotent reduct of a module over a finite ring, or a system that properly contains the system \eqref{bi}.
\item if we add the identity $q(y,x,x) \approx p(x,x,y)$ to the system \eqref{cb} we shall obtain the following:\begin{equation}
\left\{
\begin{array}{c}
x \approx q(x,x,y) \\
p(x,y,y) \approx q(x,y,x)\\
q(y,x,x) \approx p(x,x,y)
\end{array} \right.\label{ci}
\end{equation}
This system holds in a full idempotent reduct of a module over $\mathbb{Z}_5$, for we can define   $p$ and $q$ to be $3x+3z$, $3x+3y$ respectively in this reduct. We need to add more identities from \eqref{bh}, so let us go through the cases:

If we equalize the terms of the  second and the third identity of the system \eqref{ci},  we shall obtain the following:\begin{equation*}
\left\{
\begin{array}{c}
x \approx q(x,x,y) \\
p(x,y,y) \approx q(x,y,x)\approx
q(y,x,x) \approx p(x,x,y)
\end{array} \right.
\end{equation*}
This system holds in a full idempotent reduct of a module over $\mathbb{Z}_5$, for we can define   $p$ and $q$ to be $3x+3z$, $3x+3y$ respectively in this reduct. We need to add more identities from \eqref{bh}, but this can only be done by equalizing the term $p(x,y,x)$ with the  terms of the  second identity, which would give us the system \eqref{bm}, already discussed above.

If we equalize the term $p(x,y,x)$ with the terms of the  second identity of \eqref{ci}, we shall obtain the following:\begin{equation*}
\left\{
\begin{array}{c}
x \approx q(x,x,y) \\
p(x,y,y) \approx q(x,y,x) \approx p(x,y,x)\\
q(y,x,x) \approx p(x,x,y)
\end{array} \right.
\end{equation*}
This system allows both $p$ and $q$ to be defined as projection maps, and adding identities from \eqref{bh} can only give us the system \eqref{bm}.

If we equalize the term $p(x,y,x)$ with the terms of the third identity of \eqref{ci}, we shall obtain the following:\begin{equation*}
\left\{
\begin{array}{c}
x \approx q(x,x,y) \\
p(x,y,y) \approx q(x,y,x) \\
q(y,x,x) \approx p(x,x,y) \approx p(x,y,x)
\end{array} \right.
\end{equation*}
This system holds in a full idempotent reduct of a module over $\mathbb{Z}_5$, for we can define   $p$ and $q$ to be $2x+2y+2z$, $2x+4y$ respectively in this reduct. Adding more identities from \eqref{bh} can only give us the system \eqref{bm}.
\noindent By this we have finished analyzing the system \eqref{ci}. By adding identities to this system, we can obtain one of the following: a system that holds in some full idempotent reduct of a module over a finite ring, or a system that properly contains the system \eqref{bi}.
\item if we add the identity $p(x,y,x) \approx p(x,x,y)$ to the system \eqref{cb} we shall obtain the following:\begin{equation}
\left\{
\begin{array}{c}
x \approx q(x,x,y) \\
p(x,y,y) \approx q(x,y,x)\\
p(x,y,x) \approx p(x,x,y)
\end{array} \right.\label{cj}
\end{equation}
This system allows both $p$ and $q$ to be defined as projection maps in any algebra, so we need to add more identities from the system \eqref{bh}. Let us analyze each case: 

If we equalize the  terms of the second and the third identity of the system \eqref{cj},  we shall obtain the following system:\begin{equation}
\left\{
\begin{array}{c}
x \approx q(x,x,y) \\
p(x,y,y) \approx q(x,y,x) \approx
p(x,y,x) \approx p(x,x,y)
\end{array}\right.\label{ck} 
\end{equation}
Still both terms can be defined as projection maps, so we need to add more identities from \eqref{bh}. Let us analyze each case: 

If we equalize the term $q(y,x,x)$ with the terms of the  first identity of the system \eqref{ck},  we shall obtain the following:\begin{equation*}
\left\{
\begin{array}{c}
x \approx q(x,x,y) \approx q(y,x,x) \\
p(x,y,y) \approx q(x,y,x) \approx
p(x,y,x) \approx p(x,x,y)
\end{array}\right.
\end{equation*}
This system does not hold in algebra $\mathbf{B}$.

If we equalize the term $q(y,x,x)$ with the terms of the  second identity of the system \eqref{ck},  we shall obtain the following: \begin{equation*}
\left\{
\begin{array}{c}
x \approx q(x,x,y)  \\
p(x,y,y) \approx q(x,y,x) \approx
p(x,y,x) \approx p(x,x,y) \approx q(y,x,x)
\end{array}\right.
\end{equation*}
This is the system \eqref{bm} which has already been discussed above. 

If we equalize the term $q(y,x,x)$ with the terms of the  first identity of the system \eqref{cj},  we shall obtain the following:\begin{equation*}
\left\{
\begin{array}{c}
x \approx q(x,x,y) \approx q(y,x,x) \\
p(x,y,y) \approx q(x,y,x)\\
p(x,y,x) \approx p(x,x,y)
\end{array} \right.
\end{equation*}
This system holds in a full idempotent reduct of a module over $\mathbb{Z}_5$, for we can define   $p$ and $q$ to be $3y+3z$, $\pi_2$ respectively in this reduct. Adding more identities from \eqref{bh} means equalizing the terms of the  third and the second identity, so we obtain the following: \begin{equation*}
\left\{
\begin{array}{c}
x \approx q(x,x,y) \approx q(y,x,x) \\
p(x,y,y) \approx q(x,y,x) \approx
p(x,y,x) \approx p(x,x,y)
\end{array} \right.
\end{equation*}
This system does not hold in algebra $\mathbf{B}$.

If we equalize the term $q(y,x,x)$ with the terms of the  second identity of the system \eqref{cj},  we shall obtain the following:\begin{equation*}
\left\{
\begin{array}{c}
x \approx q(x,x,y)  \\
p(x,y,y) \approx q(x,y,x) \approx q(y,x,x)\\
p(x,y,x) \approx p(x,x,y)
\end{array} \right.
\end{equation*}
This system holds in a full idempotent reduct of a module over $\mathbb{Z}_3$, for we can define   $p$ and $q$ to be $2x+y+z$, $2x+2y$ respectively in this reduct. Adding more identities from the system \eqref{bh} can only mean equalizing the terms of the second and the third identity, but that gives us the system \eqref{bm}, which is already discussed. 

If we equalize the term $q(y,x,x)$ with the terms of the third identity of the system \eqref{cj},  we shall obtain the following:\begin{equation*}
\left\{
\begin{array}{c}
x \approx q(x,x,y)  \\
p(x,y,y) \approx q(x,y,x) \\
p(x,y,x) \approx p(x,x,y) \approx q(y,x,x)
\end{array} \right.
\end{equation*}
This system holds in a full idempotent reduct of a module over $\mathbb{Z}_5$, for we can define   $p$ and $q$ to be $2x+2y+2z$, $2x+4y$ respectively in this reduct. Adding more identities from the system \eqref{bh} can only give us the system \eqref{bm}. 
\noindent By this we have finished analyzing the system \eqref{cj}. By adding identities to this system, we can obtain one of the following: a system that holds in some full idempotent reduct of a module over a finite ring, or a system that does not hold in algebra $\mathbf{B}$, or a system that properly contains the system \eqref{bi}.
\end{itemize}

\vspace{0.2 cm} 

\noindent By this we have examined all proper subsets of the system \eqref{bh} with the first and the second identity being \eqref{cb}. Only the system \eqref{bi} (and some systems properly containing this one, like the system \eqref{bm}) implies omitting types 1 and 2. 
\item the second identity is $p(x,y,y) \approx q(y,x,x)$, i.e. we have the system: 
\begin{equation}
\left\{
\begin{array}{c}
x \approx q(x,x,y) \\
p(x,y,y) \approx q(y,x,x)
\end{array} \right.\label{cl}
\end{equation}
There is no need to examine subsets of the system \eqref{bh} that include identities \eqref{cl}, for if we substitute $q(x,y,z)$ for $q(y,x,z)$ in such a subset, we shall obtain a subset of \eqref{bh} that includes the identities \eqref{cb}, and all these subsets are already examined in the previous item. 
\end{itemize}

\vspace{0.3 cm}

\noindent By this we have examined all the subsets of the system \eqref{bh} with the first identity being $x \approx q(x,x,y)$. Only the system \eqref{bi} (and some systems properly containing this one, like the system \eqref{bm}) implies omitting types 1 and 2. 

\vspace{0.3 cm} 

\noindent As for the subsets of \eqref{bh} with the first identity being $x \approx q(x,y,x)$ (or $x \approx q(y,x,x)$) -- there is no need to examine any of them, for by substituting $q(x,y,z)$ for $q(x,z,y)$ (or $q(z,y,x)$) we obtain subsets of \eqref{bh} with the first identity being $x \approx q(x,x,y)$, which are already examined.

\vspace{0.3 cm}

\noindent This means we have proved the following: from the system \eqref{bh} we can obtain only two systems of identities (up to a permutation of variables, of course) that imply and possibly describe omitting types 1 and 2 -- these are the systems \eqref{dva} and \eqref{bi} (we can also obtain some systems properly containing these two, and implying omitting types 1 and 2, but we are interested only in minimal systems having this property). 

\section {} \label{sedma}  

In this section we shall prove the following (this proof is excluded from the subsection \ref{case3}):

\vspace{0.5 cm}

\noindent from the system \begin{equation}
\begin{array}{r}
x\approx p(x,x,y)\approx p(x,y,x)\approx p(y,x,x)\approx q(y,x,x)\approx q(x,y,x)\approx q(x,x,y)
\end{array}\label{cm} 
\end{equation} 
(also denoted by (6) in the subsection \ref{case3}), we can obtain only a single system implying and possibly describing omitting types 1 and 2, and that is the following system: \begin{equation}
\left\{
\begin{array}{r}
x\approx p(x,x,y)\\
p(x,y,x)\approx p(y,x,x)\approx q(y,x,x)\approx q(x,y,x)\approx q(x,x,y)
\label{cn}
\end{array}\right.
\end{equation}
(denoted by (7) in the subsection \ref{case3}).

\noindent First we shall prove the minimality of the system \eqref{cn}, so let us notice the following facts:
\begin{enumerate}
\item if we omit the identity $x\approx p(x,x,y)$ from the system \eqref{cn}, both terms can be defined as $2x+2y+2z$ in a full idempotent reduct of a module over $\mathbb{Z}_5$. Therefore this identity stays.
\item if there is no identity on both $p$ and $q$ in a subset of \eqref{cn}, again both terms exist in a full idempotent reduct of a module over $\mathbb{Z}_5$, for we can define $p$ , $q$ to be $3x+3y$, $2x+2y+2z$ respectively in this reduct. Therefore any subset of \eqref{cn} possibly describing omitting types 1 and 2 has to have at least one identity on $p$ and $q$.
\item if term $q$ occurs only once in a subset of \eqref{cn}, and that is in the identity on $p$ and $q$, then both terms exist in a full idempotent reduct of a module over $\mathbb{Z}_5$, for we can define $p$ to be $3x+3y$ and $q$ to be one of the binary terms $3x+3y$, $3x+3z$, $3y+3z$. This means $q$ has to occur in at least two identities -- one on $p$ and $q$ and the other only on $q$ or also on $p$ and $q$.  
\end{enumerate}

\noindent So, we shall analyze subsets of the system \eqref{cn} with the first identity being $x\approx p(x,x,y)$. As before we shall vary the second identity:
\begin{itemize}
\item the second identity is $p(x,y,x) \approx q(y,x,x)$, i.e. we have the following identities:\begin{equation}
\left\{
\begin{array}{r}
x\approx p(x,x,y)\\
p(x,y,x) \approx q(y,x,x)
\label{co}
\end{array}\right.
\end{equation}
Obviously both terms can be defined as projection  maps in any algebra, so we need more identities from the system \eqref{cn}. As said before, the term $q$ has to occur in at least two identities, so we have several possibilities here: we can equalize each of the terms $q(x,y,x)$, $q(x,x,y)$ with the terms of the second identity, or we can add one of the identities $q(x,y,x) \approx q(x,x,y)$, $p(y,x,x) \approx q(x,y,x)$, $p(y,x,x) \approx q(x,x,y)$  to the system \eqref{co}. Let us discuss each of these cases:
\begin{itemize}
\item if we equalize the term $q(x,y,x)$ with the terms of the  second identity of the system \eqref{co}, we shall obtain the following system:\begin{equation}
\left\{
\begin{array}{r}
x\approx p(x,x,y)\\
p(x,y,x) \approx q(y,x,x) \approx q(x,y,x)
\label{cp}
\end{array}\right.
\end{equation}
This system holds in a full idempotent reduct of a module over $\mathbb{Z}_5$, for we can define both $p$ and $q$ to be $3x+3y$ in this reduct. We need to add more identities from the system \eqref{cn}, so let us go through the cases:

If we equalize the term $p(y,x,x)$ with the terms of the second identity of the system \eqref{cp},  we shall obtain the following:\begin{equation*}
\left\{
\begin{array}{r}
x\approx p(x,x,y)\\
p(x,y,x) \approx q(y,x,x) \approx q(x,y,x) \approx p(y,x,x)
\end{array}\right.
\end{equation*}
Still both terms can be defined as $3x+3y$ in a full idempotent reduct of a module over $\mathbb{Z}_5$, and by adding identities from \eqref{cn} we can only obtain the whole system \eqref{cn}.

If we equalize the term $q(x,x,y)$ with the terms of the  second identity of the system \eqref{cp},  we shall obtain the following: \begin{equation*}
\left\{
\begin{array}{r}
x\approx p(x,x,y)\\
p(x,y,x) \approx q(y,x,x) \approx q(x,y,x) \approx q(x,x,y)
\end{array}\right.
\end{equation*}
This system holds in a full idempotent reduct of a module over $\mathbb{Z}_5$, for we can define   $p$ and $q$ to be $4x+2y$, $2x+2y+2z$ respectively in this reduct. Adding identities from \eqref{cn} can only give us the whole system \eqref{cn}.

If we add the identity $p(y,x,x) \approx q(x,x,y)$ to the system \eqref{cp}, we  shall obtain the following:\begin{equation*}
\left\{
\begin{array}{r}
x\approx p(x,x,y)\\
p(x,y,x) \approx q(y,x,x) \approx q(x,y,x)\\
p(y,x,x) \approx q(x,x,y)
\end{array}\right.
\end{equation*}
This system allows both $p$ and $q$ to be defined as projection maps in any algebra, so we need more identities from the system \eqref{cn}, but that can only be done by equalizing the terms of the  second and the third identity which gives us the whole system \eqref{cn}.
\noindent By this we have finished analyzing the system \eqref{cp}. By adding identities to this system, we can obtain one of the following: a system that holds in some full idempotent reduct of a module over a finite ring, or the system \eqref{cn}.
\item if we equalize the term $q(x,x,y)$ with the terms of the second identity of the system \eqref{co}, we shall obtain the following system:
\begin{equation}
\left\{
\begin{array}{r}
x\approx p(x,x,y)\\
p(x,y,x) \approx q(y,x,x) \approx q(x,x,y)
\label{cr}
\end{array}\right.
\end{equation}
This system holds in a full idempotent reduct of a module over $\mathbb{Z}_5$, for we can define   $p$ and $q$ to be $3x+3y$, $3x+3z$ respectively in this reduct. We need to add more identities from the system \eqref{cn}, so let us go through the cases:

If we equalize the term $p(y,x,x)$ with the terms of the second identity of the system  \eqref{cr},  we shall obtain the following:\begin{equation*}
\left\{
\begin{array}{r}
x\approx p(x,x,y)\\
p(x,y,x) \approx q(y,x,x) \approx q(x,x,y)\approx p(y,x,x)
\end{array}\right.
\end{equation*} 
This system  still holds in a full idempotent reduct of a module over $\mathbb{Z}_5$, for we can define  $p$ and $q$ to be $3x+3y$, $3x+3z$ respectively in this reduct. Adding more identities from the system \eqref{cn} can only give us the whole system \eqref{cn}.

If we equalize the term $q(x,y,x)$ with the terms of the second identity of the system \eqref{cr},  we shall obtain the following system:\begin{equation*}
\left\{
\begin{array}{r}
x\approx p(x,x,y)\\
p(x,y,x) \approx q(y,x,x) \approx q(x,x,y)\approx q(x,y,x)
\end{array}\right.
\end{equation*} 
This system holds in a full idempotent reduct of a module over $\mathbb{Z}_5$, for we can define   $p$ and $q$ to be $4x+2y$, $2x+2y+2z$ respectively in this reduct. Adding more identities from the system \eqref{cn} can only give us the whole system \eqref{cn}.

If we add the identity $p(y,x,x) \approx q(x,y,x)$ to the system \eqref{cr} we shall obtain the following system:\begin{equation*}
\left\{
\begin{array}{r}
x\approx p(x,x,y)\\
p(x,y,x) \approx q(y,x,x) \approx q(x,x,y)\\
p(y,x,x) \approx q(x,y,x)
\end{array}\right.
\end{equation*}
This system allows both terms to be defined as projection maps, and adding identities from the system \eqref{cn} can only give us the whole system \eqref{cn}.
\noindent By this we have finished analyzing the system \eqref{cr}. By adding identities to this system, we can obtain one of the following: a system that holds in some full idempotent reduct of a module over a finite ring, or the whole system \eqref{cn}.
\item if we add the identity $q(x,y,x) \approx q(x,x,y)$ to the system \eqref{co}, we shall obtain the following system:\begin{equation}
\left\{
\begin{array}{r}
x\approx p(x,x,y)\\
p(x,y,x) \approx q(y,x,x)\\
q(x,y,x) \approx q(x,x,y)
\label{cq}
\end{array}\right.
\end{equation}
This system holds in a full idempotent reduct of a module over $\mathbb{Z}_5$, for we can define   $p$ and $q$ to be $\pi_1$, $3y+3z$ respectively in this reduct. We need to add more identities from the system \eqref{cn}, so let us analyze all the cases:

If we equalize the terms of the  second and the third identity of the system \eqref{cq}, we shall obtain the following system:\begin{equation*}
\left\{
\begin{array}{r}
x\approx p(x,x,y)\\
p(x,y,x) \approx q(y,x,x) \approx
q(x,y,x) \approx q(x,x,y)
\end{array}\right.
\end{equation*}
This system holds in a full idempotent reduct of a module over $\mathbb{Z}_5$, for we can define   $p$ and $q$ to be $4x+2y$, $2x+2y+2z$ respectively in this reduct. Adding identities from the system \eqref{cn} can only give us the whole system \eqref{cn}.

If we equalize the term $p(y,x,x)$ with the terms of the second identity of the system \eqref{cq},  we shall obtain the following:\begin{equation*}
\left\{
\begin{array}{r}
x\approx p(x,x,y)\\
p(x,y,x) \approx q(y,x,x) \approx p(y,x,x)\\
q(x,y,x) \approx q(x,x,y)
\end{array}\right.
\end{equation*}
This system holds in a full idempotent reduct of a module over $\mathbb{Z}_3$, for we can define   $p$ and $q$ to be $2x+2y$, $2x+y+z$ respectively in this reduct. Adding identities from the system \eqref{cn} can only give us the whole system \eqref{cn}.

If we equalize the term $p(y,x,x)$ with the terms of the  third identity of the system \eqref{cq},  we shall obtain the following:\begin{equation*}
\left\{
\begin{array}{r}
x\approx p(x,x,y)\\
p(x,y,x) \approx q(y,x,x)\\
q(x,y,x) \approx q(x,x,y) \approx p(y,x,x)
\end{array}\right.
\end{equation*}
This system allows both $p$ and $q$ to be defined as projection maps in any algebra, and adding identities from \eqref{cn} can only give us \eqref{cn}.
\noindent By this we have finished analyzing the system \eqref{cq}. By adding identities to this system, we can obtain one of the following: a system that holds in some full idempotent reduct of a module over a finite ring, or the whole system \eqref{cn}.
\item if we add the identity $p(y,x,x) \approx q(x,y,x)$ to the system \eqref{co}, we shall obtain the following system:\begin{equation}
\left\{
\begin{array}{r}
x\approx p(x,x,y)\\
p(x,y,x) \approx q(y,x,x)\\
p(y,x,x) \approx q(x,y,x)
\label{cs}
\end{array}\right.
\end{equation}
This system allows both $p$ and $q$ to be defined as projection maps in any algebra, so we need to add more identities from the system \eqref{cn}. Let us analyze all the cases:

If we equalize the terms of the second and the third identity of the system \eqref{cs},  we shall obtain the following: \begin{equation*}
\left\{
\begin{array}{r}
x\approx p(x,x,y)\\
p(x,y,x) \approx q(y,x,x)\approx 
p(y,x,x) \approx q(x,y,x)
\end{array}\right.
\end{equation*}
This system holds in a full idempotent reduct of a module over $\mathbb{Z}_5$, for we can define   both $p$ and $q$ to be $3x+3y$ in this reduct. Adding identities from \eqref{cn} can only give us \eqref{cn}.

If we equalize the term $q(x,x,y)$ with the terms of the  second identity of the system \eqref{cs},  we shall obtain the following:\begin{equation*}
\left\{
\begin{array}{r}
x\approx p(x,x,y)\\
p(x,y,x) \approx q(y,x,x) \approx q(x,x,y)\\
p(y,x,x) \approx q(x,y,x)
\end{array}\right.
\end{equation*}
This system allows both terms to be projection maps, and adding identities from \eqref{cn} can only give us \eqref{cn}.

If we equalize the term $q(x,x,y)$ with the terms of the third identity of the system \eqref{cs},  we shall obtain the following:\begin{equation*}
\left\{
\begin{array}{r}
x\approx p(x,x,y)\\
p(x,y,x) \approx q(y,x,x) \\
p(y,x,x) \approx q(x,y,x) \approx q(x,x,y)
\end{array}\right.
\end{equation*}
This system allows both terms to be projection maps, and adding identities from \eqref{cn} can only give us \eqref{cn}.
\noindent By this we have finished analyzing the system \eqref{cs}. By adding identities to this system, we can obtain one of the following: a system that holds in some full idempotent reduct of a module over a finite ring, or the whole system \eqref{cn}.
\item if we add the identity $p(y,x,x) \approx q(x,x,y)$ to the system \eqref{co}, we shall obtain the following system: \begin{equation}
\left\{
\begin{array}{r}
x\approx p(x,x,y)\\
p(x,y,x) \approx q(y,x,x)\\
p(y,x,x) \approx q(x,x,y)
\label{ct}
\end{array}\right.
\end{equation}
This system allows both terms to be defined as projection maps in any algebra, so more identities from \eqref{cn} are needed. Let us analyze each case:

If we equalize the terms of the  second and the third identity of the system \eqref{ct}, we shall obtain the following: \begin{equation*}
\left\{
\begin{array}{r}
x\approx p(x,x,y)\\
p(x,y,x) \approx q(y,x,x)\approx
p(y,x,x) \approx q(x,x,y)
\end{array}\right.
\end{equation*}
This system holds in a full idempotent reduct of a module over $\mathbb{Z}_5$, for we can define   $p$ and $q$ to be $3x+3y$, $3x+3z$ respectively in this reduct, and adding identities from \eqref{cn} can only give us \eqref{cn}.

If we equalize the term $q(x,y,x)$ with the terms of the second identity of the system \eqref{ct}, we shall obtain the following:\begin{equation*}
\left\{
\begin{array}{r}
x\approx p(x,x,y)\\
p(x,y,x) \approx q(y,x,x) \approx q(x,y,x)\\
p(y,x,x) \approx q(x,x,y)
\end{array}\right.
\end{equation*}
This system allows both terms to be defined as projection maps, and adding identities from \eqref{cn} can only give us \eqref{cn}.

If we equalize the term $q(x,y,x)$ with the terms of the  third identity of the system \eqref{ct}, we shall obtain the following:\begin{equation*}
\left\{
\begin{array}{r}
x\approx p(x,x,y)\\
p(x,y,x) \approx q(y,x,x) \\
p(y,x,x) \approx q(x,x,y) \approx q(x,y,x)
\end{array}\right.
\end{equation*}
This system allows both terms to be defined as projection maps, and adding identities from \eqref{cn} can only give us \eqref{cn}.
\noindent By this we have finished analyzing the system \eqref{ct}. By adding identities to this system, we can obtain one of the following: a system that holds in some full idempotent reduct of a module over a finite ring, or the whole system \eqref{cn}.
\end{itemize}
\item the second identity is one of the following two identities: $p(x,y,x) \approx q(x,y,x)$, $p(x,y,x) \approx q(x,x,y)$ (the first one being  $x \approx p(x,x,y)$, of course) -- there is no need to analyze these subsets of the system \eqref{cn}, for by substituting $q(x,y,z)$ for $q(y,z,x)$ or $q(z,y,x)$ in such a system (i.e. a subset of \eqref{cn}), we obtain an equivalent system, which is also a subset of \eqref{cn} with the second identity being $p(x,y,x) \approx q(y,x,x)$. We have already examined all these systems in the previous item.
\item the second identity is one of the following three identities: $p(y,x,x) \approx q(x,y,x)$, $p(y,x,x) \approx q(y,x,x)$, $p(y,x,x) \approx q(x,x,y)$ -- here holds the same as in the previous item, namely by substituting $p(x,y,z)$ for $p(y,x,z)$ in these systems, we obtain an equivalent system (also a subset of \eqref{cn}) with the second identity having $p(x,y,x)$ on the left. We have already examined or discussed all these systems.
\end{itemize}

\vspace{0.2 cm}

\noindent By this point we have analyzed all the subsets of the system \eqref{cn}, so we can conclude that there is no proper subset of this system that possibly describe omitting types 1 and 2, i.e. the system \eqref{cn} is a minimal one in that respect.

\vspace{0.3 cm}

\noindent Let us now prove that from the system \eqref{cm} we can obtain only a single system, i.e. the system  \eqref{cn}, which implies and possibly describes omitting types 1 and 2 (of course, when we say \emph{a single system} it is actually up to the permutation of variables). 

\vspace{0.2 cm}

\noindent As before, we shall first notice some important facts:
\begin{enumerate}
\item the whole system \eqref{cm} does not hold in algebra $\mathbf{B}$ (example 1, section 2), so we are looking for a proper subset of \eqref{cm}.
\item if $x$ alone does not occur at all in a subset of \eqref{cm}, i.e. all the identities have either $p$ or $q$ on each side, then such a system (i.e. a subset being considered) holds in a full idempotent reduct of a module over $\mathbb{Z}_5$, for we can define both $p$ and $q$ to be $2x+2y+2z$ in this reduct. Therefore we need to have $x$ alone at one side of at least one identity. 
\end{enumerate} 

\vspace{0.2 cm}

\noindent Suppose we have a subset of \eqref{cm}, let us denote it by $\sigma$, with the first identity being $x \approx p(x,x,y)$. We keep in mind that $\sigma$ needs to hold in algebra $\mathbf{B}$. 
\begin{itemize}
\item if the system $\sigma$ does not include an identity on both $p$ and $q$, i.e. all the identities are either on $p$ or on $q$, then there are two possibilities:
\begin{itemize}
\item the system $\sigma$ cannot hold in algebra $\mathbf{B}$ (e.g. it includes the identity $x \approx p(x,x,y) \approx p(x,y,x) \approx p(y,x,x)$); we can exclude this case, for we need a system that holds in $\mathbf{B}$.
\item the system $\sigma$ holds in $\mathbf{B}$, which means $p$ can be defined as either a projection map or a binary term in this algebra. Then $\sigma$ also holds in a full idempotent reduct of a module over $\mathbb{Z}_5$, for we can define $p$ to be either a projection map or one of the terms $3x+3y$, $3x+3z$, $3y+3z$ and $q$ to be the term $2x+2y+2z$ in this reduct.  
\end{itemize}
Therefore we need to have an identity on both $p$ and $q$ in the system $\sigma$.
\item if the system $\sigma$ includes one of the identities $x \approx p(x,x,y) \approx p(x,y,x)$, $x \approx p(x,x,y) \approx p(y,x,x)$ (i.e. from $\sigma$ we can obtain an equivalent system including one of the previous two identities),then in algebra $\mathbf{B}$, $p$ would have to be defined as a projection map. According to the previous item, $\sigma$ includes an identity on both $p$ and $q$, which means $q$ would have to be either a projection map or a binary term in  $\mathbf{B}$. If the system $\sigma$ allows $p$ and $q$ to be defined this way in  $\mathbf{B}$, then $\sigma$ holds in a full idempotent reduct of a module over $\mathbb{Z}_5$, for we can define both terms to be projection maps, or $p$ to be a projection map and $q$ to be one of the terms $3x+3y$, $3x+3z$, $3y+3z$  in this reduct. \\

\vspace{0.1 cm}

\noindent We can conclude now that $\sigma$ cannot include any of the identities $x \approx p(x,x,y) \approx p(x,y,x)$, $x \approx p(x,x,y) \approx p(y,x,x)$.
\item if the system $\sigma$ includes one of the identities $x \approx p(x,x,y) \approx q(x,y,x)$, $x \approx p(x,x,y) \approx q(y,x,x)$, $x \approx p(x,x,y) \approx q(x,x,y)$ (i.e. from $\sigma$ we can obtain an equivalent system including one of the previous two identities), then both $p$ and $q$ would have to be defined as at most binary terms in algebra $\mathbf{B}$. If $\sigma$ allows that, then $p$ and $q$ exist in a full idempotent reduct of a module over $\mathbb{Z}_5$, for we can define them to be some of the terms $\pi_1$, $\pi_2$, $\pi_3$, $3x+3y$, $3x+3z$, $3y+3z$. \\

\vspace{0.1 cm}

\noindent Therefore $\sigma$ must not include any of the identities $x \approx p(x,x,y) \approx q(x,y,x)$, $x \approx p(x,x,y) \approx q(y,x,x)$, $x \approx p(x,x,y) \approx q(x,x,y)$. This conclusion and the conclusion of the previous item imply that neither $x$ alone nor $p(x,x,y)$ may occur in an identity other than the first one in $\sigma$. In other words, in the rest of the system $\sigma$ (i.e. in the identities besides the first one, $x \approx p(x,x,y)$) we may have only the following terms: $p(y,x,x)$, $p(x,y,x)$, $q(y,x,x)$, $q(x,y,x)$, $q(x,x,y)$. This means the system $\sigma$ is actually a subset of the system \eqref{cn}, and we have already proved that no proper subset of this system can describe omitting types 1 and 2. We can now conclude that there is only a single subset of \eqref{cm} with the first identity being $x \approx p(x,x,y)$, that implies and possibly describes omitting types 1 and 2, and this is the system \eqref{cn}.   
\end{itemize}

\vspace{0.3 cm}

\noindent Suppose there is a subset of the system \eqref{cm} with the first identity being $x \approx p(x,y,x)$, which implies and possibly describes omitting types 1 and 2. Let us denote this subset by $\sigma_1$ (so, $\sigma_1$ does not hold in any full idempotent reduct of a module over a finite ring). If we substitute $p(x,y,z)$ for $p(x,z,y)$ in the system $\sigma_1$ we shall obtain an equivalent system which is also a subset of the system \eqref{cm}, but with the first identity being $x \approx p(x,x,y)$. As we have explained, this has to be the system \eqref{cn}. So, $\sigma_1$ is equivalent to the system \eqref{cn}, and can be obtained from \eqref{cn} by a permutation of variables.

\vspace{0.3 cm}

\noindent The similar holds if we consider a subset of \eqref{cm} with the first identity being $x \approx p(y,x,x)$ -- if it implies omitting types 1 and 2, it is equivalent to the system  \eqref{cn}, and can be obtained from \eqref{cn} by a permutation of variables.

\vspace{0.3 cm}

\noindent As for the subsets of the system \eqref{cm} with the first identity being one of the identities $x \approx q(x,x,y)$, $x \approx q(x,y,x)$, $x \approx q(y,x,x)$, we can substitute term $q(x,y,z)$ for the term $p(x,y,z)$ and other way round (i.e. the term $p(x,y,z)$ for the term $q(x,y,z)$) in such a system, and by this obtain a subset of the system \eqref{cm} with the first identity being one of the identities $x \approx p(x,x,y)$, $x \approx p(x,y,x)$, $x \approx p(y,x,x)$. All these subsets are already examined -- we have obtained nothing but the system \eqref{cn} and some systems equivalent to it. This means there is no need to examine subsets of the system \eqref{cm} with the first identity being one of the identities $x \approx q(x,x,y)$, $x \approx q(x,y,x)$, $x \approx q(y,x,x)$, for we can obtain nothing new.

\vspace{0.3 cm}

\noindent By this we have proved that from the system \eqref{cm} we can obtain only a single system, and that is the system \eqref{cn}, that implies and possibly describes omitting types 1 and 2.

\end{document}